\documentclass[oneside,english,a4paper]{amsart}
\usepackage[T1]{fontenc}
\usepackage[latin9]{inputenc}
\usepackage{verbatim}
\usepackage{mathrsfs}
\usepackage{amstext}
\usepackage{amsthm}
\usepackage{amssymb}
\usepackage{stmaryrd}

\makeatletter
\numberwithin{equation}{section}
\numberwithin{figure}{section}
\numberwithin{table}{section}
\theoremstyle{plain}
\newtheorem*{thm*}{\protect\theoremname}
\theoremstyle{plain}
\newtheorem{thm}{\protect\theoremname}[section]
\theoremstyle{definition}
\newtheorem{defn}[thm]{\protect\definitionname}
\theoremstyle{remark}
\newtheorem{rem}[thm]{\protect\remarkname}
\theoremstyle{plain}
\newtheorem{prop}[thm]{\protect\propositionname}
\theoremstyle{remark}
\newtheorem{notation}[thm]{\protect\notationname}
\theoremstyle{plain}
\newtheorem{lem}[thm]{\protect\lemmaname}
\theoremstyle{plain}
\newtheorem{cor}[thm]{\protect\corollaryname}
\theoremstyle{definition}
\newtheorem{example}[thm]{\protect\examplename}

\usepackage[colorlinks=true, linkcolor=blue]{hyperref}

\usepackage[toc, page]{appendix}
\usepackage{extarrow}
\usepackage{enumerate}
\usepackage{fontenc}
\usepackage[T1]{fontenc}
\usepackage{graphicx}
\usepackage{latexsym}
\usepackage{mathrsfs}
\usepackage{mathtext}
\usepackage{stmaryrd}
\usepackage{tikz}
\usepackage{textcomp}
\usepackage[colorinlistoftodos]{todonotes}
\usepackage{upgreek}
\usepackage{xy}

\usetikzlibrary{arrows}
\usetikzlibrary{matrix}
\usetikzlibrary{shapes}
\usetikzlibrary{snakes}
\usetikzlibrary{matrix}

\DeclareMathOperator{\Add}{\textup{Add}}
\DeclareMathOperator{\Aff}{\textup{Aff}}
\DeclareMathOperator{\Alg}{\textup{Alg}}
\DeclareMathOperator{\Ann}{\textup{Ann}}
\DeclareMathOperator{\Arr}{\textup{Arr}}
\DeclareMathOperator{\Art}{\textup{Art}}
\DeclareMathOperator{\Ass}{\textup{Ass}}
\DeclareMathOperator{\Aut}{\textup{Aut}}
\DeclareMathOperator{\Autsh}{\underline{\textup{Aut}}}
\DeclareMathOperator{\Bi}{\textup{B}}
\DeclareMathOperator{\CAdd}{\textup{CAdd}}
\DeclareMathOperator{\CAlg}{\textup{CAlg}}
\DeclareMathOperator{\CMon}{\textup{CMon}}
\DeclareMathOperator{\CPMon}{\textup{CPMon}}
\DeclareMathOperator{\CRings}{\textup{CRings}}
\DeclareMathOperator{\CSMon}{\textup{CSMon}}
\DeclareMathOperator{\CaCl}{\textup{CaCl}}
\DeclareMathOperator{\Cart}{\textup{Cart}}
\DeclareMathOperator{\Cl}{\textup{Cl}}
\DeclareMathOperator{\Coh}{\textup{Coh}}
\DeclareMathOperator{\Coker}{\textup{Coker}}
\DeclareMathOperator{\Cov}{\textup{Cov}}
\DeclareMathOperator{\Der}{\textup{Der}}
\DeclareMathOperator{\Div}{\textup{Div}}
\DeclareMathOperator{\End}{\textup{End}}
\DeclareMathOperator{\Endsh}{\underline{\textup{End}}}
\DeclareMathOperator{\Ext}{\textup{Ext}}
\DeclareMathOperator{\Extsh}{\underline{\textup{Ext}}}
\DeclareMathOperator{\FAdd}{\textup{FAdd}}
\DeclareMathOperator{\FCoh}{\textup{FCoh}}
\DeclareMathOperator{\FGrad}{\textup{FGrad}}
\DeclareMathOperator{\FLoc}{\textup{FLoc}}
\DeclareMathOperator{\FMod}{\textup{FMod}}
\DeclareMathOperator{\FPMon}{\textup{FPMon}}
\DeclareMathOperator{\FRep}{\textup{FRep}}
\DeclareMathOperator{\FSMon}{\textup{FSMon}}
\DeclareMathOperator{\FVect}{\textup{FVect}}
\DeclareMathOperator{\Fib}{\textup{Fib}}
\DeclareMathOperator{\Fibr}{\textup{Fibr}}
\DeclareMathOperator{\Fix}{\textup{Fix}}
\DeclareMathOperator{\Fl}{\textup{Fl}}
\DeclareMathOperator{\Fr}{\textup{Fr}}
\DeclareMathOperator{\Funct}{\textup{Funct}}
\DeclareMathOperator{\GAlg}{\textup{GAlg}}
\DeclareMathOperator{\GExt}{\textup{GExt}}
\DeclareMathOperator{\GHom}{\textup{GHom}}
\DeclareMathOperator{\GL}{\textup{GL}}
\DeclareMathOperator{\GMod}{\textup{GMod}}
\DeclareMathOperator{\GRis}{\textup{GRis}}
\DeclareMathOperator{\GRiv}{\textup{GRiv}}
\DeclareMathOperator{\Gal}{\textup{Gal}}
\DeclareMathOperator{\Gl}{\textup{Gl}}
\DeclareMathOperator{\Grad}{\textup{Grad}}
\DeclareMathOperator{\Hilb}{\textup{Hilb}}
\DeclareMathOperator{\Hl}{\textup{H}}
\DeclareMathOperator{\Hom}{\textup{Hom}}
\DeclareMathOperator{\Homsh}{\underline{\textup{Hom}}}
\DeclareMathOperator{\ISym}{\textup{Sym}^*}
\DeclareMathOperator{\Imm}{\textup{Im}}
\DeclareMathOperator{\Irr}{\textup{Irr}}
\DeclareMathOperator{\Iso}{\textup{Iso}}
\DeclareMathOperator{\Isosh}{\underline{\textup{Iso}}}
\DeclareMathOperator{\Ker}{\textup{Ker}}
\DeclareMathOperator{\LAdd}{\textup{LAdd}}
\DeclareMathOperator{\LAlg}{\textup{LAlg}}
\DeclareMathOperator{\LMon}{\textup{LMon}}
\DeclareMathOperator{\LPMon}{\textup{LPMon}}
\DeclareMathOperator{\LRings}{\textup{LRings}}
\DeclareMathOperator{\LSMon}{\textup{LSMon}}
\DeclareMathOperator{\Left}{\textup{L}}
\DeclareMathOperator{\Lex}{\textup{Lex}}
\DeclareMathOperator{\Loc}{\textup{Loc}}
\DeclareMathOperator{\M}{\textup{M}}
\DeclareMathOperator{\ML}{\textup{ML}}
\DeclareMathOperator{\MLex}{\textup{MLex}}
\DeclareMathOperator{\Map}{\textup{Map}}
\DeclareMathOperator{\Mod}{\textup{Mod}}
\DeclareMathOperator{\Mon}{\textup{Mon}}
\DeclareMathOperator{\Ob}{\textup{Ob}}
\DeclareMathOperator{\Obj}{\textup{Obj}}
\DeclareMathOperator{\PDiv}{\textup{PDiv}}
\DeclareMathOperator{\PGL}{\textup{PGL}}
\DeclareMathOperator{\PML}{\textup{PML}}
\DeclareMathOperator{\PMLex}{\textup{PMLex}}
\DeclareMathOperator{\PMon}{\textup{PMon}}
\DeclareMathOperator{\Pic}{\textup{Pic}}
\DeclareMathOperator{\Picsh}{\underline{\textup{Pic}}}
\DeclareMathOperator{\Pro}{\textup{Pro}}
\DeclareMathOperator{\Proj}{\textup{Proj}}
\DeclareMathOperator{\QAdd}{\textup{QAdd}}
\DeclareMathOperator{\QAlg}{\textup{QAlg}}
\DeclareMathOperator{\QCoh}{\textup{QCoh}}
\DeclareMathOperator{\QMon}{\textup{QMon}}
\DeclareMathOperator{\QPMon}{\textup{QPMon}}
\DeclareMathOperator{\QRings}{\textup{QRings}}
\DeclareMathOperator{\QSMon}{\textup{QSMon}}
\DeclareMathOperator{\Quot}{\textup{Quot}}
\DeclareMathOperator{\R}{\textup{R}}
\DeclareMathOperator{\Rep}{\textup{Rep}}
\DeclareMathOperator{\Rings}{\textup{Rings}}
\DeclareMathOperator{\Riv}{\textup{Riv}}
\DeclareMathOperator{\SFibr}{\textup{SFibr}}
\DeclareMathOperator{\SMLex}{\textup{SMLex}}
\DeclareMathOperator{\SMex}{\textup{SMex}}
\DeclareMathOperator{\SMon}{\textup{SMon}}
\DeclareMathOperator{\SchI}{\textup{SchI}}
\DeclareMathOperator{\Sh}{\textup{Sh}}
\DeclareMathOperator{\Soc}{\textup{Soc}}
\DeclareMathOperator{\Spec}{\textup{Spec}}
\DeclareMathOperator{\Specsh}{\underline{\textup{Spec}}}
\DeclareMathOperator{\Stab}{\textup{Stab}}
\DeclareMathOperator{\Supp}{\textup{Supp}}
\DeclareMathOperator{\Sym}{\textup{Sym}}
\DeclareMathOperator{\TMod}{\textup{TMod}}
\DeclareMathOperator{\Top}{\textup{Top}}
\DeclareMathOperator{\Tor}{\textup{Tor}}
\DeclareMathOperator{\Vect}{\textup{Vect}}
\DeclareMathOperator{\alt}{\textup{ht}}
\DeclareMathOperator{\car}{\textup{char}}
\DeclareMathOperator{\codim}{\textup{codim}}
\DeclareMathOperator{\degtr}{\textup{degtr}}
\DeclareMathOperator{\depth}{\textup{depth}}
\DeclareMathOperator{\divis}{\textup{div}}
\DeclareMathOperator{\et}{\textup{et}}
\DeclareMathOperator{\ffpSch}{\textup{ffpSch}}
\DeclareMathOperator{\h}{\textup{h}}
\DeclareMathOperator{\ilim}{\displaystyle{\lim_{\longrightarrow}}}
\DeclareMathOperator{\ind}{\textup{ind}}
\DeclareMathOperator{\indim}{\textup{inj dim}}
\DeclareMathOperator{\lf}{\textup{LF}}
\DeclareMathOperator{\op}{\textup{op}}
\DeclareMathOperator{\ord}{\textup{ord}}
\DeclareMathOperator{\pd}{\textup{pd}}
\DeclareMathOperator{\plim}{\displaystyle{\lim_{\longleftarrow}}}
\DeclareMathOperator{\pr}{\textup{pr}}
\DeclareMathOperator{\pt}{\textup{pt}}
\DeclareMathOperator{\rk}{\textup{rk}}
\DeclareMathOperator{\spec}{\textup{Spec}}
\DeclareMathOperator{\tr}{\textup{tr}}
\DeclareMathOperator{\type}{\textup{r}}
\DeclareMathOperator*{\colim}{\textup{colim}}

\usepackage[colorinlistoftodos]{todonotes}
\usepackage{enumitem}
\setenumerate[1]{label={\arabic*)}} 
\usepackage[a4paper]{geometry}

\makeatother

\usepackage{babel}
\providecommand{\corollaryname}{Corollary}
\providecommand{\definitionname}{Definition}
\providecommand{\examplename}{Example}
\providecommand{\lemmaname}{Lemma}
\providecommand{\notationname}{Notation}
\providecommand{\propositionname}{Proposition}
\providecommand{\remarkname}{Remark}
\providecommand{\theoremname}{Theorem}

\begin{document}
\title{Stack of $S_{3}$-covers}
\author{Fabio Tonini}
\address{Universitá degli Studi di Firenze, Dipartimento di Matematica e Informatica
\textquoteright Ulisse Dini\textquoteright , Viale Giovanni Battista
Morgagni, 67/A, 50134 Firenze, Italy}
\email{fabio.tonini@unifi.it}
\thanks{The author was supported by GNSAGA of INdAM}

\maketitle
\global\long\def\A{\mathbb{A}}%

\global\long\def\Ab{(\textup{Ab})}%

\global\long\def\C{\mathbb{C}}%

\global\long\def\Cat{(\textup{cat})}%

\global\long\def\Di#1{\textup{D}(#1)}%

\global\long\def\E{\mathcal{E}}%

\global\long\def\F{\mathbb{F}}%

\global\long\def\GCov{G\textup{-Cov}}%

\global\long\def\Gcat{(\textup{Galois cat})}%

\global\long\def\Gfsets#1{#1\textup{-fsets}}%

\global\long\def\Gm{\mathbb{G}_{m}}%

\global\long\def\GrCov#1{\textup{D}(#1)\textup{-Cov}}%

\global\long\def\Grp{(\textup{Grps})}%

\global\long\def\Gsets#1{(#1\textup{-sets})}%

\global\long\def\HCov{H\textup{-Cov}}%

\global\long\def\MCov{\textup{D}(M)\textup{-Cov}}%

\global\long\def\MHilb{M\textup{-Hilb}}%

\global\long\def\N{\mathbb{N}}%

\global\long\def\PGor{\textup{PGor}}%

\global\long\def\PGrp{(\textup{Profinite Grp})}%

\global\long\def\PP{\mathbb{P}}%

\global\long\def\Pj{\mathbb{P}}%

\global\long\def\Q{\mathbb{Q}}%

\global\long\def\RCov#1{#1\textup{-Cov}}%

\global\long\def\RR{\mathbb{R}}%

\global\long\def\Sch{\textup{Sch}}%

\global\long\def\WW{\textup{W}}%

\global\long\def\Z{\mathbb{Z}}%

\global\long\def\acts{\curvearrowright}%

\global\long\def\alA{\mathscr{A}}%

\global\long\def\alB{\mathscr{B}}%

\global\long\def\alC{\mathscr{C}}%

\global\long\def\alD{\mathscr{D}}%

\global\long\def\alE{\mathscr{E}}%

\global\long\def\alF{\mathscr{F}}%

\global\long\def\alG{\mathscr{G}}%

\global\long\def\alH{\mathscr{H}}%

\global\long\def\alI{\mathscr{I}}%

\global\long\def\alJ{\mathscr{J}}%

\global\long\def\alK{\mathscr{K}}%

\global\long\def\alL{\mathscr{L}}%

\global\long\def\alM{\mathscr{M}}%

\global\long\def\alN{\mathscr{N}}%

\global\long\def\alO{\mathscr{O}}%

\global\long\def\alP{\mathscr{P}}%

\global\long\def\alQ{\mathscr{Q}}%

\global\long\def\alR{\mathscr{R}}%

\global\long\def\alS{\mathscr{S}}%

\global\long\def\alT{\mathscr{T}}%

\global\long\def\alU{\mathscr{U}}%

\global\long\def\alV{\mathscr{V}}%

\global\long\def\alW{\mathscr{W}}%

\global\long\def\alX{\mathscr{X}}%

\global\long\def\alY{\mathscr{Y}}%

\global\long\def\alZ{\mathscr{Z}}%

\global\long\def\arr{\longrightarrow}%

\global\long\def\arrdi#1{\xlongrightarrow{#1}}%

\global\long\def\catC{\mathscr{C}}%

\global\long\def\catD{\mathscr{D}}%

\global\long\def\catF{\mathscr{F}}%

\global\long\def\catG{\mathscr{G}}%

\global\long\def\comma{,\ }%

\global\long\def\covU{\mathcal{U}}%

\global\long\def\covV{\mathcal{V}}%

\global\long\def\covW{\mathcal{W}}%

\global\long\def\duale#1{{#1}^{\vee}}%

\global\long\def\fasc#1{\widetilde{#1}}%

\global\long\def\fsets{(\textup{f-sets})}%

\global\long\def\iL{r\mathscr{L}}%

\global\long\def\id{\textup{id}}%

\global\long\def\la{\langle}%

\global\long\def\odi#1{\mathcal{O}_{#1}}%

\global\long\def\ra{\rangle}%

\global\long\def\rig{\mathbin{\!\!\pmb{\fatslash}}}%

\global\long\def\set{(\textup{Sets})}%

\global\long\def\sets{(\textup{Sets})}%

\global\long\def\shA{\mathcal{A}}%

\global\long\def\shB{\mathcal{B}}%

\global\long\def\shC{\mathcal{C}}%

\global\long\def\shD{\mathcal{D}}%

\global\long\def\shE{\mathcal{E}}%

\global\long\def\shF{\mathcal{F}}%

\global\long\def\shG{\mathcal{G}}%

\global\long\def\shH{\mathcal{H}}%

\global\long\def\shI{\mathcal{I}}%

\global\long\def\shJ{\mathcal{J}}%

\global\long\def\shK{\mathcal{K}}%

\global\long\def\shL{\mathcal{L}}%

\global\long\def\shM{\mathcal{M}}%

\global\long\def\shN{\mathcal{N}}%

\global\long\def\shO{\mathcal{O}}%

\global\long\def\shP{\mathcal{P}}%

\global\long\def\shQ{\mathcal{Q}}%

\global\long\def\shR{\mathcal{R}}%

\global\long\def\shS{\mathcal{S}}%

\global\long\def\shT{\mathcal{T}}%

\global\long\def\shU{\mathcal{U}}%

\global\long\def\shV{\mathcal{V}}%

\global\long\def\shW{\mathcal{W}}%

\global\long\def\shX{\mathcal{X}}%

\global\long\def\shY{\mathcal{Y}}%

\global\long\def\shZ{\mathcal{Z}}%

\global\long\def\st{\ | \ }%

\global\long\def\stA{\mathcal{A}}%

\global\long\def\stB{\mathcal{B}}%

\global\long\def\stC{\mathcal{C}}%

\global\long\def\stD{\mathcal{D}}%

\global\long\def\stE{\mathcal{E}}%

\global\long\def\stF{\mathcal{F}}%

\global\long\def\stG{\mathcal{G}}%

\global\long\def\stH{\mathcal{H}}%

\global\long\def\stI{\mathcal{I}}%

\global\long\def\stJ{\mathcal{J}}%

\global\long\def\stK{\mathcal{K}}%

\global\long\def\stL{\mathcal{L}}%

\global\long\def\stM{\mathcal{M}}%

\global\long\def\stN{\mathcal{N}}%

\global\long\def\stO{\mathcal{O}}%

\global\long\def\stP{\mathcal{P}}%

\global\long\def\stQ{\mathcal{Q}}%

\global\long\def\stR{\mathcal{R}}%

\global\long\def\stS{\mathcal{S}}%

\global\long\def\stT{\mathcal{T}}%

\global\long\def\stU{\mathcal{U}}%

\global\long\def\stV{\mathcal{V}}%

\global\long\def\stW{\mathcal{W}}%

\global\long\def\stX{\mathcal{X}}%

\global\long\def\stY{\mathcal{Y}}%

\global\long\def\stZ{\mathcal{Z}}%

\global\long\def\then{\ \Longrightarrow\ }%

\global\long\def\L{\textup{L}}%

\global\long\def\det{\textup{det}}%

\global\long\def\l{\textup{l}}%

\begin{abstract}
The aim of this paper is to study the geometry of the stack of $S_{3}$-covers.
We show that it has two irreducible components $\stZ_{S_{3}}$ and
$\stZ_{2}$ meeting in a ``degenerate'' point $\{0\}$, $\stZ_{2}-\{0\}\simeq\Bi\GL_{2}$,
while $(\stZ_{S_{3}}-\{0\})$, which contains $\Bi S_{3}$ as open
substack, is a smooth and universally closed algebraic stack. More
precisely we show that $\stZ_{S_{3}}-\{0\}\simeq[X/\GL_{2}]$, where
$X$ is an explicit smooth non degenerate projective surface inside
$\PP^{7}$ intersection of five quadrics.

All these results are based on the description of certain families
of $S_{3}$-covers in terms of ``building data''.
\end{abstract}

\section*{Introduction}

Covers, ramified or unramified, with or without a group action, appear
in many context of algebraic geometry (as well as other branches of
mathematics) and they are often used to construct new varieties as
total space of covers of a given variety. It is therefore useful to
have a ``recipe'' for building up covers from the geometry of the
base variety, in terms of divisors or vector bundles. The easiest
example is the correspondence between double covers and line bundles
together with a section of its second tensor power (provided the characteristic
is not $2$). Other classical results are \cite{Miranda1985}, \cite{Pardini1989}
about triple covers and \cite{Pardini1991} about abelian covers.
We also recall: \cite{Casnati1996,Casnati1996a} about Gorenstein
covers, \cite{Easton2008} about $S_{3}$-covers as in the present
paper, \cite{Tokunaga1994,Reimpell1999,Catanese2017} about dihedral
covers, \cite{Tokunaga2002} about $S_{4}$ and $A_{4}$-covers, \cite{Hahn1999}
about quadruple covers, \cite{Alexeev2011} about non-normal abelian
covers, \cite{Ahlqvist2020} about stacky covers, \cite{Biswas2017}
about tamely ramified covers.

Besides constructing varieties, covers have been used to define some
moduli spaces, especially moduli of covers of curves or surfaces,
for example Hurwitz spaces (see e.g. \cite{Bertin2013}). In this
context we recall also \cite{Pagani2013,Arsie2004,Bolognesi2009,Poma2015}.
In \cite{Tonini2013} I introduced the moduli stack $\GCov$ of $G$-covers
for a finite, flat and finitely presented group scheme $G$ over some
base $S$. Although $\GCov$ and the various moduli of covers of curves/surfaces
are both stacks parametrizing covers, we want to stress that they
are very different objects: the geometric points in the first case
are particular finite schemes (covers of a point) while in the second
case are covers of curves/surfaces.

The problem of studying the geometry of $\GCov$ is strictly linked
with the problem of finding a ``recipe'' or, using the language
of \cite{Pardini1991}, a ``building data'' for constructing $G$-covers,
with the difference that we need to describe covers of \emph{any }scheme,
with no geometric restriction on source or target (``classically''
the base scheme is often assumed to be integral).

The stack $\GCov$ is algebraic and finitely presented over $S$ and
contains $\Bi G$, the stack of $G$-torsors, as an open substack.
In particular it has a special component $\stZ_{G}$, called the \emph{main
irreducible component}, which is the schematic closure of $\Bi G\subseteq\GCov$.
In \cite{Tonini2013} I investigated the geometry of those stacks
in the abelian case, more precisely when $G$ is a diagonalizable
group scheme over $S=\Spec\Z$: it turns out that, except some cases
in small ranks, this geometry can be very ``wild'' as one may expect
from an Hilbert scheme. Nevertheless one can restrict the study to
certain loci of $\GCov$ and, for instance, describe the smooth locus
of $\stZ_{G}$.

Moving to the \emph{non-abelian }case the situation worsen, because
there are no simple description of $G$-covers and therefore no obvious
way to study the geometry of $\GCov$. In \cite{Tonini2015} I propose
an alternative interpretation of those covers as (non necessarily
strong) monoidal functor, leading to some information about the geometry
of $\GCov$, for instance its reducibility for linearly reductive
non-abelian groups $G$.

For simplicity let $k$ be an algebraically closed field and assume
that $G$ is linearly reductive over $k$. A $G$-cover $f\colon X\to Y$
is completely determined by a collection of vector bundles (determining
the module $f_{*}\odi X$), one for each irreducible representation
of $G$ and with equal rank, and a collection of maps between tensor
products of those bundles (determining the ring structure of $f_{*}\odi X$).
This data is very simple to describe, but it has to satisfy certain
compatibility conditions (corresponding to the commutativity and associativity
of $f_{*}\odi X$), which are expressed as commutative diagrams of
maps between vector bundles. The complexity of the non-abelian case
lies in the complexity and numerousness of those diagrams.

In this paper we consider the simplest non-abelian group $G=S_{3}$
for $\car k\neq2,3$. In this case the complexity we discussed can
be handled directly by listing all conditions and making sense of
them. This lead to a ``concrete'' set of data describing $S_{3}$-covers:
a line bundle $\shL$, a rank $2$ vector bundle $\shF$ and maps
$\alpha\colon\shL\otimes\shF\to\shF$, $\beta\colon\Sym^{2}\shF\to\shF$
and $\la-,-\ra\colon\det\shF\to\shL$ making $5$ diagrams commutes
(see \ref{thm:description S3 covers by diagrams}). From this description
one deduce that $S_{3}\textup{-Cov}\simeq[U/(\Gm\times\GL_{2})]$
for an explicit closed subscheme $U$ of $\A_{k}^{11}$ (see \ref{thm:GCov as quotient stack}).
We summarize the results obtained about the geometry of $S_{3}\textup{-Cov}$
in the following:
\begin{thm*}
(\ref{thm:geometry of GCov and components}, \ref{thm:exceptional irreducible component},
\ref{thm:description of the main component} and \ref{cor:the projective surface})
The stack $S_{3}\textup{-Cov}$ has two irreducible components $\stZ_{S_{3}}$
and $\stZ_{2}$ and they meet in a point $0\in S_{3}\textup{-Cov}$
(which corresponds to the ``degenerate'' $S_{3}$-cover). Moreover
the stack $S_{3}\textup{-Cov}-\{0\}$ (resp. $\stZ_{S_{3}}-\{0\}$)
is the smooth locus of $S_{3}\textup{-Cov}$ (resp. $\stZ_{S_{3}}$),
while $\stZ_{2}\simeq[\A^{1}/\Gm]\times\Bi\GL_{2}$. Finally $\stZ_{S_{3}}-\{0\}$
is universally closed, more precisely 
\[
\stZ_{S_{3}}-\{0\}\simeq[X/\GL_{2}]
\]
 for an explicit non degenerate smooth projective surface $X\subseteq\PP_{k}^{7}$
complete intersection of five quadrics.
\end{thm*}
The above result is obtained by studying the equations of $U\subseteq\A^{11}$,
describing certain loci of $S_{3}\textup{-Cov}$ and mixing these
points of view. The number of equations necessary to define $U\subseteq\A^{11}$
is proof of the complexity of the data associated with $S_{3}$-covers
and it suggests that the study of $G$-covers for general groups $G$
needs an alternative approach (see \ref{eq:loc com and ass conditions}).

We describe several open substacks of $S_{3}\textup{-Cov}$ (and of
$\stZ_{S_{3}}$), namely the locus $\stU_{\omega}$ where $\la-,-\ra$
is an isomorphism, which turns out to be equivalent to the stack of
triple covers (see \ref{thm:The-locus when omega is invertible}),
the locus $\stU_{\alpha}$ where $\alpha$ is nowhere a multiple of
the identity (see \ref{thm:not degenerate locus of S3}) and the locus
$\stU_{\beta}$ where $\beta$ is nowhere zero (see \ref{thm:the locus where beta is not zero}).
We show that $\stZ_{S_{3}}-\{0\}=\stU_{\omega}\cup\stU_{\alpha}\cup\stU_{\beta}$
and it is the smooth locus of $\stZ_{S_{3}}$ (see \ref{thm:description of the main component}).
For each of those open substacks and also for $\stZ_{S_{3}}$ itself
we determine simpler sets of ``data'' for describing its $S_{3}$-covers
(i.e. its objects).

As mentioned before, in \cite{Easton2008} the author develops a similar
theory of $S_{3}$-covers, describing, locally and globally, the algebra
defining $S_{3}$-covers. The data provided coincides with the one
we associate with covers in $\stZ_{S_{3}}$. The present paper recovers
the results of \cite{Easton2008} and expands them in several directions,
by studying the geometry of $S_{3}\textup{-Cov}$ and by describing
several families of $S_{3}$-covers.

For simplicity, in this introduction, we assumed to work over an algebraically
closed field, but all results actually hold in general over $\Z[1/6]$.
Moreover, instead of looking directly at the group $S_{3}$, we work
with $G=\mu_{3}\rtimes\Z/2\Z$ over the ring $\Z[1/2]$. This is because
the group $G$ has a simpler representation theory, which simplifies
the description of $G$-covers. Over $\Z[1/6]$ the groups $G$ and
$S_{3}$ are isomorphic only étale locally, nevertheless, via the
theory of bitorsors discussed in the Appendix, we show that there
is an equivalence $\GCov\simeq S_{3}\textup{-Cov}$ inducing $\stZ_{G}\simeq\stZ_{S_{3}}$
and $\Bi G\simeq\Bi S_{3}$ (see \ref{ex:SThree covers and the other group}).

This paper follows ideas from the last chapter of my Ph.D. thesis
\cite{Tonini2013a}, but it introduces some improvements, as the study
of the projective surface covering $\stZ_{S_{3}}-\{0\}$. In \cite{Tonini2013a}
it is also present a characterization of $S_{3}$-covers between regular
schemes and some applications to surfaces. We plan to discuss and
strengthen these results in a subsequent paper, applying also criteria
from \cite{Tonini2015} and \cite{Tonini2017b}.

This paper is divided as follows. In the first section we apply the
theory of bitorsors to the theory of $G$-covers. The second section
defines the global ``building data'' for $S_{3}$-covers, while
the third one studies the geometry of $S_{3}$-Cov. The fourth and
last section focuses instead on the geometry of $\stZ_{S_{3}}$. The
are two appendices, the first one about the theory of bitorsors, the
second one about some general results on vector bundles.

\section*{Acknowledgements}

I would like to thank Rita Pardini, Mattia Talpo and Angelo Vistoli
for the useful conversations I had with them and all the suggestions
they gave me.

\section*{Notation}

We denote by the letter $T$ a scheme (over the given base if this
is specified). It will be used as base for various algebro-geometric
objects.

By a locally free sheaf we mean a locally free sheaf of finite rank.

A cover of $T$ is a finite, flat and finitely presented morphism
$f\colon X\to T$. This is the same as an affine map $f\colon X\to T$
such that $f_{*}\odi X$ is a locally free sheaf.

If $G$ is a group scheme over a base $S$ we denote by $\Loc^{G}T$
(resp. $\QCoh^{G}T$) the category of locally free sheaves (resp.
quasi-coherent sheaves) over $T$ together with an action of $G$.
When $\pi\colon G\to S$ is affine, such an action is equivalent to
a coaction of the sheaf of Hopf algebras $\pi_{*}\odi G$.

If $\shF$ is a locally free sheaf over $T$ with a given basis $v_{1},\dots,v_{n}\in\shF$
we denote by $v_{1}^{*},\dots,v_{n}^{*}\in\shF^{\vee}$ its dual basis.

\section{Bitorsors and $G$-covers}

In this section $S$ is a base scheme and $G\to S$ is a finite, flat
and finitely presented group scheme. We recall various definitions
and properties about $G$-covers.
\begin{defn}
\cite[Def. 2.1]{Tonini2013}\cite[Def 1.2]{Tonini2015} A $G$-cover
over a scheme $T$ is a cover $f\colon X\to T$ together with an action
of $G$ on $X$ over $T$ such that $f_{*}\odi X$ is fppf locally
isomorphic to the regular representation $\odi T[G]$ as quasi-coherent
sheaves with an action of $G$ (but not necessarily preserving the
ring structure).

We denote by $\GCov$ the stack of $G$-covers over $S$.
\end{defn}

\begin{rem}
\cite[Prop 2.2, Thm 2.10]{Tonini2013} The stack $\GCov$ is algebraic
and finitely presented over $S$ and contains $\Bi G$ as an open
substack.
\end{rem}

\begin{defn}
\label{def:main irreducible component} \cite[Def 3.5]{Tonini2015}
The stack $\stZ_{G}$ is the schematic closure of the open immersion
$\Bi G\to\GCov$ and it is called the main irreducible component of
$\GCov$.
\end{defn}

We apply now the theory of bitorsors (see Appendix \ref{sec:Bitorsors})
to the theory of Galois covers.
\begin{thm}
\label{thm:bitorsors and GCov}Let $G$ and $H$ be flat, finite and
finitely presented group schemes over a base scheme $S$. If $P$
is an fpqc $(G,H)$-bitorsor over $S$ then the functor $\Lambda_{P}$
of \ref{prop:bitorsors and isomorphisms SHG - SHH} fits in a commutative
diagram   \[   \begin{tikzpicture}[xscale=1.9,yscale=-1.2]     \node (A0_0) at (0, 0) {$\Bi G$};     \node (A0_1) at (1, 0) {$\stZ_G$};     \node (A0_2) at (2, 0) {$\GCov $};     \node (A0_3) at (3, 0) {$\Sh_{\Sch/S}^G$};     
\node (A0_4) at (3.8, 0) {$X$};    
\node (A1_0) at (0, 1) {$\Bi H$};     \node (A1_1) at (1, 1) {$\stZ_H$};     \node (A1_2) at (2, 1) {$\HCov $};     \node (A1_3) at (3, 1) {$\Sh_{\Sch/S}^H$};     
\node (A1_4) at (3.8, 1) {$\frac{X\times P}{G}$};     
\path (A1_2) edge [->]node [auto] {$\scriptstyle{}$} (A1_3);     \path (A0_0) edge [->]node [auto] {$\scriptstyle{}$} (A0_1);     \path (A0_1) edge [->]node [auto] {$\scriptstyle{}$} (A0_2);     \path (A1_0) edge [->]node [auto] {$\scriptstyle{}$} (A1_1);     \path (A0_3) edge [->]node [auto] {$\scriptstyle{\Lambda_P}$} (A1_3);     \path (A0_2) edge [->]node [auto] {$\scriptstyle{}$} (A1_2);     \path (A0_4) edge [|->]node [auto] {$\scriptstyle{}$} (A1_4);     \path (A1_1) edge [->]node [auto] {$\scriptstyle{}$} (A1_2);     \path (A0_0) edge [->]node [auto] {$\scriptstyle{}$} (A1_0);     \path (A0_1) edge [->]node [auto] {$\scriptstyle{}$} (A1_1);     \path (A0_2) edge [->]node [auto] {$\scriptstyle{}$} (A0_3);   \end{tikzpicture}   \] where the vertical arrows are equivalences. Moreover if $Y$ is an
$S$-scheme and $p\in P(Y)$ the induced composition $G_{Y}\to P_{Y}\to H_{Y}$
is an isomorphism of groups and $X\to\Lambda_{P}(X)$ is a natural
isomorphism, equivariant with respect to $G_{Y}\to H_{Y}$, for $X\in\Sh_{\Sch/Y}^{G}$.
\end{thm}

\begin{proof}
The functor $\Lambda_{P}\colon\Sh_{\Sch/S}^{G}\to\Sh_{\Sch/S}^{H}$
is an equivalence thanks to \ref{prop:bitorsors and isomorphisms SHG - SHH}.
Taking into account \ref{rem:trivial bitorsor}, it is enough to show
that $\Lambda_{P}(X)$ is an $H$-cover if and only if $X$ is a $G$-cover.
Indeed any equivalence $\GCov\to\HCov$ restricting to an equivalence
$\Bi G\to\Bi H$ has to induce an equivalence $\stZ_{G}\to\stZ_{H}$
of their schematic closures. Since being a $G$-cover or $H$-cover
is a fppf local property of $G$-sheaves and fpqc $G$-torsors are
fppf locally trivial, we can assume that $P$ has a global section.
The result then follows from \ref{rem:trivial bitorsor}.
\end{proof}
\begin{thm}
\label{ex:SThree covers and the other group} Let $G=\mu_{n}\rtimes(\Z/n\Z)^{*}$
and $H=\Z/n\Z\rtimes(\Z/n\Z)^{*}$ for $n\geq3$. The scheme 
\[
P=\mu_{n}\times\mu_{n}^{*}=\Spec(A_{P})\to\Spec\Z[1/n]\text{ where }A_{P}=\frac{\Z[1/n][x,y]}{(x^{n}-1,\Phi_{n}(y))}
\]
$\mu_{n}^{*}\subseteq\mu_{n}$ is the open and closed subscheme of
primitive $n$-th roots and $\Phi_{n}$ is the cyclotomic polynomial
of degree $n$, is a $(G,H)$-bitorsor with biaction 
\[
G\times P\times H\to P\comma(\zeta,l)\cdot(x,y)\cdot(i,m)=(\zeta x^{l}y^{li},y^{lm})
\]
In particular the functor $\Lambda_{P}$ of \ref{prop:bitorsors and isomorphisms SHG - SHH}
induces equivalences as in \ref{thm:bitorsors and GCov} for $S=\Spec\Z[1/n]$.
Moreover there is a canonical isomorphism 
\[
X/(\Z/n\Z)^{*}\simeq\Lambda(X)/(\Z/n\Z)^{*}\text{ for }X\in\Sh_{\Z[1/n]}^{G}
\]
For $X=\Spec\alA\in\GCov$ the $H$-cover $\Lambda_{P}(X)$ is the
spectrum of the sub-algebra 
\[
\alB=\left(\frac{\alA[x,y]}{(x^{n}-1,\Phi_{n}(y))}\right)^{G}\subseteq\frac{\alA[x,y]}{(x^{n}-1,\Phi_{n}(y))}=\alA\otimes A_{P}
\]
The (left) $H$-action on $\alA\otimes A_{P}$ is trivial on $\alA$
and given by $(i,m)x=xy^{i},(i,m)y=y^{m}$ on $A_{P}$. The (left)
$G$-action on $\alA\otimes A_{P}$ is the given one on $\alA$, while
on $A_{P}$ is generated by the $\mu_{n}$-action for which $\deg x=-1$,
$\deg y=0$, while $l\in(\Z/n\Z)^{*}$ acts as $x\mapsto x^{l'}$,
$y\mapsto y^{l'}$ where $l'=l^{-1}\in(\Z/n\Z)^{*}$.

If $Y=\mu_{n}^{*}=\Spec(\Z[1/n][w]/(\Phi_{n}(w))$ the section $(1,w)\in P(Y)$
induces the group isomorphism $\phi\colon H_{Y}\to G_{Y}$, $(i,m)\mapsto(w^{i},m)$
and, for $\alA\in G_{Y}\textup{-Cov}$, the map 
\[
\alB\subseteq\alA\otimes A_{P}\to\alA\comma x\mapsto1,y\mapsto w
\]
is an isomorphism, equivariant with respect to $\phi\colon H_{Y}\to G_{Y}$.
\end{thm}

\begin{proof}
We apply \ref{prop:bitorsors and semidirect products}. Taking into
account that $\Autsh(\mu_{n})=\Autsh(\Z/n\Z)=(\Z/n\Z)^{*}$ and $\mu_{n},\Z/n\Z$
are étale locally isomorphic over $\Z[1/n]$, the scheme $P=\mu_{n}\times\Isosh(\Z/n\Z,\mu_{n})$
is a $(G,H)$-bitorsor. Part of the statement follows directly from
\ref{thm:bitorsors and GCov}. The remaining part consists in giving
a more precise description of $P$. We have that
\[
\Isosh(\Z/n\Z,\mu_{n})\to\Homsh(\Z/n\Z,\mu_{n})=\mu_{n}
\]
is the locus of $\omega\in\mu_{n}$ such that the induced map $\Z/n\Z\to\mu_{n}$
is an isomorphism. Just looking at the order of $\omega$, we see
that $\omega\in\mu_{n}-\mu_{d}$ for any $d\mid n$ with $d<n$. As
$\mu_{d}\subseteq\mu_{n}$ is an étale subgroup, it is an open and
closed subscheme. Therefore $\omega\in\mu_{n}^{*}=\mu_{n}-\cup_{d\mid n,d<n}\mu_{d}$,
which is an open and closed subscheme. Morever we also have $\mu_{n}^{*}=\Spec(\Z[1/n][y]/(\Phi_{n}(y))$
by definition of $\Phi_{n}$. The condition $\omega\in\mu_{n}^{*}$
means that $\omega$ is a primitive $n$-th root in all the residue
fields of its base. This easily implies the equality $\Isosh(\Z/n\Z,\mu_{n})=\mu_{n}^{*}$.
The description of $P$ in the statement follows, while the biaction
of $G$ and $H$ can be computed directly from the definition in \ref{prop:bitorsors and semidirect products}.

We are left with the second part of the statement, so let $X=\Spec\alA$
be a $G$-cover. By definition $\Lambda_{P}(X)=(X\times P)/G$, where
$G$ acts on the right: $(x,p)g=(xg,g^{-1}p)$. In particular $\Lambda_{P}(X)$
is the spectrum of $(\alA\otimes A_{P})^{G}\subseteq\alA\otimes A_{P}$.
Here the right $G$-action on $P$, $-\star g=g^{-1}-$ induces a
left $G$-action on (the functor associated with) $A_{P}$: $\zeta\in\mu_{n}$
acts as $x\mapsto\zeta^{-1}x$, $y\mapsto y$, while $(1,l)$, since
its inverse in $G$ is $(1,l'),$acts as $x\mapsto x^{l'}$, $y\mapsto y^{l'}$.
In particular, as $\mu_{n}$-comodule, $A_{p}$ satisfies $\deg x=-1$
and $\deg y=0$ and the total $G$ action on $\alA\otimes A_{P}$
is the diagonal one, as claimed in the statement. The $H$-action
on $\Lambda_{P}(X)=(X\times P)/G$ is non trivial only on $P$, from
which we deduce the $H$-action on $\alA\otimes A_{P}$ and its subalgebra.

The last part follows directly from the last part of Theorem \ref{thm:bitorsors and GCov}.
\end{proof}

\section{Data for \texorpdfstring{$(\mu_{3}\rtimes\Z/2\Z)$}{(mu3 rtimes Z/2Z)}-covers
or $S_{3}$-covers\label{sec:Data for covers}}

In this section we work over the ring $\shR$ of integers with $2$
inverted, that is $\stR=\Z[1/2]$ and with the symbol $G$ we will
always denote the group scheme $G=\mu_{3}\rtimes\Z/2\Z$ defined over
$\stR$, where the action of $\Z/2\Z$ on $\mu_{3}$ is given by the
inversion, that is $\Z/2\Z\simeq\Autsh(\mu_{3})$. Note that, in this
case, $\mu_{2}\simeq\Z/2\Z$. We denote by $\sigma\in\Z/2\Z(\stR)$
the non trivial generator of $\Z/2\Z$. We will also think of $\sigma$
as an element of $G(\stR)$.

\subsection{The group \texorpdfstring{$(\mu_{3}\rtimes\Z/2\Z)$}{(mu3 rtimes Z/2Z)}
and its representation theory.\label{subsec:representation theory}}

The group $G$ is a linearly reductive group over $\shR$ (see \cite[Prop 2.6, Thm 2.16]{Abramovich2007}).

Set $V_{0}=\shR,V_{1},V_{2}$ for the representations of $\mu_{3}$
corresponding to its characters in $\Z/3\Z$. Moreover consider the
set $I_{G}$ of $G$-representations

\[
\shR\comma A=V_{\chi}\comma V=\ind_{\mu_{3}}^{G}V_{1}
\]
where $\chi\colon G\arr\Gm$ is induced by the non trivial character
of $\Z/2\Z$. Since $2$ is invertible in $\shR$, the representations
in $I_{G}$ restricts over $\overline{\Q}$ to the irreducible representations
of $G\times\overline{\Q}\simeq S_{3}$. In other words we have proved
that:
\begin{prop}
The pair $(G,I_{G})$ is a good linearly reductive group over $\shR$
in the sense of \cite[Def 1.11]{Tonini2015}.
\end{prop}

We setup the following notation and we will use it throughout the
paper. We consider the following basis $1\in\shR$, $1_{A}\in A$
and $v_{1},v_{2}\in V=V_{1}\oplus V_{2}$ such that $v_{i}\in V_{i}$.
Moreover since $\sigma$ exchanges $V_{1}$ and $V_{2}$, we also
assume that $\sigma(v_{1})=v_{2},\sigma(v_{2})=v_{1}$. Now we describe
the tensor products of the representations in $I_{G}$. We have
\[
A\otimes A\simeq\stR\comma1_{A}\otimes1_{A}\arr1\text{ and }A\otimes V\simeq V\comma1_{A}\otimes v_{1}\arr-v_{1},1_{A}\otimes v_{2}\arr v_{2}
\]
and, if we set $v_{ij}=v_{i}\otimes v_{j}\in V\otimes V$, 
\[
\stR\oplus A\oplus V\simeq V\otimes V\comma1\arr v_{12}+v_{21},1_{A}\arr v_{12}-v_{21},v_{1}\arr v_{22},v_{2}\arr v_{11}
\]
Finally note that the $G$-equivariant projection $V\otimes V\arr\stR\comma v_{ij}\arr1-\delta_{ij}$,
where $\delta_{ij}$ is the Kronecker symbol, yields an isomorphism
\[
V\simeq\duale V\comma v_{1}\arr v_{2}^{*},v_{2}\arr v_{1}^{*}
\]
The above discussion allows us to conclude the following:

\subsection{\label{sec:Global description of Sthree covers}Global description
of \texorpdfstring{$(\mu_{3}\rtimes\Z/2\Z)$}{(mu3 rtimes Z/2Z)}-covers.}

In this section we want to describe the data needed to define a $G$-cover
over any $\stR$-scheme. We proceed in the following way. First we
introduce such data, then we will state the precise relationship with
$G$-covers and only after we will prove all the claims. We remark
here that the global description obtained, although with a different
notation, has already been introduced in \cite{Easton2008}.

We define the stack $\stY$ over $\stR$ whose objects over an $\shR$-scheme
$T$ are sequences $\chi=(\shL,\shF,\alpha,\beta,\la-,-\ra)$ where:
$\shL$ is an invertible sheaf, $\shF$ is a rank $2$ locally free
sheaf and $\alpha,\beta,\la-,-\ra$ are maps
\[
\shL\otimes\shF\arrdi{\alpha}\shF\comma\Sym^{2}\shF\arrdi{\beta}\shF\comma\det\shF\arrdi{\la-,-\ra}\shL
\]
With an object $\chi\in\stY$ as above we associate the map $(-,-)_{\chi}\colon\shF\otimes\shF\arr\odi T$
given by
\begin{equation}
(-,-)_{\chi}\colon\shF\otimes\shF\simeq\duale{\shF}\otimes\det\shF\otimes\shF\arrdi{\id\otimes\la-,-\ra\otimes\id}\duale{\shF}\otimes\shL\otimes\shF\arrdi{\id\otimes\alpha}\duale{\shF}\otimes\shF\arr\odi T\label{eq:symmetric bracket from alternating one}
\end{equation}
where we are using the canonical isomorphism $\shF\simeq\duale{\shF}\otimes\det\shF$.
Notice that, although we are using the symbol $(-,-)$ of a symmetric
product, $(-,-)_{\chi}$ is not necessarily symmetric. Moreover we
also associate with $\chi$ the maps $\gamma_{\chi},\gamma_{\chi}'\colon\shF\otimes\shF\arr\odi T\oplus\shL$
given by
\[
\gamma_{\chi}=(-,-)_{\chi}+\la-,-\ra\comma\gamma'_{\chi}=(-,-)_{\chi}-\la-,-\ra
\]
Finally we define (see \ref{not: trace of a map})
\begin{equation}
m_{\chi}=(1/2)\tr(\shL^{2}\otimes\shF\arrdi{\id_{\shL}\otimes\alpha}\shL\otimes\shF\arrdi{\alpha}\shF)\colon\shL^{2}\to\odi T\label{eq:m from alpha}
\end{equation}
and 
\[
\alA_{\chi}=\odi T\oplus\shL\oplus\shF_{1}\oplus\shF_{2}\text{ with }\shF_{1}=\shF_{2}=\shF
\]
For convenience we also set $\shL_{\chi}=\shL$, $\shF_{\chi}=\shF$,
$\alpha_{\chi}=\alpha$, $\beta_{\chi}=\beta$ and $\langle-,-\rangle_{\chi}=\langle-,-\rangle$:
given $\chi\in\stY$ we don't need to specify the whole sequence to
refer to one of its elements, for instance we could simply write $\beta_{\chi}=0$
and so on. On the other hand, when $\chi$ is given and there is no
possibility of confusion, we will omit the $-_{\chi}$ and simply
write $(-,-),\gamma,\gamma',m,\alA$ or $\chi=(\shL,\shF,m,\alpha,\beta,(-,-),\la-,-\ra)\in\stY$.
\begin{defn}
We denote by $\LRings_{\shR}^{G}$ the stack of locally free shaves
$\alA$ of finite rank with a coaction of $G$ and an equivariant
multiplication map $\alA\otimes\alA\to\alA$ (not necessarilly commutative
or associative) with an invariant unit $1\in\alA^{G}$.
\end{defn}

\begin{prop}
\label{prop:Achi is G-equivariant}Given $\chi\in\stY$ as above,
the sheaf $\alA_{\chi}$ has a unique $G$-comodule structure such
that $\shF_{0}=\odi T\oplus\shL,\shF_{1},\shF_{2}$ define the $\mu_{3}$-action
and $\sigma$ acts as $-\id_{\shL}$ on $\shL$ and induces $\id_{\shF}\colon\shF_{1}\arr\shF_{2}$.
\end{prop}

This proposition will be proved later. We makes $\alA_{\chi}$ into
an object of $\LRings_{\shR}^{G}$ defining the following multiplication
map $\alA_{\chi}\otimes\alA_{\chi}\to\alA_{\chi}$:

\begin{gather}
\shL^{2}\arrdi m\odi T\comma\shL\otimes\shF_{1}\arrdi{\alpha}\shF_{1}\comma\shL\otimes\shF_{2}\arrdi{-\alpha}\shF_{2}\comma\shF_{1}\otimes\shL\arrdi{\hat{\alpha}}\shF_{1}\comma\shF_{2}\otimes\shL\arrdi{-\hat{\alpha}}\shF_{2}\label{eq: crucial multiplication list-1}\\
\shF_{1}\otimes\shF_{1}\arrdi{\beta}\shF_{2}\comma\shF_{2}\otimes\shF_{2}\arrdi{\beta}\shF_{1}\comma\shF_{1}\otimes\shF_{2}\arrdi{\eta_{1}+\eta_{2}}\odi T\oplus\shL\comma\shF_{2}\otimes\shF_{1}\arrdi{\eta_{1}-\eta_{2}}\odi T\oplus\shL\nonumber 
\end{gather}
where $\hat{\alpha}$ is obtained by $\alpha$ just swapping the factors
in the source and, for future reference, we set $\eta_{1}=(-,-)$
and $\eta_{2}=\la-,-\ra$. We are implicitly assuming that the maps
$\odi T\otimes\alA_{\chi},\alA_{\chi}\otimes\odi T\arr\alA_{\chi}$
are just the usual isomorphisms, or, in other words, that $1\in\odi T$
is the unity for $\alA_{\chi}$.

We want now to give a list of equations involving the maps $\alpha,\beta,\la-,-\ra$,
which we will show are the relationships needed for the associativity
of $\alA_{\chi}$. Such equations will be 'local' relations and therefore
we introduce the following notation:
\begin{notation}
\label{not: local notation for Sthree}When we fix a generator $t$
of $\shL$, the maps $m,\alpha,\beta,\la-,-\ra$ will be thought of
as: $m\in\odi T$, given by $m(t\otimes t)$; $\alpha\colon\shF\arr\shF$,
given by ``$\alpha(u)=\alpha(t\otimes u)$''; $\la-,-\ra\colon\det\shF\arr\odi S$,
given by ``$\la u,v\ra=\la u,v\ra t$''. When we will say that some
particular relation among the maps $m,\alpha,\beta,\la-,-\ra$ locally
holds, this will always mean that such relation holds as soon as basis
$t$ and $y,z$ of, respectively, $\shL$ and $\shF$ are given.
\end{notation}

For instance, the following relation holds locally (use \ref{rem:F is Fdual tensor det F}):
\begin{equation}
(u,v)=\la\alpha(v),u\ra\text{ for }u,v\in\shF\label{eq:ass alpha, gamma first}
\end{equation}
Consider the equations:

\begin{align}
\alpha^{2} & =m\id_{\shF}\label{eq:ass m,alpha}\\
\la\alpha(u),v\ra & =\la\alpha(v),u\ra & \text{ for }u,v\in\shF\label{eq:ass alpha,gamma}\\
\alpha(\beta(u\otimes v)) & =-\beta(u\otimes\alpha(v)) & \text{ for }u,v\in\shF\label{eq:ass alpha,beta}\\
\beta(\beta(u\otimes v)\otimes w) & =\la\alpha(w),v\ra u+\la v,w\ra\alpha(u) & \text{ for }u,v,w\in\shF\label{eq:ass gamma,beta}\\
\la u,\beta(v\otimes w)\ra & =\la w,\beta(u\otimes v)\ra & \text{ for }u,v,w\in\shF\label{eq:ass beta,beta}
\end{align}

\begin{thm}
\label{thm:global data for Sthree}The map of stacks   \[   \begin{tikzpicture}[xscale=3.7,yscale=-0.6]     \node (A0_0) at (0, 0) {$\stY$};     \node (A0_1) at (1, 0) {$\LRings^G_\stR$};     \node (A1_0) at (0, 1) {$\chi=(\shL,\shF,\alpha,\beta, \la-,-\ra)$};     \node (A1_1) at (1, 1) {$\alA_\chi$};     \path (A0_0) edge [->]node [auto] {$\scriptstyle{}$} (A0_1);     \path (A1_0) edge [|->,gray]node [auto] {$\scriptstyle{}$} (A1_1);   \end{tikzpicture}   \] is
well defined, fully faithful and induces an equivalence between the
substack of $\stY$ of objects that locally satisfy the relations
(\ref{eq:ass m,alpha}), (\ref{eq:ass alpha,gamma}), (\ref{eq:ass alpha,beta}),
(\ref{eq:ass gamma,beta}), (\ref{eq:ass beta,beta}) and $\GCov$
(where a cover is thought of as its corresponding sheaf of algebras).
\end{thm}

\begin{notation}
Assuming Theorem \ref{thm:global data for Sthree}, we will think
of $\GCov$ as substack of $\stY$, using for instance expressions
like $\chi=(\shL,\shF,\alpha,\beta,\la-,-\ra)\in\GCov$.
\end{notation}

The aim of this section is to prove Theorem \ref{thm:global data for Sthree}.
We will often use results and notations from \cite{Tonini2014} and
\cite{Tonini2015}.

Denote by $\stX$ the stack whose fibers over a scheme $T$ is the
groupoid of pseudo-monoidal (see \cite[Def 2.21]{Tonini2014}) and
$\shR$-linear functors $\Gamma\colon\Loc^{G}\shR\arr\Loc T$ such
that $\rk\Gamma_{U}=\rk U$ for all $U\in\Loc^{G}\shR$, $\Gamma_{\shR}=\odi T$
and $1\in\Gamma_{\shR}$ is a unity. We are going to embed $\stY$
into $\stX$.

We use results and notations from \ref{subsec:representation theory}.
By \cite[Lemma 1.9]{Tonini2015} and \cite[Rmk 1.17]{Tonini2015}
there are isomorphisms 
\[
\bigoplus_{W\in I_{G}}\Hom^{G}(W,U)\otimes W\to U\text{ and }\bigoplus_{W\in I_{G}}\Hom^{G}(W,U)\otimes\Gamma_{W}\to\Gamma_{U}
\]
natural for $U\in\Loc^{G}\shR$ and for a $\shR$-linear functor $\Gamma\colon\Loc^{G}\shR\arr\Loc T$.
It follows that $\Gamma$ is completely determined by the collection
of locally free sheaves $(\Gamma_{W})_{W\in I_{G}}$. Moreover a pseudo-monoidal
structure on $\Gamma$ corresponds to a sequence of maps 
\[
\Gamma_{W}\otimes\Gamma_{W'}\to\Gamma_{W\otimes W'}\arrdi{\simeq}\bigoplus_{Z\in I_{G}}\Hom^{G}(Z,W\otimes W')\otimes\Gamma_{Z}\text{ for }W,W'\in I_{G}
\]
 Since in \ref{subsec:representation theory} we fixed basis for the
modules $\Hom^{G}(Z,W\otimes W')$ as above, it follows that an object
$\Gamma\in\stX(T)$ can be represented by a sequence $(\shL,\shF,m,\alpha,\hat{\alpha},\eta_{1},\eta_{2},\beta)$
where 
\[
\shL(=\Gamma_{A})\comma\shF(=\Gamma_{V})
\]
are an invertible sheaf and a rank $2$ sheaf on $T$ respectively
and 
\[
\shL\otimes\shL\arrdi m\odi T\comma\shL\otimes\shF\arrdi{\alpha}\shF\comma\shF\otimes\shL\arrdi{\hat{\alpha}}\shF\comma\shF\otimes\shF\arrdi{\eta_{1}\oplus\eta_{2}\oplus\beta}\odi T\oplus\shL\oplus\shF
\]
are maps. In particular $\stY$ can be embedded in $\stX$ by sending
$\chi=(\shL,\shF,\alpha,\beta,\la-,-\ra)\in\stY$ to the sequence
$(\shL,\shF,m_{\chi},\alpha,\hat{\alpha},(-,-)_{\chi},\la-,-\ra,\beta)$,
where $\hat{\alpha}$ is obtained from $\alpha$ exchanging the factors
in the source.

By \cite[Thm A]{Tonini2015}, \cite[Rmk 1.17]{Tonini2015} and \cite[Thm 8.6]{Tonini2014}
there is a fully faithful functor 
\[
\alB_{*}\colon\stX\to\LRings_{\shR}^{G}\comma\alB_{\Gamma}=\bigoplus_{W\in I_{G}}W^{\vee}\otimes\Gamma_{W}
\]
which restricts to an equivalence between the substack of $\stX$
of monoidal functors and $\GCov$. In particular a $G$-cover is (the
spectrum of) a $\alB_{\Gamma}$ (for some $\Gamma\in\stX$) which
is commutative and associative. Here the multiplication of $\alB_{\Gamma}$
is induced by the pseudo-monoidal structure on $\Gamma$.

Let us assume that $\Gamma$ is the functor associated with 
\[
\chi=(\shL,\shF,m,\alpha,\hat{\alpha},\eta_{1},\eta_{2},\beta)\in\stX(T)
\]
In this case we simply write $\alB_{\chi}=\alB_{\Gamma}$. The choice
of basis for the modules in $I_{G}$ in \ref{subsec:representation theory}
defines an isomorphism 
\[
\alA=\odi T\oplus\shL\oplus(\shF_{1}\oplus\shF_{2})\to(\duale{\shR}\otimes\odi T)\oplus(\duale A\otimes\shL)\oplus(\duale V\otimes\shF)=\alB_{\chi}
\]
\[
1\oplus x\oplus f_{1}\oplus f_{2}\mapsto(1^{*}\otimes1)\oplus(1_{A}^{*}\otimes x)\oplus[(v_{2}^{*}\otimes f_{1})\oplus(v_{1}^{*}\otimes f_{2})]
\]

\begin{lem}
The $G$-comodule structure and the multiplication induced on $\alA$
by $\alB_{\chi}$ are the ones described in \ref{prop:Achi is G-equivariant}
and (\ref{eq: crucial multiplication list-1}) respectively. In particular
$(\alB_{*})_{|\stY}$ is the map $\alA_{*}$ of Theorem \ref{thm:global data for Sthree},
which is well defined and fully faithful.
\end{lem}

\begin{proof}
The claim about the $G$-comodule structure is clear, so we just have
to translate the multiplication. This is possible using properties
$1)$ to $4)$ of \cite[Rmk 1.17]{Tonini2015}. We are going to discuss
in details only one case, while for the other ones we will just present
the relevant computations.

The sheaf $\odi T\oplus\shL$ is a $\Z/2\Z$-cover and we claim that
the induced multiplication $\shL\otimes\shL\to\odi T$ is just $m$.
By \cite[Rmk 1.17]{Tonini2015}, in particular point $4)$, there
is a commutative diagram    \[   \begin{tikzpicture}[xscale=4.1,yscale=-1.2]     \node (A0_0) at (0, 0) {$A^\vee\otimes\Gamma_A\otimes A^\vee\otimes\Gamma_A$};     \node (A0_1) at (1, 0) {$\alB_\chi\otimes \alB_\chi$};     \node (A0_2) at (2, 0) {$\alB_\chi$};     \node (A1_0) at (0, 1) {$(A\otimes A)^\vee \otimes \Gamma_A\otimes\Gamma_A$};     \node (A1_1) at (1, 1) {$(A\otimes A)^\vee \otimes \Gamma_{A\otimes A}$};     \node (A1_2) at (2, 1) {$\shR^\vee \otimes \Gamma_\shR$};     \path (A0_0) edge [->]node [auto] {$\scriptstyle{}$} (A0_1);     \path (A0_1) edge [->]node [auto] {$\scriptstyle{}$} (A0_2);     \path (A1_0) edge [->]node [auto] {$\scriptstyle{}$} (A1_1);     \path (A1_1) edge [->]node [auto] {$\scriptstyle{(\xi^\vee)^{-1}\otimes \Gamma_\xi}$} (A1_2);     \path (A0_0) edge [->]node [auto] {$\scriptstyle{}$} (A1_0);     \path (A1_2) edge [->]node [auto] {$\scriptstyle{}$} (A0_2);   \end{tikzpicture}   \] 
where $\xi\colon A\otimes A\to\shR$ is \emph{any }$G$-equivariant
isomorphism. On the other hand the composition $\shL\otimes\shL=\Gamma_{A}\otimes\Gamma_{A}\to\Gamma_{A\otimes A}\arrdi{\Gamma_{\xi}}\Gamma_{\shR}=\odi T$
is $m$ if and only if $\xi$ is the isomorphism chosen in section
\ref{subsec:representation theory}, i.e. such that $\xi(1_{A}\otimes1_{A})=1$,
because this is the way we defined the correspondence $\chi\leftrightarrow\Gamma$.
Since, in this case, $(\xi^{\vee})^{-1}((1_{A}\otimes1_{A})^{*})=1^{*}$,
by diagram chasing we see that the multiplication of $\alB_{\chi}$
maps
\[
(1_{A}^{*}\otimes x)\otimes(1_{A}^{*}\otimes y)\longmapsto1^{*}\otimes m(x\otimes y)\text{ for }x,y\in\shL=\Gamma_{A}
\]
This shows that the multiplication $\alA\otimes\alA\to\alA$ induced
by $\alB_{\chi}$ restricts to $m\colon\shL\otimes\shL\to\odi T$
as claimed.

Since $(A\otimes V)^{\vee}\to V^{\vee}$ maps $(1_{A}\otimes v_{1})^{*}\mapsto-v_{1}^{*}$,
($1_{A}\otimes v_{2})^{*}\mapsto v_{2}^{*}$ the induced map $\shL\otimes(\shF_{1}\oplus\shF_{2})\to\shF_{1}\oplus\shF_{2}$
splits into $\alpha$ and $-\alpha$. Similarly, the maps $(\shF_{1}\oplus\shF_{2})\otimes\shL\to\shF_{1}\oplus\shF_{2}$
splits as $\hat{\alpha}$ and $-\hat{\alpha}$. Finally one has
\[
(V\otimes V)^{\vee}\to\shR^{\vee}\oplus A^{\vee}\oplus V^{\vee}\comma v_{11}^{*}\mapsto v_{2}^{*}\comma v_{22}^{*}\mapsto v_{1}^{*}\comma v_{12}^{*}\mapsto1^{*}+1_{A}^{*}\comma v_{21}^{*}\mapsto1^{*}-1_{A}^{*}
\]
which implies that the induced multiplication on $\alA$ is given
by 
\[
\shF_{1}\otimes\shF_{1}\arrdi{\beta}\shF_{2}\comma\shF_{2}\otimes\shF_{2}\arrdi{\beta}\shF_{1}\comma\shF_{1}\otimes\shF_{2}\arrdi{\eta_{1}+\eta_{2}}\odi T\oplus\shL\comma\shF_{2}\otimes\shF_{1}\arrdi{\eta_{1}-\eta_{2}}\odi T\oplus\shL
\]
\end{proof}
In order to prove Theorem \ref{thm:global data for Sthree} we have
to show that $\alA$ is (the algebra of) a $G$-cover if and only
if $\chi\in\stY$ and it satisfies the properties listed in Theorem
\ref{thm:global data for Sthree}. Thus we have to translate commutativity
and associativity conditions of $\alA$.

\subsubsection{\textbf{Commutativity conditions. \label{subsec:comm conditions}}}

We claim that $\alA$ is commutative if and only if: $\beta\colon\shF\otimes\shF\to\shF$
is symmetric, $\hat{\alpha}\colon\shF\otimes\shL\arr\shF$ is obtained
from $\alpha\colon\shL\otimes\shF\to\shF$ swapping factors in the
source, $\eta_{1}$ is symmetric, $\eta_{2}$ is antisymmetric. Indeed
by \cite[Prop 2.25 and Prop 2.26]{Tonini2014} $\alA$ is commutative
if and only if $\Gamma$ is symmetric. Moreover this is equivalent
to: for $U,U'\in I_{G}$ the maps $\Gamma_{U}\otimes\Gamma_{U'}\to\Gamma_{U\otimes U'}$
and $\Gamma_{U'}\otimes\Gamma_{U}\to\Gamma_{U'\otimes U}\arrdi{\Gamma_{\text{swap}}}\Gamma_{U\otimes U'}$
differ by a swap of factors in the source. For $U=U'=A$ we see that
$m\colon\shL\otimes\shL\to\odi T$ is automatically symmetric because
$\shL$ has rank $1$. For $U=A$, $U'=V$ (or the converse) we obtain
the relation between $\alpha$ and $\hat{\alpha}$. For $U=U'=V$,
notice that the swap map on $V\otimes V\simeq\shR\oplus A\oplus V$
is the identity on $\shR$ and $V$ and minus the identity on $A$.
This translates in the symmetry of $\eta_{1}$ and $\beta$ and the
antisymmetry of $\eta_{2}$.

\subsubsection{\textbf{Associativity conditions.\label{subsec:ass conditions}}}

Let us assume that $\alA$ is commutative. Moreover we use the notation
$\eta_{1}=(-,-)$, $\eta_{2}=\langle-,-\rangle,$$\gamma=\eta_{1}+\eta_{2}$
and $\gamma'=\eta_{1}-\eta_{2}$. We now express some diagrams that
have to commute if $\alA$ is associative. We use the notation introduced
in \ref{not: local notation for Sthree}.
\begin{itemize}
\item  \begin{align} \label{Diagram A}
   \begin{tikzpicture}[xscale=2.8,yscale=-1.0,baseline=(current  bounding  box.center)]     \node (A0_0) at (0, 0) {$\shL\otimes\shL\otimes \shF_1$};     \node (A0_1) at (1, 0) {$\odi{S}\otimes \shF_1$};     \node (A1_0) at (0, 1) {$\shL\otimes \shF_1$};     \node (A1_1) at (1, 1) {$\shF_1$};     \path (A0_0) edge [->]node [auto] {$\scriptstyle{m\otimes\id}$} (A0_1);     \path (A0_0) edge [->]node [auto] {$\scriptstyle{\id\otimes\alpha}$} (A1_0);     \path (A0_1) edge [->]node [auto] {$\scriptstyle{\id}$} (A1_1);     \path (A1_0) edge [->]node [auto] {$\scriptstyle{\alpha}$} (A1_1);   \end{tikzpicture}   \end{align} Locally we obtain the condition (\ref{eq:ass m,alpha}).
\item   \begin{align} \label{Diagram B} 
\begin{tikzpicture}[xscale=3.2,yscale=-0.7,baseline=(current  bounding  box.center)]   
\node (A0_0) at (0, 0) {$\shF_1 \otimes \shF_2 \otimes \shL$};     
\node (A0_1) at (1, 0) {$(\odi{S}\oplus\shL)\otimes \shL$};     
\node (A1_1) at (1, 1) {$\shL\otimes\shL\oplus\shL$};     
\node[rotate=-90] (A1_3) at (1, 0.5) {$\simeq$};     
\node (A3_0) at (0, 3) {$\shF_1\otimes \shF_2$};     
\node (A3_1) at (1, 3) {$\odi{S}\oplus\shL$};     
\path (A0_0) edge [->]node [auto] {$\scriptstyle{\gamma\otimes \id}$} (A0_1);     \path (A0_0) edge [->]node [auto] {$\scriptstyle{\id\otimes -\hat \alpha}$} (A3_0);     \path (A3_0) edge [->]node [auto] {$\scriptstyle{\gamma}$} (A3_1);     \path (A1_1) edge [->]node [auto] {$\scriptstyle{m\oplus\id}$} (A3_1);   \end{tikzpicture}
\end{align} The commutativity of this diagram is locally equivalent to $(u,\alpha(v))=-m\la u,v\ra$,
$(u,v)=-\la u,\alpha(v)\ra$ and, assuming (\ref{eq:ass m,alpha}),
to (\ref{eq:ass alpha, gamma first}).
\item   \begin{align} \label{Diagram C}
   \begin{tikzpicture}[xscale=3.2,yscale=-1.2,baseline=(current  bounding  box.center)]     
\node (A0_0) at (0, 0) {$\shF_1\otimes\shF_1\otimes\shL$};     
\node (A0_1) at (1, 0) {$\shF_2\otimes\shL$};     
\node (A1_0) at (0, 1) {$\shF_1\otimes\shF_1$};     
\node (A1_1) at (1, 1) {$\shF_2$};     
\path (A0_0) edge [->]node [auto] {$\scriptstyle{\beta\otimes\id}$} (A0_1);     
\path (A0_0) edge [->]node [auto] {$\scriptstyle{\id\otimes\hat \alpha}$} (A1_0);     
\path (A0_1) edge [->]node [auto] {$\scriptstyle{-\hat \alpha}$} (A1_1);     
\path (A1_0) edge [->]node [auto] {$\scriptstyle{\beta}$} (A1_1);   
\end{tikzpicture}   \end{align}The commutativity of this diagram is locally equivalent to (\ref{eq:ass alpha,beta}).
\item   \begin{align} \label{Diagram D}    \begin{tikzpicture}[xscale=3.2,yscale=-0.7,baseline=(current  bounding  box.center)]     
\node (A0_0) at (0, 0) {$\shF_2 \otimes \shF_2\otimes\shF_1$};     
\node (A0_1) at (1, 0) {$\shF_1\otimes\shF_1$};     
\node[rotate=-90] (A1_2) at (0, 2.5) {$\simeq$};     
\node (A2_0) at (0, 2) {$\shF_2 \otimes (\odi{S}\oplus\shL)$};     
\node (A3_0) at (0, 3) {$\shF_2 \oplus \shF_2\otimes \shL$};     
\node (A3_1) at (1, 3) {$\shF_2$};     
\path (A0_0) edge [->]node [auto] {$\scriptstyle{\beta\otimes\id}$} (A0_1);     
\path (A3_0) edge [->]node [auto] {$\scriptstyle{\id\oplus(-\hat \alpha)}$} (A3_1);     
\path (A0_1) edge [->]node [auto] {$\scriptstyle{\beta}$} (A3_1);     
\path (A0_0) edge [->]node [auto] {$\scriptstyle{\id\otimes\gamma'}$} (A2_0);   
\end{tikzpicture}   \end{align}The commutativity of this diagram, assuming (\ref{eq:ass alpha, gamma first}),
is locally equivalent to (\ref{eq:ass gamma,beta}).
\item   \begin{align} \label{Diagram E}    \begin{tikzpicture}[xscale=3.2,yscale=-1.2,baseline=(current  bounding  box.center)]     \node (A0_0) at (0, 0) {$\shF_1\otimes\shF_1\otimes\shF_1$};     \node (A0_1) at (1, 0) {$\shF_1\otimes\shF_2$};     \node (A1_0) at (0, 1) {$\shF_2\otimes\shF_1$};     \node (A1_1) at (1, 1) {$\odi{S}\oplus\shL$};     \path (A0_0) edge [->]node [auto] {$\scriptstyle{\id\otimes\beta}$} (A0_1);     \path (A0_0) edge [->]node [auto] {$\scriptstyle{\beta\otimes\id}$} (A1_0);     \path (A0_1) edge [->]node [auto] {$\scriptstyle{\gamma}$} (A1_1);     \path (A1_0) edge [->]node [auto] {$\scriptstyle{\gamma'}$} (A1_1);   \end{tikzpicture}   \end{align} Since
$\gamma'(u\otimes v)=\gamma(v\otimes u)$, the commutativity of this
diagram is locally equivalent to (\ref{eq:ass beta,beta}) and the
analogous one for $(-,-)$, which however follows from (\ref{eq:ass alpha, gamma first}),
(\ref{eq:ass m,alpha}), (\ref{eq:ass alpha,beta}) and (\ref{eq:ass beta,beta}).
Indeed
\[
\begin{alignedat}{1}(u,\beta(v\otimes w)) & =\la\alpha(\beta(v\otimes w)),u\ra=-\la\beta(v\otimes\alpha(w)),u\ra=\la u,\beta(v\otimes\alpha(w))\ra\\
 & =\la\alpha(w),\beta(u\otimes v)\ra=(\beta(u\otimes v),w)=(w,\beta(u\otimes v))
\end{alignedat}
\]
\end{itemize}
\begin{rem}
\label{lem:essential associative conditions}Let $A$ be a commutative
(but not necessary associative) ring and $x,y,z\in A$. If 
\[
(xy)z=x(yz)\text{ and }(yx)z=y(xz)
\]
then all the permutations of $x,y,z$ satisfy associativity. Indeed
\[
y(zx)=(yx)z=x(yz)=(yz)x\comma z(xy)=(yx)z=y(xz)=(zx)y
\]
\[
(zy)x=x(yz)=(xy)z=z(yx)\comma(xz)y=y(zx)=(yz)x=x(zy)
\]
\end{rem}

\begin{proof}
(\emph{of Theorem }\ref{thm:global data for Sthree}) We have to show
that $\alA$ is commutative and associative if and only if $\chi\in\stY$
and it satisfies the properties listed in Theorem \ref{thm:global data for Sthree}.
The ``only if'' part is an easy consequence of the above discussion.
We just highlight some points. Condition (\ref{eq:ass m,alpha}) implies
that $m$ is obtained from $\alpha$ as in (\ref{eq:m from alpha}),
while condition (\ref{eq:ass alpha, gamma first}) implies that $(-,-)$
is obtained from $\la-,-\ra$ as in (\ref{eq:symmetric bracket from alternating one}):
in particular $\chi\in\stY$. Finally the symmetry of $(-,-)$ implies
that equation (\ref{eq:ass alpha,gamma}) holds.

We now focus on the converse. So assume $\chi\in\stY$ and that it
satisfies the properties listed in Theorem \ref{thm:global data for Sthree}.
By (\ref{eq:ass alpha, gamma first}) and (\ref{eq:ass alpha,gamma})
we obtain the symmetry of $(-,-)$ and therefore the commutativity
of $\alA$.

We need to show that $\alA$ is associative. Given $A,B,C\in\{\odi S,\shL,\shF_{1},\shF_{2}\}$
we will say that $(A,B,C)$ holds if $a(bc)=(ab)c$ for all $a\in A,b\in B,c\in C$.
Since $\sigma\in\Z/2\Z$ induces a ring automorphism of $\alA$, if
$(A,B,C)$ holds then $(\sigma(A),\sigma(B),\sigma(C))$ holds. Moreover,
by \ref{lem:essential associative conditions}, if also $(B,A,C)$
holds then all permutations of $(A,B,C)$ and $(\sigma(A),\sigma(B),\sigma(C))$
hold. Recall that $\sigma$ fix $\odi S$ and $\shL$ and exchanges
$\shF_{1}$ and $\shF_{2}$.

Clearly $(\shL,\shL,\shL)$ holds. Condition (\ref{eq:ass m,alpha})
insures that $(\shL,\shL,\shF_{1})$, $(\shL,\shL,\shF_{2})$ and
all their permutations hold. Conditions (\ref{eq:ass alpha, gamma first})
and (\ref{eq:ass m,alpha}) say that all the permutations of $(\shF_{1},\shF_{2},\shL)$
hold, while condition (\ref{eq:ass alpha,beta}) tells us that all
the permutations of $(\shF_{1},\shF_{1},\shL)$ and $(\shF_{2},\shF_{2},\shL)$
hold. The relation (\ref{eq:ass gamma,beta}) implies that $(\shF_{2},\shF_{2},\shF_{1})$,
$(\shF_{1},\shF_{1},\shF_{2})$ and all their permutations hold. Finally
(\ref{eq:ass beta,beta}) says that $(\shF_{1},\shF_{1},\shF_{1})$
and $(\shF_{2},\shF_{2},\shF_{2})$ hold. It is now easy to check
that we have obtained all the possible triples.
\end{proof}
\begin{thm}
\label{thm:description S3 covers by diagrams} The functor $\stY\to\LRings_{\shR}^{G}$
is an equivalence between the substack of $\stY$ of objects making
the following diagrams (\ref{Diagram A}), (\ref{Diagram C}), (\ref{Diagram D}), (\ref{Diagram E})
and   \[   \begin{tikzpicture}[xscale=2.6,yscale=-1.2]     \node (A0_0) at (0, 0) {$\shF\otimes \shL \otimes \shF$};     \node (A0_1) at (1, 0) {$\shF\otimes\shF$};     \node (A1_0) at (0, 1) {$\shF\otimes\shF$};     \node (A1_1) at (1, 1) {$\det \shF$};     \path (A0_0) edge [->]node [auto] {$\scriptstyle{\id\otimes \alpha}$} (A0_1);     \path (A0_0) edge [->]node [auto,swap] {$\scriptstyle{-\alpha\otimes\id}$} (A1_0);     \path (A0_1) edge [->]node [auto] {$\scriptstyle{\la-,-\ra}$} (A1_1);     \path (A1_0) edge [->]node [auto] {$\scriptstyle{\la-,-\ra}$} (A1_1);   \end{tikzpicture}   \] commutative
and $\GCov$.
\end{thm}

\begin{proof}
Let $\chi\in\stY$. First of all notice that the commutativity of
the diagram in the statement is locally equivalent to (\ref{eq:ass alpha,gamma})
and, using (\ref{eq:ass alpha, gamma first}), to the symmetry of
$(-,-)_{\chi}$. By definition of $\stY$ and \ref{subsec:comm conditions},
this is also equivalent to the commutativity of $\alA_{\chi}$. Thus
the claim follows from \ref{thm:global data for Sthree} and \ref{subsec:ass conditions}.
\end{proof}

\subsection{Local analysis.}

Let $\chi=(\shL,\shF,\alpha,\beta,\la-,-\ra)\in\stY$ and assume that
$t\in\shL$ is a generator and that $y,z$ is a basis of $\shF$.
The aim of this subsection is to translate conditions (\ref{eq:ass m,alpha}),
(\ref{eq:ass alpha,gamma}), (\ref{eq:ass alpha,beta}), (\ref{eq:ass gamma,beta})
and (\ref{eq:ass beta,beta}), writing all the maps $\alpha,\beta,\la-,-\ra$
with respect to the given basis. In particular we will use notation
from \ref{not: local notation for Sthree}, so that $m\in\odi T$,
$\alpha$ is a map $\shF\arr\shF$ and $\la-,-\ra\colon\det\shF\arr\odi T$.
\begin{notation}
Write
\[
\beta(y^{2})=ay+bz\comma\beta(yz)=cy+dz\comma\beta(z^{2})=ey+fz\comma\la y,z\ra=\omega\comma\alpha=\left(\begin{array}{cc}
A & B\\
C & D
\end{array}\right)
\]
In particular:
\[
(y,y)=-C\omega\comma(y,z)=-D\omega\comma(z,y)=A\omega\comma(z,z)=B\omega\comma m=(A^{2}+D^{2})/2+BC
\]
\end{notation}

\begin{lem}
\label{lem:local conditions for the map for Sthree}The object $\chi=(\shL,\shF,\alpha,\beta,\la-,-\ra)\in\stY$
belongs to $\RCov G$ if and only if the following relations hold.
\begin{equation}
\begin{array}{ccc}
(\ref{eq:ass m,alpha}) & \iff & (A-D)(A+D)=B(A+D)=C(A+D)=0\\
(\ref{eq:ass alpha,gamma}) & \iff & \omega(A+D)=0\\
(\ref{eq:ass alpha,beta}) & \iff & \left\{ \begin{array}{c}
(2aA+bB+cC)=(2cA+dB+eC)=0\\
C(a+d)+b(A+D)=C(c+f)+d(A+D)=0\\
B(a+d)+c(A+D)=B(c+f)+e(A+D)=0\\
a(A+D)-D(a+d)=c(A+D)-D(c+f)=0
\end{array}\right.\\
(\ref{eq:ass gamma,beta}) & \iff & \left\{ \begin{array}{c}
a^{2}+bc=-\omega C\comma ac+be=\omega(A-D)\comma c^{2}+de=B\omega\\
(a-d)(a+d)=b(a+d)=c(a+d)=0\\
(c-f)(c+f)=d(c+f)=e(c+f)=0\\
a(a+d)+b(c+f)=e(a+d)+c(c+f)=0
\end{array}\right.\\
(\ref{eq:ass beta,beta}) & \iff & \omega(a+d)=\omega(c+f)=0
\end{array}\label{eq:loc com and ass conditions}
\end{equation}
\end{lem}

\begin{proof}
The claims follow from the following relations, which can be computed
directly.
\[
\begin{array}{ll}
\la z,\beta(y^{2})\ra=-a\omega & \la y,\beta(zy)\ra=\la y,\beta(yz)\ra=d\omega\\
\la y,\beta(z^{2})\ra=f\omega & \la z,\beta(yz)\ra=\la z,\beta(zy)\ra=-c\omega
\end{array}
\]
\[
\begin{array}{rcl}
\alpha(\beta(y^{2}))+\beta(y\alpha(y)) & = & (2aA+bB+cC)y+(C(a+d)+b(A+D))z\\
\alpha(\beta(zy))+\beta(z\alpha(y)) & = & (2cA+dB+eC)y+(C(c+f)+d(A+D))z\\
\alpha(\beta(yz))+\beta(y\alpha(z)) & = & (B(a+d)+c(A+D))y+(2dD+bB+cC)z\\
\alpha(\beta(z^{2}))+\beta(z\alpha(z)) & = & (B(c+f)+e(A+D))y+(2fD+eC+dB)z
\end{array}
\]
If we set $\Gamma(u,v,w)=\la\alpha(w),v\ra u+\la v,w\ra\alpha(u)=(v,w)u+\langle v,w\rangle\alpha(u)$
we have
\[
\begin{array}{ll}
\beta(\beta(y^{2})y)=(a^{2}+bc)y+b(a+d)z & \Gamma(y,y,y)=-C\omega y\\
\beta(\beta(y^{2})z)=(ac+be)y+(ad+bf)z & \Gamma(y,y,z)=\omega(A-D)y+\omega Cz\\
\beta(\beta(z^{2})y)=(ea+fc)y+(eb+fd)z & \Gamma(z,z,y)=-B\omega y+\omega(A-D)z\\
\beta(\beta(z^{2})z)=e(c+f)y+(ed+f^{2})z & \Gamma(z,z,z)=B\omega z
\end{array}
\]
\[
\begin{array}{ll}
\beta(\beta(yz)y)=\beta(\beta(zy)y)=c(a+d)y+(cb+d^{2})z & \Gamma(y,z,y)=\Gamma(z,y,y)=-C\omega z\\
\beta(\beta(yz)z)=\beta(\beta(zy)z)=(c^{2}+de)y+d(c+f)z & \Gamma(y,z,z)=\Gamma(z,y,z)=B\omega y
\end{array}
\]
\end{proof}
\begin{notation}
\label{not: associated parameters for chi and Sthree}Given $\chi=(\shL,\shF,m,\alpha,\beta,\la-,-\ra)\in\stY$,
a basis $y,z$ of $\shF$ and a generator $t\in\shL$, we will denote
by 
\[
a_{\chi},b_{\chi},c_{\chi},d_{\chi},e_{\chi},f_{\chi},\omega_{\chi},A_{\chi},B_{\chi},C_{\chi},D_{\chi},m_{\chi}
\]
the data associated with $\chi$ as above. We will always omit the
$-_{\chi}$ if this will not lead to confusion.
\end{notation}

\subsection{From \texorpdfstring{$(\mu_{3}\rtimes\Z/2\Z)$}{(mu3 rtimes Z/2Z)}-covers
to $S_{3}$-covers}

In this section we compare $G$-covers and $S_{3}$-covers. We denote
by $\stR_{3}$ the ring of integers with $6$ inverted, that is $\stR_{3}=\Z[1/6]=\shR[1/3]$.
Theorem \ref{ex:SThree covers and the other group} for $n=3$ applies
to $G$ and $H=\Z/3\Z\rtimes(\Z/3\Z)^{*}$ over $\shR_{3}\supseteq\Z[1/3]$.
We fix the isomorphism $H\to S_{3}$ mapping $(1,1)\mapsto(123)$,
$(0,2)\mapsto(12)$.
\begin{prop}
\label{thm:from G covers to S3} Over $\shR_{3}$ the equivalence
$\GCov\to S_{3}\textup{-Cov}$ of \ref{ex:SThree covers and the other group}
maps a $G$-cover $\alA=\odi{}\oplus\shL\oplus(V^{\vee}\otimes\shF)$
(where $\shF_{1}=v_{2}^{*}\otimes\shF$ and $\shF_{2}=v_{1}^{*}\otimes\shF)$
to the sub algebra
\[
\alC=\odi{}\oplus\shL z\oplus[(v_{2}^{*}x+v_{1}^{*}x^{2})\otimes\shF]\oplus[(v_{2}^{*}zx-v_{1}^{*}zx^{2})\otimes\shF]\subseteq\alA[x,z]/(x^{3}-1,z^{2}+3)
\]
The left $S_{3}$-action on the above algebras is trivial on $\alA$
and satisfies
\[
(123)x=x(z-1)/2,(123)z=z\comma(12)x=x,(12)z=-z
\]
Over $Y=\Spec(\shR_{3}[w]/(w^{2}+3))$ the map $\alC\to\alA,$$x\mapsto1$,
$z\mapsto w$ is an isomorphism equivariant with respect to the isomorphism
$(S_{3})_{Y}\to G_{Y}$, $(123)\mapsto(w-1)/2$, $(12)\mapsto\sigma$.
\end{prop}

\begin{proof}
We are going to rewrite \ref{ex:SThree covers and the other group}
in this simplified situation. We have $\Phi_{3}(y)=1+y+y^{2}$ and,
over $\shR_{3}$,
\[
K=\shR_{3}[y]/(y^{2}+y+1)=\shR_{3}[z]/(z^{2}+3)=\shR_{3}\oplus\shR_{3}z\text{ where }z=1+2y,y=(z-1)/2
\]
Moreover $\mu_{3}^{*}=\spec K$ and the natural involution of $\mu_{3}^{*}$
acts as $y\mapsto y^{2}$ or $z\mapsto-z$. In particular $\alA\otimes A_{P}=\alA[x,z]/(x^{3}-1,z^{2}+3)$
and its $S_{3}$-action is 
\[
(123)\leftrightarrow(1,1)\colon x\mapsto xy=x(z-1)/2,z\mapsto z\comma(12)\leftrightarrow(0,2)\colon x\mapsto x,z\mapsto-z
\]
Instead $\sigma$ acts on $\alA\otimes A_{P}$ with the usual action
on $\alA$, while on $x,z$, since $G\ni\sigma\mapsto(1,2)\in\mu_{3}\rtimes(\Z/3\Z)^{*}$,
as $x\mapsto x^{2},z\mapsto-z$. The $\mu_{3}$-action on $\alA\otimes A_{P}$
is the usual one on $\alA$, while on $A_{P}$ we have $\deg x=-1=2$
and $\deg z=0$. In other words $A_{P}=K\oplus Kx^{2}\oplus Kx$ is
the graduation induced by $\mu_{3}$. We have
\[
(\alA\otimes A_{P})^{\mu_{3}}=[(\odi{}\oplus\shL)\otimes K]\oplus[v_{2}^{*}\otimes\shF\otimes Kx]\oplus[v_{1}^{*}\otimes\shF\otimes Kx^{2}]
\]
In particular $(\alA\otimes A_{P})^{G}=\odi{}\oplus\shL z\oplus W$
where $W$ is the $\sigma$-invariant of $[v_{2}^{*}\otimes\shF\otimes Kx]\oplus[v_{1}^{*}\otimes\shF\otimes Kx^{2}]$.
As $\sigma$ exchanges the two factors we have that $W$ is the sheaf
of elements of the form
\[
v_{2}^{*}\otimes f\otimes\lambda x+\sigma(v_{2}^{*}\otimes f\otimes\lambda x)=v_{2}^{*}\otimes f\otimes\lambda x+v_{1}^{*}\otimes f\otimes\sigma(\lambda)x^{2}\text{ for }f\in\shF,\lambda\in K
\]
 Since $K=\shR_{3}\oplus\shR_{3}z$, for $\lambda=1$ and $\lambda=z$
we obtain
\[
v_{2}^{*}\otimes f\otimes x+v_{1}^{*}\otimes f\otimes x^{2}=(v_{2}^{*}x+v_{1}^{*}x^{2})\otimes f\text{ and }v_{2}^{*}\otimes f\otimes zx+v_{1}^{*}\otimes f\otimes-zx^{2}=(v_{2}^{*}zx-v_{1}^{*}zx^{2})\otimes f
\]
The last statement follows directly from \ref{ex:SThree covers and the other group}.
\end{proof}
\begin{rem}
It is a classical result that $\textup{Et}_{3}\to\Bi S_{3}$, $Q\mapsto\Isosh(\Z/3\Z,Q)$,
where $\textup{Et}_{3}$ is the stack of degree $3$ étale covers
and $\Z/3\Z$ is thought of as a scheme, is an equivalence (see \cite[Prop 2.7]{Tonini2017}).
Thus over $\shR_{3}$ the $S_{3}$-torsor $P=\mu_{3}\times\mu_{3}^{*}$
of \ref{ex:SThree covers and the other group} is induced by a degree
$3$ étale cover over $\shR_{3}$: this is $\mu_{3}\to\Spec\shR_{3}$
itself. Indeed a direct check shows that 
\[
P=\mu_{3}\times\Isosh_{\text{groups}}(\Z/3\Z,\mu_{3})\to\Isosh_{\text{sets}}(\Z/3\Z,\mu_{3})\comma(g,\phi)\mapsto m_{g}\circ\phi
\]
where $m_{g}$ denotes the multiplication by an element of $g\in\mu_{3}$,
is an $S_{3}$-equivariant map of $S_{3}$-torsors, hence an isomorphism.
\end{rem}

\section{Geometry of \texorpdfstring{$(\mu_{3}\rtimes\Z/2\Z)$}{(mu3 rtimes Z/2Z)}-Cov
and $\protect\RCov{S_{3}}$}

We keep the notation from Section \ref{sec:Data for covers}. In particular
$\shR=\Z[1/2]$, $\shR_{3}=\Z[1/6]$ and $G=\mu_{3}\rtimes\Z/2\Z$.
The aim of this section is to describe the geometry of the stacks
$\GCov$ over $\shR$. In particular, over $\shR_{3}$, we obtain
a description of $S_{3}\textup{-Cov}\simeq\GCov$ thanks to \ref{thm:from G covers to S3}.

\subsection{A smooth atlas for \texorpdfstring{$(\mu_{3}\rtimes\Z/2\Z)$}{(mu3 rtimes Z/2Z)}-Cov}

Set
\[
P=\shR[a,b,c,d,e,f,A,B,C,D,\omega]/(\text{relations \ref{lem:local conditions for the map for Sthree}})
\]
By \ref{lem:local conditions for the map for Sthree} the scheme $\Spec P$
represents the functor $(\Sch/\shR)^{\op}\to\set$ which with any
$\shR$-scheme associates the setoid of $G$-covers $\chi=(\shL,\shF,m,\alpha,\beta,\la-,-\ra)$
with a given basis for $\shL$ and $\shF$. In other words $\Spec P$
is isomorphic to the fiber product (over $\shR$) of $(\shL,\shF)\colon\GCov\to(\Bi\Gm\times\Bi\GL_{2})$
and the trivial torsor $\Spec\shR\to\Bi\Gm\times\Bi\GL_{2}$. The
map $\Spec P\to\GCov$ corresponds to 
\[
(P,P^{2},\alpha,\beta,\langle-,-\rangle)\text{ where }\alpha=\left(\begin{array}{cc}
A & B\\
C & D
\end{array}\right)\comma\beta=\left(\begin{array}{ccc}
a & c & e\\
b & d & f
\end{array}\right)\comma\la-,-\ra=\omega
\]
Here we are using the canonical basis of $P$ and $P^{2}$. In particular
\begin{thm}
\label{thm:GCov as quotient stack} The map $\Spec P\to\GCov$ is
a $(\Gm\times\GL_{2})$-torsor and 
\[
\GCov\simeq[\Spec P/(\Gm\times\GL_{2})]
\]
\end{thm}

We now want to discuss the ring $P$ more in details and, in particular,
prove the following result:
\begin{thm}
\label{thm:main algebra main structure} The ring $P$ is a flat $\shR$-algebra,
it has two minimal primes 
\[
Q_{1}=(a+c,d+f,A+D)\text{ and }Q_{2}=(a,b,c,d,e,f,B,C,A-D,\omega)
\]
$\Spec(P/Q_{1})\to\Spec\shR\text{ is flat and geometrically integral and }\Spec(P/Q_{2})\simeq\A_{\shR}^{1}$.

Moreover $Q_{1}+Q_{2}=(a,b,c,d,e,f,A,B,C,D,\omega)$, $Q_{1}\cap Q_{2}=(a+c,d+f)=\sqrt{0_{P}}$
and $\Spec(P)^{\textup{red}}\to\Spec\shR$ is flat and geometrically
reduced.
\end{thm}

\begin{defn}
\label{def:the zero section} We denote by $\{0\}$ the closed substack
of $\chi\in\GCov$ such that $\alpha_{\chi}=\beta_{\chi}=\la-,-\ra_{\chi}=0$.
\end{defn}

\begin{rem}
The inclusion $\{0\}\to\GCov$, over the atlas $\Spec P\to\GCov$,
corresponds to the map $\Spec\shR\to\Spec P$ mapping all variables
to $0$. In particular it follows that $\{0\}\simeq\Bi\Gm\times\Bi\GL_{2}$.
\end{rem}

A direct consequence of \ref{thm:GCov as quotient stack} and \ref{thm:main algebra main structure}
is the following.
\begin{thm}
\label{thm:geometry of GCov and components} The stack $\GCov$ is
a flat and finitely generated algebraic stack over $\shR$, the closed
substacks (see \ref{def:locally multiple identity})
\[
\stZ_{G}=\{\chi\st\tr\beta_{\chi}=\tr\alpha_{\chi}=0\}\comma\stZ_{2}=\{\chi\st\beta_{\chi}=\la-,-\ra_{\chi}=0,\alpha_{\chi}=((\tr\alpha_{\chi})/2)\otimes\id_{\shF_{\chi}}\}
\]
 are the irreducible components of $\GCov$ with their reduced structure
and they are flat and geometrically integral over $\shR$. Moreover
$\stZ_{G}\cap\stZ_{2}=\{0\}$, $(\GCov)^{\text{red}}=\{\chi\st\tr\beta_{\chi}=0\}$
and it is flat and geometrically reduced over $\shR$. Finally we
have a decomposition into open substacks
\begin{equation}
(\GCov-\{0\})=(\stZ_{G}-\{0\})\sqcup(\stZ_{2}-\{0\})\label{eq:decomposition of GCov minus zero}
\end{equation}
\end{thm}

We collect some partial results before proving Theorem \ref{thm:main algebra main structure}.
\begin{lem}
\label{rem: cover with non zero trace} We have $(a+d)^{3}=(c+f)^{3}=0$
in $P$. On the other hand $a+d$ and $c+f$ are not zero in any geometric
fiber of $\Spec P\to\Spec\shR$.
\end{lem}

\begin{proof}
Indeed, by (\ref{eq:loc com and ass conditions}), we have $a(a+d)=d(a+d$)
and 
\[
0=[a(a+d)+b(c+f)](a+d)=a(a+d)^{2}=d(a+d)^{2}\then(a+d)^{3}=0
\]
The expression $(c+f)^{3}=0$ is proved similarly.

For the second claim, if $k$ is an algebraically closed field over
$\shR$, the choice

\[
a=b=c=d=e=f=\omega=A=B=C=D=x
\]
defines a map $P\to k[x]/(x^{2})$ sending $a+d$ and $c+f$ to $x$.
\end{proof}
\begin{lem}
\label{lem:result about second prime} The map $P\to P_{A+D}$ has
kernel $Q_{2}$ of Theorem \ref{thm:main algebra main structure}
and factors as 
\[
P\twoheadrightarrow\shR[A]\hookrightarrow\shR[A]_{A}\comma a,b,c,d,e,f,\omega,B,C\mapsto0,D\mapsto A,A\mapsto A
\]
\end{lem}

\begin{proof}
Clearly $P/Q_{2}\simeq\shR[A]$, so it is enough to show that $a,b,c,d,e,f,\omega,B,C,A-D$
are zero in $P_{A+D}$. We makes use of (\ref{eq:loc com and ass conditions}).
We immediately have $B=C=\omega=0$ and $A=D$, which is therefore
invertible in $P_{A+D}$. From the local equations corresponding to
(\ref{eq:ass alpha,beta}) we first obtain $b=c=d=e=0$, then $2aA=0$
and therefore, as $A$ is invertible, $a=0$. Finally the last equation
yields $Df=Af=0$ and thus $f=0$.
\end{proof}
We now focus on the first prime ideal $Q_{1}$.
\begin{lem}
\label{lem: primaty lemma 1} Let $D$ be a ring and $a,b_{1},\dots,b_{m}$
be a regular sequence in $D$. Then $a$ is a non zero divisor in
$D[X_{1},\dots,X_{m}]/(aX_{i}-b_{i}\st i\leq m)$.
\end{lem}

\begin{proof}
Denote by $S_{m}$ the last ring in the statement. We proceed by induction
on $m$.

\emph{Base case $m=1$. }So let $f\in D[X_{1}]$ such that $af=0$
in $S_{1}$, so that
\[
af(X_{1})=g(X_{1})(aX_{1}-b_{1})\text{ in }D[X_{1}]\text{ for some }g\in D[X_{1}]
\]
It follows that $b_{1}g(X_{1})=0$ in $(D/(a))[X_{1}]$ and, by hypothesis,
that $a\mid g(X_{1})$ in $D[X_{1}]$, say $g=a\tilde{g}.$ As $a$
is a non zero divisor in $D[X_{1}]$ we can conclude that $f=\tilde{g}(aX_{1}-b_{1})$,
that is $f=0$ in $S_{1}$.

\emph{Inductive step $m\then m+1$.} From the base case it is enough
to show that $a,b_{m+1}$ is a regular sequence in $S_{m}$. From
the case $m$ we know that $a$ is a non zero divisor in $S_{m}$.
On the other hand
\[
S_{m}/(a)=D[X_{1},\dots,X_{m}]/(a,b_{1},\dots,b_{m})=(D/(a,b_{1},\dots,b_{m}))[X_{1},\dots,X_{m}]
\]
which shows that $b_{m+1}$ is a non zero divisor on $S_{m}/(a)$.
\end{proof}
\begin{rem}
\label{rem:graded kernel} Let $D$ be a ring and $b_{1},\dots,b_{m}\in D$.
Then the kernel of the map
\[
D[X_{1},\dots,X_{m}]\to D[t]\comma X_{i}\longmapsto tb_{i}
\]
is the graded ideal generated by the homogeneous polynomials $f$
such that $f(b_{1},\dots,b_{m})=0$.
\end{rem}

\begin{lem}
\label{lem: primaty lemma 2}Let $D$ be a ring, $b_{1},\dots,b_{m}\in D$
and denote by $J$ the graded ideal of $D[X_{1},\dots X_{m}]$ generated
by the homogeneous polynomials $f$ such that $f(b_{1},\dots,b_{m})=0$.
Consider also the ring
\[
S=\frac{D[X_{1},\dots,X_{m},\omega]}{(J,\omega X_{i}-b_{i})}
\]
Then $S_{\omega}=D[\omega]_{\omega}$ and the map
\[
S\to D[\omega]_{\omega}\comma X_{i}\longmapsto b_{i}/\omega
\]
is well defined and injective.
\end{lem}

\begin{proof}
The ring $S_{\omega}$ is the quotient of $D[\omega]_{\omega}$ by
the ideal generated by $f(\underline{b}/\omega)$ for $f\in J$, where
$\underline{b}=(b_{1},\dots,b_{m})$. Since $f(\underline{b}/\omega)=f(\underline{b})/\omega^{d}=0$
if $f\in D[X_{1},\dots,X_{m}]$ is an homogeneous polynomial of degree
$d$ and $f\in J$, we can conclude that $S_{\omega}=D[\omega]_{\omega}$.
Consider the map
\[
\phi\colon D[X_{1},\dots,X_{m},\omega]/J\to D[\omega,t]\comma X_{i}\longmapsto tb_{i}
\]
It is injective by \ref{rem:graded kernel}. We have to prove that
the kernel $\shK$ of the map
\[
D[X_{1},\dots,X_{m},\omega]/J\to D[\omega,t]\to D[\omega,t]/(\omega t-1)\simeq D[\omega]_{\omega}\comma X_{i}\longmapsto b_{i}/\omega
\]
is generated by the $\omega X_{i}-b_{i}$. Thus let $f\in\shK$, which
just means $\phi(f)\in(\omega t-1)$. Since $\phi(\omega X_{i}-b_{i})=0$
we can assume
\[
f=\sum_{\alpha\in\N^{m}}f_{\alpha}\underline{X}^{\alpha}\text{ with }f_{0}\in D[\omega]\text{ and }f_{\alpha}\in D\text{ if }|\alpha|>0
\]
By assumption
\[
\phi(f)=\sum_{\alpha\in\N^{m}}f_{\alpha}t^{|\alpha|}\underline{b}^{\alpha}=g(\omega,t)(\omega t-1)\then\sum_{\alpha\st|\alpha|=l}f_{\alpha}\underline{b}^{\alpha}=g_{l-1}\omega-g_{l}\text{ for }l\geq0
\]
where $g=\sum_{l\in\Z}g_{l}t^{l}\in D[\omega][t]$, that is with $g_{l}=0$
if $l<0$ or $l\gg0$. We claim that $g=0$ and, therefore, $f=0$.
Indeed, otherwise, if $q$ is the degree of $g$ with respect to $t$,
we would have
\[
g_{q}\omega=\sum_{\alpha\st|\alpha|=q+1}f_{\alpha}\underline{b}^{\alpha}\in D\then g_{q}=0
\]
\end{proof}
\begin{prop}
\label{prop:ZG result on atlas} Let $S=\shR[a,b,c,e,A,B,C,\omega]/I$
where $I$ is the ideal
\[
I=(2aA+bB+cC,2cA-aB+eC,2\omega A-(ac+be),\omega B-(c^{2}-ae),\omega C+a^{2}+bc)
\]
Then $\Spec S\to\Spec\shR$ is flat and geometrically integral.
\end{prop}

\begin{proof}
Flatness follows from geometric integrality. Indeed this is equivalent
to the torsion freeness of $S$, which, in case $S$ is a domain,
coincides with the injectivity of $\shR\to S$. This follows from
the existence of the morphism $S\to\shR$ mapping all variables to
$0$.

We now have to prove that, if we replace $\shR$ by any domain $R$,
the ring $S$ is a domain. Consider $b_{1}=(ac+be)/2$, $b_{2}=(c^{2}-ae)$
and $b_{3}=-a^{2}-bc$ and $D=R[a,b,c,e]$. By \ref{rem:graded kernel}
and \ref{lem: primaty lemma 2} we have to prove that, if
\[
J=(2aX_{1}+bX_{2}+cX_{3},2cX_{1}-aX_{2}+eX_{3})
\]
then $\phi\colon P=D[X_{1},X_{2},X_{3}]/J\to D[t]$, $X_{i}\mapsto tb_{i}$
is well defined and injective. It is well defined because
\[
\det\left(\begin{array}{ccc}
Y_{1} & Y_{2} & Y_{3}\\
2a & b & c\\
2c & -a & e
\end{array}\right)=2(Y_{1}b_{1}+Y_{2}b_{2}+Y_{3}b_{3})
\]
Thus we focus on the injectivity. Over $D_{e}$ we have
\[
P_{e}=D_{e}[X_{1},X_{2},X_{3}]/J=D_{e}[X_{1},X_{2}]/(b_{1}X_{2}-b_{2}X_{1})
\]
It follows that $\phi$ is an isomorphism over $D_{eb_{1}}$. Since
$eb_{1}\neq0$, the claim holds if we show that $P$ is a domain.
We apply \ref{lem: primaty lemma 1} twice. We have that
\[
P_{1}=D[a,b,c,X_{1},X_{2}][X_{3}]/(cX_{3}-(-2aX_{1}-bX_{2}))
\]
is a subdomain of $(P_{1})_{c}=D[a,b,c,X_{1},X_{2}]_{c}$ because
$c,2aX_{1}+bX_{2}$ is a regular sequence in $D[a,b,c,X_{1},X_{2}]$.
Now consider
\[
P=P_{1}[e]/(X_{3}e-(aX_{2}-2cX_{1}))
\]
We prove that $f=aX_{2}-2cX_{1}$ is a non zero divisor in $P_{1}/(X_{3})$,
so that $P$ would be a subdomain of $P_{X_{3}}=(P_{1})_{X_{3}}=D[a,b,X_{1},X_{2},X_{3}]_{X_{3}}$.
Set $g=2aX_{1}+bX_{2}$. The map
\[
P_{1}/(X_{3})=(D[a,b,X_{1},X_{2}]/(g))[c]\to(P_{1})_{c}/(X_{3})=(P_{1})_{c}/(g)
\]
is injective, thus we have to exclude that $f$ is a zero divisor
in the bigger ring. Consider the change of variable 
\[
\overline{X}_{1}=-f/(2c)=X_{1}-(a/(2c))X_{2}\then g=2a\overline{X}_{1}-b_{3}X_{2}\comma(P_{1})_{c}=D[a,b,c,\overline{X}_{1},X_{2}]_{c}
\]
If $f$, and so $\overline{X}_{1}$, is a zero divisor in $(P_{1})_{c}/(g)$
we have
\[
\overline{X}_{1}p=gq\text{ in }(P_{1})_{c}\then gq=b_{3}X_{2}q=0\text{ in }(P_{1})_{c}/(\overline{X}_{1})=D[a,b,c,X_{2}]_{c}\then q=\tilde{q}\overline{X}_{1}
\]
and thus $p=\tilde{q}g$, that is $p=0$ in $(P_{1})_{c}/(g)$ .
\end{proof}
\begin{proof}
[Proof of Theorem \ref{thm:main algebra main structure}] From \ref{lem:result about second prime}
we know that $Q_{2}$ is a prime ideal, $P/Q_{2}\simeq\shR[A]$, $\Spec(P/Q_{2})$
is the schematic closure of the open immersion $\Spec P_{A+D}\to\Spec P$
and therefore that $Q_{2}$ is a minimal prime. From \ref{prop:ZG result on atlas},
since $S/I=P/Q_{1}$, we deduce that $Q_{1}$ is a prime and $\Spec(P/Q_{1})\to\Spec\shR$
is flat and geometrically integral. Since $Q_{1}$ goes to zero along
the map $P\to P_{\omega}$, it has to be the kernel and therefore
a minimal prime.

From \ref{rem: cover with non zero trace} we have $(a+d,c+f)\subseteq\sqrt{0_{P}}\subseteq Q_{1}\cap Q_{2}$.
In order to check the equality, we can work over $P'=P/(a+d,c+f)$.
Denote by $Q_{1}'$ $Q_{2}'$ the images of $Q_{1}$, $Q_{2}$ respectively.
In particular $Q_{1}'=(A+D)$ and $Q_{1}'Q_{2}'=0$. If $u\in Q_{1}'\cap Q_{2}'$
in $P'$, we have $u=(A+D)v\in Q_{2}'$. Since $A+D\notin Q_{2}'$
it follows that $v\in Q_{2}'$. Hence $u\in Q_{1}'Q_{2}'=0$. In conclusion
$Q_{1}\cap Q_{2}=(a+d,c+f)=\sqrt{0_{P}}$ and $Q_{1}$, $Q_{2}$ are
the only minimal primes of $P$. Moreover, since $Q_{1}+Q_{2}$ is
the kernel of the ``zero map'' $P\to\shR$, there is an exact sequence
of $\shR$-modules
\[
0\to P/(Q_{1}\cap Q_{2})\to P/Q_{1}\times P/Q_{2}\to(P/Q_{1}+Q_{2})\simeq\shR\to0
\]
It follows that $Q_{1}\cap Q_{2}$ commutes with arbitrary base changes
of $\shR$ and, since $\Spec P/Q_{1}$ and $\Spec P/Q_{2}$ are geometrically
integral over $\shR$, we can conclude that $\Spec(P/(Q_{1}\cap Q_{2}))=(\Spec P)^{\textup{red}}\to\Spec\shR$
is geometrically reduced. From the above sequence, since $P/Q_{1}$
and $P/Q_{2}$ are $\shR$-flat, we can also conclude that $P/(Q_{1}\cap Q_{2})$
is $\shR$-flat. This also implies that, if $x\in P$ is a non-zero
divisor as $\shR$-module, then $x\in Q_{1}\cap Q_{2}=\sqrt{0_{P}}$,
that is $x$ is nilpotent in $P$. Applying the ``zero map'' $P\to\shR$
it follows that $x$ is nilpotent in $\shR$ and thus $x=0$. In conclusion
$P$ is $\shR$-torsion free and therefore $\shR$-flat.
\end{proof}
\begin{lem}
\label{lem:The Proj action} The action $\Spec P\times\Gm\to\Spec P$
obtained restricting the $\GL_{2}\times\Gm$-action on $P$ via $j\colon\Gm\to\GL_{2}\times\Gm$,
$j(\lambda)=(\lambda\id,\lambda)$ is given by $(\chi,\lambda)\mapsto\lambda^{-1}\chi$.
\end{lem}

\begin{proof}
We think of $F=\Spec P$ as the functor $(\Sch/\shR)^{\op}\to\sets$
such that $F(T)$ is the set of sequences $\chi=(\odi T,\odi T^{2},\alpha,\beta,\la-,-\ra)$
belonging to $\GCov(T)$. If $\chi=(\odi T,\odi T^{2},\alpha,\beta,\la-,-\ra)\in F(T)$
and $(M,\lambda)\in\GL_{2}\times\Gm(T)$ then $\chi\cdot(M,\lambda)=(\odi T,\odi T^{2},\alpha',\beta',\la-,-\ra')\in F(T)$
is the unique element such that $(M,\lambda)$ defines a morphism
$\chi\to\chi\cdot(M,\lambda)$ in $\GCov(T)$. In other words we must
have: $M\alpha=\alpha'(\lambda\otimes M)$, $M\beta=\beta'\Sym^{2}M$
and $\lambda\la-,-\ra=\la-,-\ra'\det M$. If $M=\lambda\id$ a simple
computation shows that $\alpha'=\lambda^{-1}\alpha$, $\beta'=\lambda^{-1}\beta$
and $\la-,-\ra'=\lambda^{-1}\la-,-\ra$, which implies the claim.
\end{proof}
\begin{thm}
\label{thm:The Proj atlases} The $\GL_{2}$-torsor induced by $\shF\otimes\shL^{-1}\colon(\GCov-\{0\})\to\Bi\GL_{2}$
is of the form $\Proj P\to(\GCov-\{0\})$, where $P$ has the natural
grading of a polynomial algebra, and it splits accordingly to (\ref{eq:decomposition of GCov minus zero})
as $\Proj P=\Proj P_{1}\sqcup\Spec\shR$, where $P_{1}$ is the algebra
of \ref{prop:ZG result on atlas}. In particular we have presentations
\[
(\GCov-\{0\})\simeq[\Proj P/\GL_{2}]\comma(\stZ_{G}-\{0\})\simeq[\Proj(P_{1})/\GL_{2}]\comma(\stZ_{2}-\{0\})\simeq\Bi\GL_{2}
\]
\end{thm}

\begin{proof}
Consider the Cartesian diagrams   \[   \begin{tikzpicture}[xscale=3.4,yscale=-1.2]     \node (A0_1) at (1, 0) {$\Spec P-\{0\}$};     \node (A0_2) at (2, 0) {$\Spec \shR$};     \node (A1_1) at (1, 1) {$V$};     \node (A1_2) at (2, 1) {$\Bi \Gm$};     \node (A1_3) at (3, 1) {$\Spec \shR$};     \node (A2_1) at (1, 2) {$\GCov-\{0\}$};     \node (A2_2) at (2, 2) {$\Bi(\GL_2\times\Gm)$};     \node (A2_3) at (3, 2) {$\Bi(\GL_2)$};     \path (A0_1) edge [->]node [auto] {$\scriptstyle{}$} (A1_1);     \path (A1_3) edge [->]node [auto] {$\scriptstyle{}$} (A2_3);     \path (A0_1) edge [->]node [auto] {$\scriptstyle{}$} (A0_2);     \path (A1_1) edge [->]node [auto] {$\scriptstyle{}$} (A1_2);     \path (A2_2) edge [->]node [auto] {$\scriptstyle{\Bi(\phi)}$} (A2_3);     \path (A0_2) edge [->]node [auto] {$\scriptstyle{}$} (A1_2);     \path (A1_1) edge [->]node [auto] {$\scriptstyle{}$} (A2_1);     \path (A1_2) edge [->]node [auto] {$\scriptstyle{}$} (A1_3);     \path (A2_1) edge [->]node [auto] {$\scriptstyle{(\shF,\shL)}$} (A2_2);     \path (A1_2) edge [->]node [auto] {$\scriptstyle{\Bi(j)}$} (A2_2);   \end{tikzpicture}   \] 
Here $\{0\}\subseteq\Spec P$ is the closed subscheme defined by the
graded irrelevant ideal. The map $\phi\colon\GL_{2}\times\Gm\to\GL_{2}$
is $\phi(M,\lambda)=\lambda^{-1}M$, which is a surjective group homomorphism
whose kernel is $j\colon\Gm\to\GL_{2}\times\Gm$, $j(\lambda)=(\lambda\id,\lambda)$.
The bottom right diagram is Cartesian thanks to \cite[Cor 1.21]{Hochenegger2020}.
The horizontal arrow $\GCov-\{0\}\to\Bi\GL_{2}$ at the bottom is
exactly induced by $\shF\otimes\shL^{-1}$. It follows that $V\simeq[\Spec P-\{0\}/\Gm]$
where the action of $\Gm$ is the restriction along $j$ of the action
of $\GL_{2}\times\Gm$ on $\Spec P$. By \ref{lem:The Proj action}
we can conclude that $[\spec P-\{0\}/\Gm]\simeq\Proj P$. For the
last part notice that $P/Q_{1}=P_{1}$ and $P/Q_{2}=\shR[A]$, so
that $\Proj(P/Q_{2})=\Spec\shR$ (see \ref{thm:main algebra main structure}).
As $Q_{1}$ and $Q_{2}$ are graded ideals also the last remaining
claims holds.
\end{proof}
\begin{cor}
The stacks $(\GCov-\{0\})\to\Spec\shR$ and $(\stZ_{G}-\{0\})\to\Spec\shR$
are universally closed morphisms of stacks.
\end{cor}

\begin{proof}
Both stacks $f\colon\stX\to\Spec\shR$ admits a smooth atlas $U\to\stX$
such that $U\to\Spec\shR$ is universally closed thanks to \ref{thm:The Proj atlases}.
It follows easily that $f\colon\stX\to\Spec\shR$ is universally closed
as well.
\end{proof}

\subsection{Exceptional irreducible component of \texorpdfstring{$(\mu_{3}\rtimes\Z/2\Z)$}{(mu3 rtimes Z/2Z)}-Cov}

In this section we describe the irreducible component $\stZ_{2}=\{\chi\st\beta_{\chi}=\la-,-\ra_{\chi}=0,\alpha_{\chi}=((\tr\alpha_{\chi})/2)\otimes\id_{\shF_{\chi}}\}$
of $\GCov$ (see \ref{thm:geometry of GCov and components}).
\begin{thm}
\label{thm:exceptional irreducible component}Let $\stZ_{2}'$ be
the stack over $\shR$ of triples $(\shL,\shF,\mu)$ where $\shL$
is an invertible sheaf, $\shF$ is a rank $2$ locally free sheaf
and $\mu\colon\shL\to\odi{}$ is a map. Then $\stZ_{2}'\simeq[\A_{\shR}^{1}/\Gm]\times\Bi\GL_{2}$
and we have an equivalence
\[
\stZ_{2}'\to\stZ_{2}\comma(\shL,\shF,\mu)\longmapsto(\shL,\shF,\mu\otimes\id_{\shF},0,0)
\]
Moreover $\stZ_{2}-\{0\}$ is an open substack of $\GCov$ and it
is equivalent to $\Bi\GL_{2}$.
\end{thm}

\begin{proof}
The last claim is included in \ref{thm:The Proj atlases}. Since $\Bi\GL_{2}$
is the stack of rank $2$ locally free sheaves and $[\A^{1}/\Gm]$
is the stack of invertible sheaves with a section, it follows that
$\stZ_{2}'\simeq[\A_{\shR}^{1}/\Gm]\times\Bi\GL_{2}$. It is easy
to check locally that the functor $\stZ_{2}'\to\stZ_{2}$ is well
defined and, using \ref{lem:multiple identity}, that it is an equivalence.
\end{proof}

\subsection{Smooth locus of \texorpdfstring{$(\mu_{3}\rtimes\Z/2\Z)$}{(mu3 rtimes Z/2Z)}-Cov}
\begin{thm}
\label{thm:smooth locus of GCov} We have
\[
\GCov-\{0\}\simeq\Bi\GL_{2}\sqcup(\stZ_{G}-\{0\})
\]
and it is the smooth locus of $\GCov\to\Spec\shR$.
\end{thm}

We state here the above theorem to be consistent with the section
argument. Anyway we will prove it in Section \ref{subsec:smooth locus of ZG},
after showing the smoothness of $\stZ_{G}-\{0\}$ over $\shR$. Here
we show that $\{0\}$ cannot be smooth:
\begin{lem}
\label{lem:zero not smooth} The closed substack $\{0\}$ does not
meet the smooth locus of any of $\stZ_{G}$, $\GCov$ or $(\GCov)^{\text{red}}$
(see \ref{thm:geometry of GCov and components}).
\end{lem}

\begin{proof}
By \ref{thm:main algebra main structure} any of these stacks $\stX$
has a smooth atlas $f\colon\Spec S\to\stX$ of the form $S=\shR[y_{1},\dots,y_{m}]/(q_{1},\dots,q_{l})$,
with $q_{i}$ quadratic polynomials, and such that $f^{-1}\{0\}=V(y_{1},\dots,y_{m})$,
the zero locus of the $y_{i}$. Moreover $l\geq1$. It follows that
the Jacobian ideal in $S$ is contained in $(y_{1},\dots,y_{m})$,
which proves the claim.
\end{proof}

\section{Geometry of the main irreducible components \texorpdfstring{$\stZ_{(\mu_{3}\rtimes\Z/2\Z)}$}{of (mu3 rtimes Z/2Z)-Cov}
and \texorpdfstring{$\stZ_{S_3}$}{of S3-Cov}}

We keep the notation from Section \ref{sec:Data for covers}. In particular
$\shR=\Z[1/2]$, $\shR_{3}=\Z[1/6]$ and $G=\mu_{3}\rtimes\Z/2\Z$.
The aim of this section is to describe the geometry of the irreducible
component $\stZ_{G}$ of $\GCov$ over $\shR$. In particular, over
$\shR_{3}$, we obtain a description of $\stZ_{S_{3}}\simeq\stZ_{G}$
thanks to \ref{thm:from G covers to S3}. We proceed by looking at
particular open substacks of $\stZ_{G}$ and $\GCov$.

\subsection{Triple covers and the locus where \texorpdfstring{$\la-,-\ra\colon\det\shF\arr\shL$}{<-,->}
is an isomorphism.}

The aim of this section is to study the following locus.
\begin{defn}
We denote by $\stU_{\omega}$ the (open) subastack of $\chi\in\GCov$
such that $\la-,-\ra_{\chi}\colon\det\shF_{\chi}\arr\shL_{\chi}$
is an isomorphism.
\end{defn}

\begin{defn}
\label{def:definition of Cthree}Define $\shC_{3}$ as the stack whose
objects are pairs $(\shF,\delta)$ where $\shF$ is a locally free
sheaf of rank $2$ and $\delta$ is a map
\[
\delta\colon\Sym^{3}\shF\arr\det\shF
\]
Denote by $\Cov_{3}$ the stack of degree $3$ covers, or, equivalently,
the stack of locally free sheaves of algebras of rank $3$.
\end{defn}

Notice that $\stC_{3}$ is a smooth stack over $\shR$ because it
is a vector bundle over $\Bi\Gl_{2}$, which is the stack of rank
$2$ locally free sheaves. We are going to define a functor $\shC_{3}\to\stU_{\omega}$
and prove that it is an isomorphism. This also explains the reason
of the section name: it is a classical result (see \cite{Miranda1985,Bolognesi2009,Pardini1989})
that, over $\stR_{3}$, the stack $\shC_{3}$ is isomorphic to the
stack $\Cov_{3}$. We will show that, in this case, the map $\Cov_{3}\simeq\shC_{3}\arr\stU_{\omega}$
is a section of the map $\GCov\arr\Cov_{3}$, obtained taking invariants
by $\sigma\in\Z/2\Z\subseteq G$.

In the appendix \ref{subsec:appendix delta Miranda} we discussed
several constructions obtained from an object $(\shF,\delta)\in\shC_{3}$.
By \ref{prop:Miranda correspondence} $\shC_{3}$ can also be described
as the stack of pairs $(\shF,\beta)$ where $\shF$ is a locally free
sheaf of rank $2$ and $\beta\colon\Sym^{2}\shF\arr\shF$ is a map
such that $\tr\beta=0$. We follow notation from \ref{nota:Miranda},
in particular the definition of $\delta_{\beta},\beta_{\delta},\eta_{\delta},\alpha_{\delta},m_{\delta}$.
\begin{thm}
\label{thm:The-locus when omega is invertible}The maps of stacks
  \[   \begin{tikzpicture}[xscale=4.0,yscale=-0.6]     
\node (A0_0) at (0, 0) {$(\shF,\delta)$};
\node (A0_1) at (1, 0) {$(\det \shF,\shF,\alpha_\delta,\beta_\delta,\id_{\det \shF})$};     
\node (A1_0) at (0, 1) {$\shC_3$};     
\node (A1_1) at (1, 1) {$\stU_\omega$};     
\node (A2_0) at (0, 2) {$(\shF,\delta_\beta)$};     
\node (A2_1) at (1, 2) {$(\shL,\shF,\alpha,\beta,\la-,-\ra)$};     
\path (A0_0) edge [|->,gray]node [auto] {$\scriptstyle{}$} (A0_1);          
\path (A1_0) edge [->]node [auto] {$\scriptstyle{\Lambda}$} (A1_1);     
\path (A2_1) edge [|->,gray]node [auto] {$\scriptstyle{}$} (A2_0);   
\end{tikzpicture}   \] are well defined and they are quasi-inverses of each other. In particular
$\stU_{\omega}$ is a smooth open substack of $\GCov$. Moreover,
over $\stR_{3}$, the composition $\Cov_{3}\simeq\shC_{3}\arrdi{\Lambda}\GCov$
is a section of the map $\GCov\arr\Cov_{3}$ obtained by taking invariants
by $\sigma\in\Z/2\Z$ and the same result hold if we replace $\GCov$
by $\RCov{S_{3}}$.
\end{thm}

We will prove this theorem after collecting some preliminary remarks.
\begin{rem}
\label{rem:etadelta and symmetric product} By \ref{rem:local Miranda},
if $\chi\in\GCov$ and $\tr\beta_{\chi}=0$, so that $\beta_{\chi}=\beta_{\delta}$
(with $\delta=\delta_{\beta_{\chi}})$, then, by (\ref{eq:ass alpha, gamma first})
and (\ref{eq:loc com and ass conditions}), we have
\begin{equation}
\eta_{\delta}=2(-,-)_{\chi}\label{eq:eta delta is two times symmetric product}
\end{equation}
\end{rem}

We now want to show the relationship between $\shC_{3}$ and $\Cov_{3}$.
The reader can refer to \cite{Miranda1985,Bolognesi2009} for details
and proofs.
\begin{rem}
\label{rem:triple covers and invariants by sigma}If $\Phi=(\shF,\delta)\in\shC_{3}$
and we set $\alA_{\Phi}=\odi T\oplus\shF$, we can endow $\alA_{\Phi}$
by a structure of $\odi T$-algebras given by
\[
\Sym^{2}\shF\arrdi{\eta_{\delta}+\beta_{\delta}}\alA_{\Phi}
\]
This association defines a map of stacks $\shC_{3}\arr\Cov_{3}$.
This map is an isomorphism if $3$ is inverted in the base scheme.
Indeed, over $\shR_{3}$, if $\alA\in\Cov_{3}$, the trace map $\tr_{\alA/\odi T}\colon\alA\arr\odi T$
is surjective and we can write $\alA=\odi S\oplus\shF$, where $\shF=\ker\tr_{\alA}$.
The multiplication of $\alA$ induces a map $\beta\colon\Sym^{2}\shF\arr\shF$
such that $\tr\beta=0$ and therefore a $\delta\colon\Sym^{3}\shF\arr\det\shF$
such that $\beta_{\delta}=\beta$.

Now let $\chi=(\shL,\shF,m,\alpha,\beta,\la-,-\ra)\in\GCov$. It's
easy to see that 
\[
\alA_{\chi}^{\sigma}=\{a\oplus0\oplus x_{1}\oplus x_{2}\st a\in\odi T\comma x_{1}=x_{2}\in\shF\}
\]
where $\sigma\in\Z/2\Z\subseteq G$. The map   \[   \begin{tikzpicture}[xscale=3.0,yscale=-0.5]     \node (A0_0) at (0, 0) {$\odi{T}\oplus\shF$};     \node (A0_1) at (1, 0) {$\alA^\sigma$};     \node (A1_0) at (0, 1) {$a\oplus x$};     \node (A1_1) at (1, 1) {$a\oplus 0\oplus x\oplus x$};     \path (A0_0) edge [->]node [auto] {$\scriptstyle{}$} (A0_1);     \path (A1_0) edge [|->,gray]node [auto] {$\scriptstyle{}$} (A1_1);   \end{tikzpicture}   \] 
is an isomorphism of $\odi S$-modules and the induced algebra structure
on $\odi T\oplus\shF$ is given by
\[
\beta\colon\Sym^{2}\shF\arr\shF\text{ and }2(-,-)\colon\Sym^{2}\shF\arr\odi T
\]
Moreover $(\tr_{\alA^{\sigma}})_{|\odi{}}=3\id$ and $(\tr_{\alA^{\sigma}})_{|\shF}=\tr\beta$.
Over $\shR_{3}$ this means $\shF=\ker\tr_{\alA^{\sigma}}$ if and
only if $\tr\beta=0$. In general we obtain a map of stacks $\{\tr\beta=0\}=(\GCov)^{\text{red}}\to\shC_{3}$
(see \ref{thm:geometry of GCov and components}) and, by \ref{rem:etadelta and symmetric product},
$(\GCov)^{\text{red}}\to\shC_{3}\to\Cov_{3}$ consists in taking invariants
by $\sigma$.
\end{rem}

\begin{rem}
\label{rem:invariants by sigma for Sthree and GCov}Theorem \ref{ex:SThree covers and the other group}
tells us that, over $\stR_{3}$, the isomorphism $\GCov\simeq\RCov{S_{3}}$
preserves the quotient by $\sigma\in\Z/2\Z$, that is we have a commutative
diagram   \[   \begin{tikzpicture}[xscale=0.9,yscale=-1.0]     \node (A0_0) at (0, 0) {$\GCov$};     \node (A0_2) at (2, 0) {$\RCov{S_3}$};     \node (A0_4) at (4, 0) {$X$};     \node (A0_5) at (5, 0) {$\alA$};     \node (A1_1) at (1, 1) {$\Cov_3$};     \node (A1_4) at (4, 1) {$X/\sigma$};     \node (A1_5) at (5, 1) {$\alA^\sigma$};     \path (A0_4) edge [|->,gray]node [auto] {$\scriptstyle{}$} (A1_4);     \path (A0_0) edge [->]node [auto] {$\scriptstyle{}$} (A1_1);     \path (A0_5) edge [|->,gray]node [auto] {$\scriptstyle{}$} (A1_5);     \path (A0_2) edge [->]node [auto] {$\scriptstyle{}$} (A1_1);     \path (A0_0) edge [->]node [auto] {$\scriptstyle{\simeq}$} (A0_2);   \end{tikzpicture}   \] 
\end{rem}

\begin{proof}
[Proof of Theorem \ref{thm:The-locus when omega is invertible}] We
need to prove that $\Lambda$ is well defined. Let $\Phi=(\shF,\delta)\in\shC_{3}$.
We have that $\chi=\Lambda(\Phi)\in\stY$ and we have to prove that
$\chi$ satisfies the conditions of \ref{lem:local conditions for the map for Sthree}.
We can therefore work locally and fix a basis $y,z$ of $\shF$. By
(\ref{eq:expression of delta}), the parameters associated to $\chi$
(see \ref{not: associated parameters for chi and Sthree}) are 
\[
a,b,c,d=-a,e,f=-c,\omega=1,A=-D=(ac+be)/2\comma B=c^{2}-ae\comma C=-a^{2}-bc
\]
It is easy to check that all the conditions in \ref{lem:local conditions for the map for Sthree}
are satisfied. So $\Lambda(\Phi)\in\GCov$ and, by definition, $\Lambda(\Phi)\in\stU_{\omega}$.

Following notation from \ref{thm:GCov as quotient stack} we have
Cartesian diagrams   \[   \begin{tikzpicture}[xscale=2.4,yscale=-1.2]     \node (A0_0) at (0, 0) {$\Spec S$};     \node (A0_1) at (1, 0) {$\Spec P_\omega$};     \node (A0_2) at (2, 0) {$\Spec P$};     \node (A0_3) at (3, 0) {$\Spec \shR$};     \node (A1_0) at (0, 1) {$\shC_3$};     \node (A1_1) at (1, 1) {$\stU_\omega$};     \node (A1_2) at (2, 1) {$\GCov$};     \node (A1_3) at (3, 1) {$\Bi \GL_2 \times \Bi \Gm$};     \path (A0_1) edge [->]node [auto] {$\scriptstyle{}$} (A1_1);     \path (A0_0) edge [->]node [auto] {$\scriptstyle{}$} (A0_1);     \path (A0_2) edge [->]node [auto] {$\scriptstyle{}$} (A0_3);     \path (A1_0) edge [->]node [auto] {$\scriptstyle{\Lambda}$} (A1_1);     \path (A0_3) edge [->]node [auto] {$\scriptstyle{}$} (A1_3);     \path (A1_1) edge [->]node [auto] {$\scriptstyle{}$} (A1_2);     \path (A0_2) edge [->]node [auto] {$\scriptstyle{}$} (A1_2);     \path (A0_0) edge [->]node [auto] {$\scriptstyle{}$} (A1_0);     \path (A0_1) edge [->]node [auto] {$\scriptstyle{}$} (A0_2);     \path (A1_2) edge [->]node [auto] {$\scriptstyle{}$} (A1_3);   \end{tikzpicture}   \] 
By definition of $\shC_{3}$ one have $S=\shR[a,b,c,e,\omega]_{\omega}$
and $P\to S$ is defined as above. In order to conclude that $\Lambda$
is an equivalence it is enough to notice that $P_{\omega}\to S$ is
an isomorphism. This holds because, if $\omega$ is invertible, then
$a=-d$, $c=-f$, $A=-D$ and $B$ and $C$ are functions of the $a,b,c,e$.
By a direct check, the only two non trivial relations missing are
automatically satisfied.

Now assume we are over $\stR_{3}$. The map $\GCov\arr\Cov_{3}\simeq\shC_{3}$
extends the map $\stU_{\omega}\arr\shC_{3}$ defined in the statement.
Therefore $\Cov_{3}\simeq\shC_{3}\arr\stU_{\omega}\subseteq\GCov$
is a section of such map. The claim about $S_{3}$ follows from \ref{rem:invariants by sigma for Sthree and GCov}.
\end{proof}
\begin{cor}
Set $\shF=\stR^{2}$ with basis $e_{1},e_{2}$ and consider $\delta\colon\Sym^{3}\shF\arr\det\shF$
given by $\delta(e_{2}^{3})=-\delta(e_{1}^{3})=1$ and $\delta(e_{1}e_{2}^{2})=\delta(e_{1}^{2}e_{2})=0$.
Then
\[
G\simeq\Autsh_{\shC_{3}}(\shF,\delta)
\]
Over $\shR_{3}$, the map $\Bi G\arr\Cov_{3}$ obtained by taking
invariants by $\sigma\in\Z/2\Z$, is an equivalence onto the locus
$\text{Et}_{3}$ of étale degree $3$ covers. Moreover
\[
G\simeq\Autsh_{\Cov_{3}}(\shR_{3}[t]/(t^{3}-1))
\]
\end{cor}

\begin{proof}
By \cite[Thm A]{Tonini2015}, the trivial $G$-torsor $G\to\Spec\shR$
is associated with the forgetful functor $\Omega\colon\Loc^{G}\stR\arr\Loc\stR$.
In particular, taking into account section \ref{subsec:representation theory},
the sequence $\chi=(\shL,\shF,m,\alpha,\beta,\la-,-\ra)\in\GCov(\stR)$
associated with $\Omega$ (and the trivial $G$-torsor) is given by
\[
\shL=A,\shF=V;\begin{array}{l}
\alpha(1_{A}\otimes v_{1})=-v_{1}\\
\alpha(1_{A}\otimes v_{2})=v_{2}
\end{array};\begin{array}{l}
\beta(v_{1}^{2})=v_{2}\comma\beta(v_{1}v_{2})=0\comma\beta(v_{2}^{2})=v_{1}\\
\la v_{1},v_{2}\ra=(1/2)1_{A}\comma m(1_{A}\otimes1_{A})=1
\end{array}
\]
In particular $\chi\in\stU_{\omega}$ and, by \ref{thm:The-locus when omega is invertible},
the image of $\chi$ via the equivalence $\Phi\colon\stU_{\omega}\to\shC_{3}$
is $(\shF,\delta_{\beta})$. Moreover $G\simeq\Autsh_{\GCov}\chi\simeq\Autsh_{\shC_{3}}(\shF,\delta_{\beta})$
and, by (\ref{iso:delta beta}), $\delta_{\beta}=\delta$.

Assume now that the base ring is $\shR_{3}$, so that $\Phi\colon\stU_{\omega}\to\shC_{3}\simeq\Cov_{3}$
consists in taking invariants by $\sigma$. By \ref{rem:triple covers and invariants by sigma}
the algebra $\Phi(\chi)$ associated with $\chi$ or $(\shF,\delta)$
is $S=\shR_{3}\oplus\shF$ with multiplication $2(-,-)\oplus\beta\colon\Sym^{2}\shF\to\shR_{3}\oplus\shF$.
A direct computation shows that $v_{1}^{2}=v_{2}$ and $v_{1}^{3}=1$.
It follows that $S=\shR_{3}[t]/(t^{3}-1)$ and, in particular, the
last isomorphism. Moreover $\Phi_{|\Bi G}\colon\Bi G\to\Cov_{3}$
is an equivalence onto the substack $E$ of $\Cov_{3}$ of objects
which are locally isomorphic to $\Phi(\chi)=S$. Since $\Spec S\to\Spec\shR_{3}$
is an étale cover of degree $3$, it follows that $E=\text{Et}_{3}$.
\end{proof}

\subsection{The locus where \texorpdfstring{$\alpha\colon\shL\otimes\shF\arr\shF$}{alpha}
is nowhere a multiple of the identity.}
\begin{defn}
We denote by $\stU_{\alpha}$ the full substack of $\chi\in\GCov$
such that $\alpha_{\chi}\colon\shL_{\chi}\otimes\shF_{\chi}\arr\shF_{\chi}$
is nowhere a multiple of the identity, i.e. it is not a multiple of
the identity over all geometric points of the base (see also \ref{def:locally multiple identity}).
\end{defn}

\begin{thm}
\label{thm:not degenerate locus of S3}Let $R=\stR[m,a,b]$. Then
\[
(R,R^{2},\alpha,\beta,\la-,-\ra)\text{ where }\alpha=\left(\begin{array}{cc}
0 & m\\
1 & 0
\end{array}\right)\comma\beta=\left(\begin{array}{ccc}
a & -mb & ma\\
b & -a & mb
\end{array}\right)\comma\la-,-\ra=mb^{2}-a^{2}
\]
is an object of $\RCov G(R)$. The induced map $\A^{3}\arr\RCov G$
is a smooth Zariski epimorphism onto $\stU_{\alpha}$. In particular
$\stU_{\alpha}$ is a smooth open substack of $\GCov$.
\end{thm}

Before proving this Theorem we need two lemmas.
\begin{lem}
\label{lem:the not degenerate locus is open}Let $\shF$ be a locally
free sheaf of rank $2$, $\shL$ be an invertible sheaf, both over
a scheme $T$ and $\alpha\colon\shL\otimes\shF\arr\shF$ be a map.
Let also $k$ be a field, $\Spec k\arr T$ be a map and $p\in T$
the induced point. If $\alpha\otimes k$ is not a multiple of the
identity, then there exists a Zariski open neighborhood $V$ of $p$
in $T$ and $y\in\shF_{|V}$ such that $\shL_{|V}=\odi Vt$ and $y,\alpha(t\otimes y)$
is a basis of $\shF_{|V}$.
\end{lem}

\begin{proof}
By Nakayama's lemma we can assume $T=\Spec k$. By contradiction assume
that a basis as in the statement does not exist. It is easy to deduce
that any vector of $\shF$ is an eigenvector for $\alpha$. By standard
linear algebra we can conclude that $\alpha$ is a multiple of the
identity.
\end{proof}
\begin{lem}
\label{lem:associated parameters for alpha nowhere multiple of the identity}Let
$\chi=(\shL,\shF,\alpha,\beta,\la-,-\ra)\in\stY$ and $y\in\shF$
be such that $\shL=\odi T$ and $y,z=\alpha(y)$ is a basis of $\shF$.
Then $\chi\in\GCov$ if and only if the associated parameters (see
\ref{not: associated parameters for chi and Sthree}) of $\chi$ with
respect to the basis $y,z$ are 
\[
a,b,c=-mb,d=-a,e=ma,f=mb,\omega=mb^{2}-a^{2},A=D=0,B=m,C=1
\]
In this case $\chi\in\stU_{\alpha}$.
\end{lem}

\begin{proof}
First of all note that, if the associated parameters of $\chi$ are
as above, then they satisfy equations (\ref{eq:loc com and ass conditions}).
Therefore $\chi\in\GCov$ and, by definition, $\alpha$ is nowhere
a multiple of the identity, i.e. $\chi\in\stU_{\alpha}$. Consider
now the converse implication and denote by $a,b,c,d,e,f,\omega,A,B,C,D,m$
the parameters associated with $\chi\in\GCov$ with respect to the
basis $y,z$ of $\shF$. In particular equations (\ref{eq:loc com and ass conditions})
hold true. By definition of $y,z$ we have $A=0$, $C=1$ and therefore
\[
C(A+D)=0\then D=0\comma m=(A^{2}+D^{2})/2+BC=B
\]
\[
b(A+D)+C(a+d)=d(A+D)+C(c+f)=0\then d=-a\comma f=-c
\]
\[
(2aA+bB+cC)=(2cA+dB+eC)=0\then c=-mb,e=ma
\]
\[
a^{2}+bc=-\omega C\then\omega=mb^{2}-a^{2}
\]
\end{proof}

\begin{proof}
(\emph{of theorem }\ref{thm:not degenerate locus of S3}). By \ref{lem:the not degenerate locus is open}
and \ref{lem:associated parameters for alpha nowhere multiple of the identity},
$\stU_{\alpha}$ is an open substack of $\GCov$, $\chi\in\stU_{\alpha}(R)$
and the induced map $\pi\colon\A^{3}\arr\stU_{\alpha}$ is a Zariski
epimorphism.

It remains to prove that $\pi$ is smooth. The scheme $\A^{3}$ represents
the functor $(\Sch/\shR)^{\op}\to\sets$ associating with a scheme
$T$ the setoid of tuples $(\chi,t,y)$ where $\chi=(\shL,\shF,\alpha,\beta,\la-,-\ra)\in\GCov(T)$,
$t\in\shL$ is a generator and $y\in\shF$ is an element such that
$y,\alpha(t\otimes y)$ is a basis for $\shF$. In particular, given
a map $\chi=(\shL,\shF,\alpha,\beta,\la-,-\ra)\colon T\to\stU_{\alpha}$
the fiber product $Z=T\times_{\stU_{\alpha}}\A^{3}$ represents the
functor $(\Sch/T)^{\op}\to\sets$ of pairs $(t,y)$ where $t\in\shL$
is a generator and $y,\alpha(t\otimes y)$ is a basis of $\shF$.

The smoothness of $Z\to T$ is local on $T$, therefore we can assume
$\chi=\pi(m,a,b)$ for $m,a,b\in\odi T$. If $y=(u,v)\in\shF=\odi T^{2}$
then $\alpha(y)=(mv,u)$ and $y,\alpha(y)$ is a basis of $\shF$
if and only if $u^{2}-mv^{2}$ is invertible. This observation allows
us to conclude that
\[
Z=\Spec(\odi T[X,Y,W]_{X(Y^{2}-mW^{2})})\subseteq\A_{T}^{3}
\]
is open, hence smooth.
\end{proof}

\subsection{The locus where \texorpdfstring{$\beta\colon\Sym^{2}\shF\arr\shF$}{beta}
is nowhere zero.}

Given an object $\chi\in\GCov$ or a morphism $\beta\colon\Sym^{2}\shF\to\shF$
(where $\shF$ is a rank $2$ locally free sheaf) we define the map
\[
d_{\beta}\colon\shF\to\det\shF\comma x\longmapsto x\wedge\beta(x^{2})
\]
Notice that this is not a map of quasi-coherent sheaves because it
is not $\odi{}$-linear, it is just a map of sheaves of sets. We remark
that its formation commutes with arbitrary base changes of the base.
\begin{defn}
We define $\stU_{\beta}$ as the full substack of objects $\chi\in\GCov$
such that $d_{\beta_{\chi}}\colon\shF_{\chi}\to\det\shF_{\chi}$ is
nowhere zero, i.e. $d_{\beta_{\chi}}$ is not zero on all geometric
points of the base.
\end{defn}

\begin{thm}
\label{thm:the locus where beta is not zero}Let $R=\stR[\omega,A,C]$.
Then 
\[
(R,R^{2},\alpha,\beta,\la-,-\ra)\text{ where }\alpha=\left(\begin{array}{cc}
A & \omega C^{2}\\
C & -A
\end{array}\right)\comma\beta=\left(\begin{array}{ccc}
0 & -\omega C & 2\omega A\\
1 & 0 & \omega C
\end{array}\right)\comma\la-,-\ra=\omega
\]
is an object of $\RCov G(R)$. The associated map $\pi\colon\A^{3}\arr\RCov G$
is smooth and its image is $\stU_{\beta}$, which is therefore a smooth
open substack of $\GCov$.

Over $\shR_{3}$, $\stU_{\beta}$ coincides with the full substack
of $\chi\in\GCov$ such that $\beta_{\chi}$ is nowhere zero.
\end{thm}

We prove the above result at the end of the section.
\begin{lem}
\label{lem:local basis for beta nowhere zero} Let $T$ be an $\shR$-scheme,
$\shF$ a free sheaf of rank $2$ and
\[
\beta=\left(\begin{array}{ccc}
a & c & e\\
b & d & f
\end{array}\right)\colon\Sym^{2}\shF\to\shF
\]
 a morphism. Then the locus $U$ where $d_{\beta}\colon\shF\to\det\shF$
is nowhere zero is open and its complement is defined by the ideal
$(b,e,2d-a,f-2c)$. In particular if $T$ is an $\shR_{3}$-scheme
and $\tr\beta=0$ then $d_{\beta}$ is nowhere zero if and only if
$\beta$ is nowhere zero.

Moreover for any $u\in U$ there exists an open neighborhood $V$
of $u$ inside $U$ and $y\in\shF(V)$ such that $y,\beta(y^{2})$
is a basis of $\shF_{|V}$.
\end{lem}

\begin{proof}
Set $\shF=\odi T^{2}$ with basis $e_{1},e_{2}$. Given $y=ue_{1}+ve_{2}$
a direct computation shows that
\[
d_{\beta}(y)=y\wedge\beta(y^{2})=(u^{3}b+u^{2}v(2d-a)+uv^{2}(f-2c)-v^{3}e)e_{1}\wedge e_{2}
\]
Choosing $y=e_{1},e_{2},e_{1}+e_{2},e_{1}-e_{2}$ we see that
\[
d_{\beta}=0\iff b=e=2d-a=f-2c=0
\]
which proves the first claims. For the last one, by Nakayama there
exist a neighborhood $u\in V$, $y\in\shF$ and $q\in\odi V$ not
zero in $u$ such that $d_{\beta}(y)=q(e_{1}\wedge e_{2})\in\det\shF_{|V}$.
Inverting $q$ in $V$ we find the desired neighborhood\@.
\end{proof}
\begin{lem}
\label{lem:associated parameters for beta nowhere zero}Let $\chi=(\shL,\shF,\alpha,\beta,\la-,-\ra)\in\stY$
and $y\in\shF$ be such that $\shL=\odi T$ and $y,z=\beta(y^{2})$
is a basis of $\shF$. Then $\chi\in\GCov$ if and only if the associated
parameters (see \ref{not: associated parameters for chi and Sthree})
of $\chi$ with respect to the basis $y,z$ are
\[
a=0,b=1,c=-\omega C,d=0,e=2\omega A,f=\omega C,\omega,A,B=\omega C^{2},C,D=-A
\]
In this case $\chi\in\stU_{\beta}$. 
\end{lem}

\begin{proof}
First of all, it is easy to check that, if the associated parameters
of $\chi$ are the ones listed in the statement, then they satisfy
equations (\ref{eq:loc com and ass conditions}). Therefore $\chi\in\GCov$
and, since $\beta(y^{2})\neq0$ after all base changes, $\chi\in\stU_{\beta}$.

Assume now that $\chi\in\GCov$. By definition of the basis $y,z$,
we have $a=0$ and $b=1$. Using relations (\ref{eq:loc com and ass conditions}),
we also have
\[
b(a+d)=a(a+d)+b(c+f)=0\then d=-a=0\comma f=-c
\]
\[
b(A+D)+C(a+d)=0\then D=-A
\]
\[
a^{2}+bc=-\omega C\comma ac+be=2\omega A\then c=-\omega C\comma e=2\omega A
\]
\[
2aA+bB+cC=0\then B=\omega C^{2}
\]
\end{proof}

\begin{proof}
[Proof of Theorem \ref{thm:the locus where beta is not zero}] From
\ref{lem:local basis for beta nowhere zero} and \ref{lem:associated parameters for beta nowhere zero}
we see that $(R,R^{2},\alpha,\beta,\la-,-\ra)\in\stU_{\beta}(R)$,
$\pi\colon\A^{3}\arr\stU_{\beta}$ is a Zariski epimorphism and $\stU_{\beta}$
is open. The last claim follows from \ref{lem:local basis for beta nowhere zero}
and the fact that if $\chi\in\GCov(k)$ for a field $k$ then $\tr\beta_{\chi}=0$
by \ref{thm:geometry of GCov and components}.

It remains to prove that $\pi$ is smooth. By \ref{lem:associated parameters for beta nowhere zero}
the scheme $\A^{3}$ represents the functor $(\Sch/\shR)^{\op}\to\sets$
associating with a scheme $T$ the setoid of tuples $(\chi,t,y)$
where $\chi=(\shL,\shF,\alpha,\beta,\la-,-\ra)\in\GCov(T)$, $t\in\shL$
is a generator and $y\in\shF$ is an element such that $y,\beta(y^{2})$
is a basis for $\shF$. In particular, given a map $\chi=(\shL,\shF,\alpha,\beta,\la-,-\ra)\colon T\to\stU_{\alpha}$
the fiber product $Z=T\times_{\stU_{\alpha}}\A^{3}$ represents the
functor $(\Sch/T)^{\op}\to\sets$ of pairs $(t,y)$ where $t\in\shL$
is a generator and $y,\beta(y^{2})$ is a basis of $\shF$.

The smoothness of $Z\to T$ can be checked locally on $T$ and therefore
we can assume $\shL=\odi T$ and $\shF=\odi T^{2}$, so that $Z\subseteq\A_{T}^{3}$.
If $y=(u,v)\in\shF$ then $\beta(y^{2})=(f(u,v),g(u,v))$ for some
$f,g\in\odi T[X,Y]$ and $y,\beta(y^{2})$ is a basis of $\shF$ if
and only if $ug(u,v)-vf(u,v)$ is invertible in the base. This implies
that
\[
Z=\Spec\odi T[W,X,Y]_{W(Xg-Yf)}\subseteq\A_{T}^{3}
\]
is open, hence smooth.
\end{proof}

\subsection{The stack of torsors \texorpdfstring{$\Bi (\mu_{3}\rtimes\Z/2\Z)$}{B(mu3 rtimes Z/2Z)}.}

In this section we want to describe the stack $\Bi G$ of $G$-torsors.
\begin{defn}
\label{def:discriminant maps for Sthree} Given $\chi=(\shF,\delta)\in\shC_{3}$
(see \ref{thm:The-locus when omega is invertible}) (resp. $\chi=(\shL,\shF,\alpha,\beta,\la-,-\ra)\in\GCov$)
we define the discriminant map $\Delta_{\Phi}\colon(\det\shF)^{2}\arr\odi T$
as the map obtained from $\shF\otimes\shF\to\Sym^{2}\shF\arrdi{\eta_{\delta}}\odi T$
(resp. $\shF\otimes\shF\to\Sym^{2}\shF\arrdi{(-,-)_{\chi}}\odi T$)
as in \ref{rem:Lambda and duals}.
\end{defn}

\begin{rem}
\label{rem:general discriminant}For $\chi=(\shL,\shF,\alpha,\beta,\la-,-\ra)\in\GCov$
the map $\Delta_{\chi}$ coincides with
\[
(\det\shF)^{2}\arrdi{\la-,-\ra^{\otimes2}}\shL^{2}\arrdi{-m_{\chi}}\odi T
\]
Indeed locally one has $\Delta_{\chi}=(y,y)(z,z)-(y,z)^{2}=-BC\omega^{2}-A^{2}\omega^{2}=-\omega^{2}m$.

Moreover, if $\tr\beta=0$, then $\Delta_{(\shF,\delta_{\beta})}=4\Delta_{\chi}$
thanks to \ref{eq:eta delta is two times symmetric product}.
\end{rem}

\begin{rem}
\label{rem: fake and true discriminant}Let $\Phi=(\shF,\delta)\in\shC_{3}$
and $\alA_{\Phi}=\odi T\oplus\shF$ be the algebra associated with
$\Phi$ (see \ref{rem:triple covers and invariants by sigma}). By
\ref{rem:discriminant} the discriminant of $\alA$ is $\Delta_{\alA_{\Phi}}=3^{3}\Delta_{\Phi}$
and, over $\shR_{3}$, the algebra $\alA_{\Phi}$ is étale over $T$
if and only if $\Delta_{\Phi}$ is an isomorphism.
\end{rem}

\begin{thm}
\label{thm:description of Sthree torsors}An object $\chi=(\shL,\shF,m,\alpha,\beta,\la-,-\ra)\in\GCov$
corresponds to a $G$-torsor if and only if the maps
\[
m\colon\shL^{2}\arr\odi T\comma\la-,-\ra\colon\det\shF\arr\shL
\]
are isomorphisms, or, equivalently, $\Delta_{\chi}\colon(\det\shF)^{2}\arr\odi T$
is an isomorphism. In this case the maps 
\[
\alpha\colon\shL\otimes\shF\to\shF\text{ and }(-,-)\oplus\langle-,-\rangle\oplus\beta\colon\shF\otimes\shF\to\odi{}\oplus\shL\oplus\shF
\]
are isomorphisms and $\tr\beta=\tr\alpha=0$. Moreover $\Bi G\subseteq\stU_{\omega},\stU_{\alpha},\stU_{\beta}$.

Finally the map $\Lambda\colon\shC_{3}\to\GCov$ of Theorem \ref{thm:The-locus when omega is invertible}
restricts to an equivalence between the full substack of $\shC_{3}$
of objects $\Phi$ such that $\Delta_{\Phi}$ is an isomorphism and
$\Bi G$.
\end{thm}

\begin{proof}
Let $\Gamma$ be the functor associated with $\chi$. By \cite[Thm A]{Tonini2015}
$\chi$ corresponds to a $G$-torsor if and only if $\Gamma$ is strong
monoidal, that is all maps $\Gamma_{U}\otimes\Gamma_{W}\to\Gamma_{U\otimes W}$
are isomorphisms for $U,W\in\Loc^{G}\shR$. Equivalently: the maps
\[
\Gamma_{A}\otimes\Gamma_{A}\to\Gamma_{A\otimes A}\simeq\Gamma_{\shR}\comma\Gamma_{A}\otimes\Gamma_{V}\to\Gamma_{A\otimes V}\simeq\Gamma_{V}\comma\Gamma_{V}\otimes\Gamma_{V}\to\Gamma_{V\otimes V}\simeq\Gamma_{\shR}\oplus\Gamma_{A}\oplus\Gamma_{V}
\]
which are exactly $m,\alpha$ and $(-,-)\oplus\langle-,-\rangle\oplus\beta$,
are isomorphisms. If this is the case then $\langle-,-\rangle\colon\det\shF\to\shL$
is surjective, hence an isomorphism. In particular $\Bi G\subseteq\stU_{\omega},\stU_{\beta}$,
$\tr\beta=0$, and, thanks to \ref{rem:general discriminant}, the
last claim follows from the first ones.

Conversely assume that $m$ and $\langle-,-\rangle$ are isomorphisms.
The map $\alpha$ is an isomorphism because $\alpha^{2}=m\id$ locally.
Moreover $\chi\in\stU_{\omega}$ implies that $\tr\alpha=0$, as one
can check locally. If by contradiction $\chi\notin\stU_{\alpha}$
then, on some geometric point, we would have that $\alpha$ is a multiple
of the identity and therefore $\alpha=(\tr\alpha/2)\id=0$, which
is not the case. So $\chi\in\stU_{\alpha}$.

It remains to prove that $\chi\in\Bi G$. Since this is a local statement,
we can assume $\chi=\pi(m,a,b)$, where $\pi\colon\A^{3}\to\stU_{\alpha}$
is the map defined in \ref{thm:not degenerate locus of S3}. We have
to prove that $(-,-)\oplus\langle-,-\rangle\oplus\beta\colon\shF\otimes\shF\to\odi{}\oplus\shL\oplus\shF$
is an isomorphism. If $M$ is its attached matrix we have to show
that $\det M$ is invertible. We have
\[
M=\left(\begin{array}{cccc}
-\omega & 0 & 0 & m\omega\\
0 & -\omega & \omega & 0\\
a & -mb & -mb & ma\\
b & -a & -a & mb
\end{array}\right)\then N=\left(\begin{array}{cccc}
-\omega & 0 & 0 & 2m\omega\\
0 & -\omega & 2\omega & 0\\
a & -mb & 0 & 0\\
b & -a & 0 & 0
\end{array}\right)
\]
The matrix $N$ is obtained replacing the third and fourth columns
of $M$ by $M^{3}-M^{2}$ and $M^{4}-mM^{1}$ respectively, where
$M^{j}$ denotes the $j$-th column. Thus
\[
\det M=\det N=-2m\omega\cdot2\omega\cdot(mb^{2}-a^{2})=-4m\omega^{3}\text{ is invertible}
\]
\end{proof}

\subsection{The main irreducible component \texorpdfstring{$\stZ_{(\mu_{3}\rtimes\Z/2\Z)}$}{of (mu3 rtimes Z/2Z)-Cov}
\label{subsec:smooth locus of ZG}}

In this subsection we want to give a more precise description of the
irreducible component $\stZ_{G}$ of $\GCov$ and, because of \ref{ex:SThree covers and the other group},
of $\stZ_{S_{3}}\subseteq\RCov{S_{3}}$ over $\stR_{3}$ (see \ref{def:main irreducible component}).
We are going to use results and notation from section \ref{subsec:Trace zero alpha}.
In particular we use notation \ref{not: associated check map}.

Define the stack $\stZ$ whose objects are tuples $(\shM,\shF,\delta,\zeta,\omega)$
where $\shM$ is an invertible sheaf, $\shF$ is a locally free sheaf
of rank $2$, $\omega$ is a section of $\shM$ and $\delta,\zeta$
are maps
\[
\delta\colon\Sym^{3}\shF\arr\det\shF\comma\zeta\colon(\det\shF)^{2}\otimes\shM\arr\Sym^{2}\shF
\]
satisfying the following conditions:
\begin{enumerate}
\item the composition
\begin{equation}
(\det\shF)^{2}\otimes\shM\otimes\shF\arrdi{\zeta\otimes\id}\Sym^{2}\shF\otimes\shF\arr\Sym^{3}\shF\arrdi{\delta}\det\shF\label{eq:first condition on M,F,delta,zeta,omega}
\end{equation}
is zero;
\item the composition 
\begin{equation}
\Sym^{2}\shF\arrdi{\check{\zeta}}\shM^{-1}\arrdi{\duale{\omega}}\odi{}\label{eq:second condition on M,F,delta,zeta,omega}
\end{equation}
 coincides with $\eta_{\delta}$ (see \ref{not: associated check map}
and (\ref{eq:etabeta})).
\end{enumerate}
We define a functor 
\[
\stZ\to\stY\comma\Omega=(\shM,\shF,\delta,\zeta,\omega)\longmapsto(\shL_{\Omega},\shF_{\Omega},\alpha_{\Omega},\beta_{\Omega},\la-,-\ra_{\Omega})
\]
setting $\shL_{\Omega}=\shM\otimes\det\shF$, $\shF_{\Omega}=\shF$,
$\alpha_{\Omega}\colon\shL_{\Omega}\otimes\shF\arr\shF$ the trace
$0$ map obtained from $\zeta$ via \ref{lem:Trace zero map alpha description},
$\beta_{\Omega}=\beta_{\delta}\colon\Sym^{2}\shF\arr\shF$ (see (\ref{iso:delta beta}))
and finally $\la-,-\ra_{\Omega}=\omega\otimes\id_{\det\shF}\colon\det\shF\arr\shM\otimes\det\shF=\shL_{\Omega}$.

Remembering the notation introduced in \ref{not: trace of beta},
we want to prove the following Theorem.
\begin{thm}
\label{thm:description of the main component}We have 
\[
\stZ_{G}-\{0\}=\stU_{\omega}\cup\stU_{\alpha}\cup\stU_{\beta}
\]
(see \ref{def:the zero section}) and it is the smooth locus of $\stZ_{G}\to\Spec\shR$.

Moreover we have an equivalence of stacks.   \[   \begin{tikzpicture}[xscale=4.5,yscale=-0.6]     \node (A0_0) at (0, 0) {$\stZ$};     \node (A0_1) at (1, 0) {$\stZ_G$};     \node (A1_0) at (0, 1) {$\Omega=(\shM,\shF,\delta,\zeta,\omega)$};     \node (A1_1) at (1, 1) {$(\shL_{\Omega},\shF_{\Omega},\alpha_{\Omega},\beta_{\Omega},\la-,-\ra_{\Omega})$};     \path (A0_0) edge [->]node [auto] {$\scriptstyle{}$} (A0_1);     \path (A1_0) edge [|->,gray]node [auto] {$\scriptstyle{}$} (A1_1);   \end{tikzpicture}   \] 
\end{thm}

\begin{cor}
\label{cor:the projective surface} The scheme $X=\Proj(\shR[a,b,c,e,A,B,C,\omega]/I)\to\Spec\shR$,
where
\[
I=(2aA+bB+cC,2cA-aB+eC,2\omega A-(ac+be),\omega B-(c^{2}-ae),\omega C+a^{2}+bc)
\]
is smooth and projective and its geometric fibers are non degenerated
surfaces in $\PP^{7}$. Moreover $(\stZ_{G}-\{0\})\simeq[X/\GL_{2}]$.
\end{cor}

\begin{proof}
The global claims follows from \ref{thm:The Proj atlases} and \ref{thm:description of the main component}.
For the geometry of the fibers, we base change to a geometric point
of the base. Since $f\colon X\to(\stZ_{G}-\{0\})$ has relative dimension
$2$ and $\Bi G$ is a $0$-dimensional open substack of $\stZ_{G}-\{0\}$,
it follows that $\dim X=\dim f^{-1}(\Bi G)=2$.
\end{proof}
We will prove Theorem \ref{thm:description of the main component}
after the following lemma.
\begin{lem}
\label{lem:for the description of Z Sthree}Let $\chi=(\shL,\shF,\alpha,\beta,\la-,-\ra)\in\stY(T)$
be such that $\tr\alpha=\tr\beta=0$ and set $\shM=\shL\otimes\det\shF^{-1}$.
Let also $\zeta\colon\shM\otimes(\det\shF)^{2}\arr\Sym^{2}\shF$ be
the map associated with $\alpha$ via \ref{lem:Trace zero map alpha description}
and $\delta=\delta_{\beta}\colon\Sym^{3}\shF\arr\det\shF$ (see (\ref{iso:delta beta})).
If $\shL=\odi T$, $y,z$ is a basis of $\shF$ and we use notation
from \ref{not: associated parameters for chi and Sthree}, we have
equivalences
\[
\text{the map (}\ref{eq:first condition on M,F,delta,zeta,omega}\text{) is zero}\iff\beta\circ\zeta=0\iff\left\{ \begin{array}{c}
2aA+bB+cC=0\\
2cA+eC-aB=0
\end{array}\right.
\]
\[
\text{the map (}\ref{eq:second condition on M,F,delta,zeta,omega}\text{) coincides with }\eta_{\delta}\iff\left\{ \begin{array}{c}
a^{2}+bc=-\omega C\\
ac+be=2\omega A\\
c^{2}-ae=B\omega
\end{array}\right.
\]
\end{lem}

\begin{proof}
The conditions $\tr\alpha=\tr\beta=0$ means that $a+d=c+f=A+D=0$.
We have expressions 
\[
\zeta=By^{2}-2Ayz-Cz^{2}\comma\check{\zeta}=-2C(y^{2})^{*}+2A(yz)^{*}+2B(z^{2})^{*}
\]
thanks to (\ref{iso:primo alpha}) and (\ref{iso:terzo alpha}). In
particular
\[
\beta(\zeta)=(aB-2cA-eC)y+(2aA+bB+cC)z
\]
By definition of $\delta_{\beta}$, the composition \ref{eq:first condition on M,F,delta,zeta,omega}
is $\shF\ni x\mapsto x\wedge\beta(\zeta)\in\det\shF$. The first equivalence
follows from expressions
\[
y\wedge\beta(\zeta)=(2aA+bB+cC)y\wedge z\text{ and }z\wedge\beta(\zeta)=-(aB-2cA-eC)y\wedge z
\]
The second one instead follows from the expression of $\eta_{\delta}$
given in \ref{eq:expression of eta delta} and the fact that the map
\ref{eq:second condition on M,F,delta,zeta,omega} is just $\omega\check{\zeta}$.
\end{proof}
\begin{proof}
[Proof of Theorem \ref{thm:description of the main component}] From
\ref{lem:for the description of Z Sthree} we see that $\stZ\to\GCov$
is well defined and is an equivalence onto the locus $\{\tr\alpha=\tr\beta=0\}$,
which coincides with $\stZ_{G}$ thanks to \ref{thm:geometry of GCov and components}.

For the first claim we use Theorems \ref{thm:geometry of GCov and components},
\ref{thm:The-locus when omega is invertible}, \ref{thm:not degenerate locus of S3}
and \ref{thm:the locus where beta is not zero}. They tell us that
$\stU_{\omega}\cup\stU_{\alpha}\cup\stU_{\beta}\subseteq\stZ_{G}$
is an open substack and it is smooth and geometrically integral over
$\shR$. By \ref{lem:zero not smooth} it remains to show that, topologically,
$\stZ_{G}-(\stU_{\omega}\cup\stU_{\alpha}\cup\stU_{\beta})=\{0\}$.
If $k$ is an algebraically closed field and $\chi\in\stZ_{G}(k)$
is represented by local parameters as usual then $\chi\notin(\stU_{\omega}\cup\stU_{\alpha}\cup\stU_{\beta})$
if and only if $\omega=0$, $d_{\beta}=0$ and $\alpha=\lambda\id$.
As $\tr\alpha=0$ it follows that $\alpha=0$. By \ref{lem:local basis for beta nowhere zero}
the condition $d_{\beta}=0$ implies $b=e=0$. By \ref{eq:loc com and ass conditions}
we also have $a^{2}=c^{2}=0$ which tells us that $a=c=0$ and, since
$\tr\beta=0$, $d=f=0$. In conclusion $\chi\notin(\stU_{\omega}\cup\stU_{\alpha}\cup\stU_{\beta})$
if and only if $\chi\in\{0\}$.
\end{proof}
\begin{proof}
[Proof of Theorem \ref{thm:smooth locus of GCov}] The first equivalence
is \ref{thm:The Proj atlases}. Moreover $\GCov$ is smooth outside
$\{0\}$ by \ref{thm:description of the main component}. The result
then follows from \ref{lem:zero not smooth}.
\end{proof}
\begin{rem}
Given $\chi\in\stZ_{G}$ and denoted by $(\shM,\shF,\delta,\zeta,\omega)$
the corresponding object via \ref{thm:description of the main component},
one can check locally that the composition 
\[
\shM^{2}\otimes(\det\shF)^{2}\arrdi{\id_{\shM}\otimes\zeta}\shM\otimes\Sym^{2}\shF\arrdi{\id_{\shM}\otimes\check{\zeta}}\shM\otimes\shM^{-1}\simeq\odi T
\]
is $-4m_{\chi}$, where as usual $m_{\chi}\colon\shL_{\chi}^{2}\simeq\shM^{2}\otimes(\det\shF)^{2}\arr\odi S$.
\end{rem}

\appendix

\section{\label{sec:Bitorsors}Bitorsors.}

In this section we recall some properties of bitorsors. One can compare
definitions and results with \cite[Chapter III, Section 1.5]{Giraud1971}.
Let us fix a site $\catC$ and two sheaves of groups $G,H\colon\catC^{\op}\to\Grp$.

If $F\colon\catC^{\op}\arr\sets$ is a functor and $S\in\catC$ we
denote by $F_{S}$ the composition $(\catC/S)^{\op}\arr\catC^{\op}\arr\sets$
and we (improperly) call it the restriction to $S$. If $F$ is a
sheaf (of groups) then $F_{S}$ is a sheaf (of groups). If $F\colon(\catC/S)^{\op}\arr\sets$
is a functor an action of $G$ on $F$ is an action of $G_{S}$ on
$F$.

In this section
\begin{defn}
\label{def:bitorsors} We denote by $\Sh^{G}(\catC)$ (resp. $\Sh_{\catC}^{G}$)
the category (resp. fibered category) of sheaves over $\catC$ with
a right $G$-action. In particular $\Sh_{\catC}^{G}(S)=\Sh^{G_{S}}(\catC/S)$
for $S\in\catC$ and the pullbacks are the restrictions. 

The left (resp. right) regular action of $G$ on itself is the left
(resp. right) action given by 
\[
G\times G\arr G\comma(g,h)\arr gh
\]
A left (resp. right) $G$-torsor over $\catC$ is a sheaf $P\colon\catC^{op}\arr\set$
with a left action $G\times P\arr P$ (resp. right action $P\times G\arr P$)
such that, for all $T\in\catC$, $P_{T}$ is locally isomorphic to
$G_{T}$ endowed with the left (resp. right) regular action. A left
(resp. right) $G$-torsor over $S\in\catC$ is a left (resp. right)
$G_{S}$-torsor over $\catC/S$.

A $(G,H)$-biaction on a sheaf $P\colon\catC^{op}\arr\set$ is a pair
$(G\times P\arrdi uP,P\times H\arrdi vP)$ where $u$ and $v$ are,
respectively, a left $G$-action and a right $H$-action on $P$,
such that the following diagram is commutative.   \[   \begin{tikzpicture}[xscale=3.0,yscale=-1.2]     \node (A0_0) at (0, 0) {$G\times P \times H$};     \node (A0_1) at (1, 0) {$P\times H$};     \node (A1_0) at (0, 1) {$G\times P$};     \node (A1_1) at (1, 1) {$P$};     \path (A0_0) edge [->]node [auto] {$\scriptstyle{u\times \id_{H}}$} (A0_1);     \path (A0_0) edge [->]node [auto] {$\scriptstyle{\id_{G} \times v}$} (A1_0);     \path (A0_1) edge [->]node [auto] {$\scriptstyle{v}$} (A1_1);     \path (A1_0) edge [->]node [auto] {$\scriptstyle{u}$} (A1_1);   \end{tikzpicture}   \] 
A $(G,H)$-bitorsor over $\catC$ (resp. $S\in\catC$) is a sheaf
$P\colon\catC{}^{op}\arr\set$ (resp. $P\colon(\catC/S)^{op}\arr\set$)
with a $(G,H)$-biaction for which $P$ is both a left $G$-torsor
and a right $H$-torsor. Denote by $\Bi(G,H)$ the fibered category
over $\catC$ of $(G,H)$-bitorsors. Notice that a $(G,H)$ torsor
over $\catC$ corresponds to a section $\catC\arr\Bi(G,H)$.
\end{defn}

\begin{rem}
The fibered category $\Bi(G,H)$ is a stack over $\catC$. This is
easy to prove directly, using the fact that $\Sh_{\catC}$, the fibered
category of sheaves of sets over $\catC$, is a stack (see \cite[Part 1, Example 4.11]{Fantechi2007}).
\end{rem}

\begin{rem}
Given a left $G$-action $u\colon G\times P\arr P$ and a right $H$-action
$v\colon P\times H\arr P$ , the pair $(u,v)$ is a $(G,H)$-biaction
if and only if the homomorphism $G\arr\Autsh P$ induced by $u$ factors
through $\Autsh^{H}P$, that is if $G$ acts through $H$-equivariant
isomorphisms.
\end{rem}

\begin{lem}
\label{lem:criterion bitorsor} Let $P$ be a sheaf with a $(G,H)$-biaction.
Then $P$ is a $(G,H)$-bitorsor if and only if $P$ is a right $H$-torsor
and $G\to\Autsh^{H}(P)$ is an isomorphism (or $P$ is a left $G$-torsor
and $H\to\Autsh^{G}(P)$ is an isomorphism).
\end{lem}

\begin{proof}
We prove only the first claim. In particular we can assume that $P$
is an $H$-torsor. Moreover, as the claim is local, we can also assume
$P=H$. We therefore have a map
\[
\phi\colon G\to\Autsh^{H}(H)\simeq H\text{ such that }gh=\phi(g)h
\]
As $1$ is a global section of $P=H$, $P$ is a $G$-torsor (that
is $P$ is a $(G,H)$-bitorsor) if and only if the orbit map $G\to H$,
$g\mapsto g\cdot1$ is an isomorphism. Since this last map is $\phi$
we get the claim.
\end{proof}
We want to define a functor
\[
\Lambda\colon\Bi(G,H)\to\Homsh(\Sh_{\catC}^{G},\Sh_{\catC}^{H})
\]
As usual, we define this only over the global sections. If $P$ is
a $(G,H)$-torsor over $\catC$ and $X\in\Sh^{G}(\catC)$, then $X\times P$
has a right free action of $G$ given by $(x,p)g=(xg,g^{-1}p)$. Its
quotient $\Lambda_{P}(X)=(X\times P)/G$ has a right $H$-action induced
by the one of $P$, so that $\Lambda_{P}(X)\in\Sh^{H}(\catC)$. Notice
that, since $G$ acts freely on $X\times P$, the quotient $(X\times P)/G$
can be defined avoiding to sheafify the naive quotient (and thus avoiding
the corresponding set theoretic problems): $(X\times P)/G\colon\shC\to\sets$
maps an object $S\in\catC$ to the set of $G$-torsors $Q\to S$ together
with a $G$-equivariant map $Q\to X\times P$, in other words the
quotient stack $[X\times P/G]$ is actually equivalent to a sheaf.

\begin{lem}
\label{rem:trivial bitorsor} Let $P$ be a $(G,H)$-torsor over $S\in\catC$
with a section $p_{0}\in P(S)$. Then:
\begin{enumerate}
\item the orbit maps $G_{S}\to P,g\mapsto gp_{0}$ and $H_{S}\to P,h\mapsto p_{0}h$
are isomorphisms and the induced map $\phi\colon G_{S}\to H_{S}$
is an isomorphism of groups such that $gp_{0}=p_{0}\phi(g)$: in other
words $P$ is isomorphic to the $(G,H)$-torsor $H$ with left $G$-action
$gh=\phi(g)h$;
\item the composition $\Sh_{\catC/S}^{G_{S}}\arrdi{\Lambda_{P}}\Sh_{\catC/S}^{H_{S}}\to\Sh_{\catC/S}^{G_{S}}$,
where the second map is induced by $\phi$, is isomorphic to the identity.
\end{enumerate}
\end{lem}

\begin{proof}
Since $P$ is a left $G$-torsor and right $H$-torsor the orbit maps
are isomorphisms. Thus the map $\phi\colon G_{S}\to H_{S}$ is well
defined: $\phi(g)$ is the unique element such that $gp_{0}=p_{0}\phi(g)$.
This also allows to prove that $\phi$ is an isomorphism of groups.
For the second part, it is enough to notice that the maps
\[
X\to\Lambda_{P}(X),x\mapsto(x,p_{0})\comma\Lambda_{P}(X)\to X,(x,gp_{0})\to xg
\]
are well defined, inverses of each other and $G$-equivariant.
\end{proof}
\begin{rem}
Any $(G,H)$-biaction on a sheaf $P$ induces an $(H,G)$-biaction
on $P$ by the rule: $g\star p\star h=h^{-1}pg^{-1}$. Moreover a
$(G,H)$-torsor $P$ is naturally also an $(H,G)$-torsor, which we
denote by $P^{\star}$, and this operation defines an isomorphism
of stacks $\Bi(G,H)\to\Bi(H,G).$
\end{rem}

\begin{lem}
\label{lem:LambdaP equivalence} If $P$ is a $(G,H)$-torsor and
$X\in\Sh^{G}(\catC)$ there is a canonical $G$-equivariant isomorphism
of sheaves $\Lambda_{P^{\star}}(\Lambda_{P}(X))\to X$. In particular
the functor $\Lambda_{P}\colon\Sh_{\catC}^{G}\to\Sh_{\catC}^{H}$
is an equivalence of stacks.
\end{lem}

\begin{proof}
There is a map $\psi\colon P\times P^{\star}\to G$ determined by
the rule $p=\psi(p,q)q$ for $p,q\in P$. A simple computation shows
that $\psi(up,vq)=u\psi(p,q)v^{-1}$ for $u,v\in G$ and $\psi(ph,qh)=\psi(p,q)$
for $h\in H$. Using those expressions it is easy to show that the
map
\[
X\times P\times P^{\star}\to X\comma(x,p,q)\longmapsto x\psi(p,q)
\]
factors through a $G$-equivariant map $\Lambda_{P^{\star}}(\Lambda_{P}(X))\to X$.
Going locally where $P$ has a section and using \ref{rem:trivial bitorsor}
it is easy to show that this map is an isomorphism.
\end{proof}
\begin{prop}
\label{prop:bitorsors and isomorphisms SHG - SHH}The functor $\Lambda\colon\Bi(G,H)\to\Homsh(\Sh_{\catC}^{G},\Sh_{\catC}^{H})$
is an equivalence onto the full substack of functors which are fully
faithful and maps $G$ to an $H$-torsor. Moreover $\Lambda_{P}$
is an equivalence for $P\in\Bi(G,H)$ and restricts to an equivalence
$\Lambda_{P}\colon\Bi G\to\Bi H$. The induced functor $\Lambda\colon\Bi(G,H)\to\Homsh(\Bi G,\Bi H)$
is an equivalence onto the full substack of functors which are fully
faithful.
\end{prop}

\begin{proof}
By \ref{rem:trivial bitorsor}, we see that if $P\in\Bi(G,H)$ and
$X\in\Bi G$ then $\Lambda_{P}(X)\in\Bi H$. Thus the functor $\Lambda\colon\Bi(G,H)\to\Homsh(\Bi G,\Bi H)$
is well defined and $\Lambda_{P}(G)\in\Bi H$. We prove both statement
at the same time. For this set either $\underline{H}=\Homsh(\Sh_{\catC}^{G},\Sh_{\catC}^{H})$
or $\underline{H}=\Homsh(\Bi G,\Bi H)$ and denote by $E$ the full
substack considered in the statement.

By \ref{lem:LambdaP equivalence} the functor $\Lambda\colon\Bi(G,H)\to\underline{H}$
has values in $E$. On the other hand if $\Omega\in E$ then $\Omega(G)$
is an $H$-torsor and the map
\[
G\arrdi{\simeq}\Autsh^{G}(G)\to\Autsh^{H}(\Omega(G))
\]
is an isomorphism because $\Omega$ is fully faithful. Thus $\Omega(G)$
is a $(G,H)$-bitorsor thanks to \ref{lem:criterion bitorsor}. Evaluation
in $G$ therefore defines a functor $Y\colon E\to\Bi(G,H)$. We first
show that $Y\circ\Lambda\simeq\id_{\Bi(G,H)}$. The following maps
\[
P\to\Lambda_{P}(G),p\mapsto(1,p)\comma\Lambda_{P}(G)\to P,(g,p)\mapsto gp\text{ for }P\in\Bi(G,H)
\]
 are quasi-inverses of each other and $H$-equivariant. Moreover we
can check that they also preserve the $G$-action.

We now prove that $\Lambda\circ Y\simeq\id_{E}$. Let $\Omega\in E$.
Using the isomorphism $X\simeq\Homsh^{G}(G,X)$ we obtain natural
morphisms
\[
X\times\Omega(G)\simeq\Homsh^{G}(G,X)\times\Omega(G)\to\Homsh^{H}(\Omega(G),\Omega(X))\times\Omega(G)\to\Omega(X)\comma(x,e)\longmapsto\Omega(u_{x})(e)
\]
Here, for $(x,e)\in X(S)\times\Omega(G)(S)$ the map $u_{x}\colon G_{S}\to X_{S}$
is the orbit map $u_{x}(g)=xg$. By going through the definitions
one can check that the above map is $H$-equivariant and $G$-invariant.
Thus it induces a natural morphism morphism $\Lambda_{\Omega(G)}(X)\to\Omega(X)$
that we claim is an isomorphism.

Since this is a local claim and $\Omega(G)$ is an $H$-torsor, we
can assume that $\Omega(G)$ has a section. Taking into account \ref{rem:trivial bitorsor},
we can assume $G=H$ and $\Omega(G)=H$ with left and right actions
given by multiplication. We now prove that $\Omega\simeq\id$, so
that, in particular, the morphism $\Lambda_{\Omega(G)}\to\Omega$
will be an isomorphism. We have a natural isomorphism of sheaves of
sets
\[
\delta\colon X\simeq\Homsh^{G}(G,X)\simeq\Homsh^{H}(\Omega(G),\Omega(X))=\Homsh^{H}(H,\Omega(X))\simeq\Omega(X)\comma\delta(x)=\Omega(u_{x})(1)
\]
We just have to show that it is $G$-equivariant. On the other hand
$u_{xg}=u_{x}\circ m_{g}$, where $m_{g}\colon G\to G$ is the left
multiplication by $g$ and, by construction, $\Omega(m_{g})(1)=g\cdot1=g$,
because the left action of $G$ on $\Omega(G)=H$ is just the left
multiplication for $G=H$. In particular
\[
\delta(xg)=\Omega(u_{xg})(1)=\Omega(u_{x})(\Omega(m_{g})(1))=\Omega(u_{x})(g)=\Omega(u_{x})(1\cdot g)=\Omega(u_{x})(1)g=\delta(x)g
\]
as required.
\end{proof}
\begin{cor}
We have $\Bi G\simeq\Bi H$ if and only if there exists an $H$-torsor
$P$ over $\catC$ with an isomorphism $G\simeq\Autsh^{H}P$.
\end{cor}

We now want to describe two examples of non trivial bitorsors.
\begin{example}
Set $P=\Isosh(G,H)$. The maps
\[
\Autsh G\times P\arr P\comma P\times\Autsh(H)\arr P\text{, both given by }(\phi,\psi)\longmapsto\phi\circ\psi
\]
induce a $(\Autsh G,\Autsh H)$-action on $\Isosh(G,H)$ and, if $G$
and $H$ are locally isomorphic, then $\Isosh(G,H)$ is a $(\Autsh G,\Autsh H)$-bitorsor.
In particular, in this case, we obtain an isomorphism
\[
\Bi\Autsh(G)\simeq\Bi\Autsh(H)
\]
\end{example}

The second bitorsor we want to describe is a refinement of the previous
one.
\begin{prop}
\label{prop:bitorsors and semidirect products} Set $P=G\times\Isosh(H,G)$.
The maps
\[
P\times(H\rtimes\Autsh H)\arr P\comma(G\rtimes\Autsh G)\times P\arr P\text{, both given by }(x,\phi)\cdot(y,\psi)=(x\phi(y),\phi\psi)
\]
define a $((G\rtimes\Autsh G),(H\rtimes\Autsh H))$-action on $P$
and, if $G$ and $H$ are locally isomorphic, then $P$ is a $((G\rtimes\Autsh G),(H\rtimes\Autsh H))$-bitorsor.
In particular, in this case, we have an isomorphism
\[
\Bi(G\rtimes\Autsh G)\simeq\Bi(H\rtimes\Autsh H)
\]
and, if $\Lambda_{P}\colon\Sh_{\catC}^{G\rtimes\Autsh G}\arr\Sh_{\catC}^{H\rtimes\Autsh H}$
is the functor defined in \ref{prop:bitorsors and isomorphisms SHG - SHH},
we have a canonical isomorphism of sheaves of sets
\[
(X/\Autsh G)\simeq(\Lambda_{P}(X)/\Autsh H)\text{ for all }X\in\Sh_{\catC}^{G\rtimes\Autsh G}
\]
\end{prop}

\begin{proof}
A direct computation shows that the maps in the statement yield compatible
actions. Moreover, if $\gamma\colon G\arr H$ is an isomorphism, it
is also straightforward to check that the maps $g\longmapsto g\cdot\gamma$
and $h\longmapsto\gamma\cdot h$ are equivariant isomorphisms $G\rtimes\Autsh G\arr P$
and $H\rtimes\Autsh H\arr P$ respectively. Finally consider the map
\[
\pi\colon X\times P=X\times G\times\Isosh(H,G)\arr X\text{ given by }\pi(x,g,\phi)=x(g,\id_{G})
\]
It is easy to check that $\pi(z(u,\psi))=\pi(z)(1_{G},\psi)$ and
$\pi(z(1_{H},\delta))=\pi(z)$ for all $z\in X\times P$, $(u,\psi)\in G\rtimes\Autsh G$
and $(1_{H},\delta)\in H\rtimes\Autsh H$. In particular $\pi$ yields
a map $(\Lambda_{P}(X)/\Autsh H)\arr(X/\Autsh G)$. This is an isomorphism
since it is so locally, i.e. when we have an isomorphism $H\arrdi{\phi}G$:
in this case the inverse is given by $x\arr(x,1_{G},\phi)$.
\end{proof}

\section{Sheaves and traces}

In this appendix we want to collect several constructions involving
quasi-coherent sheaves, especially locally free ones, that are used
throughout the paper. In particular we will introduce and discuss
the trace map associated with particular morphisms of sheaves. We
fix a base scheme $T$. All sheaves will be defined over this scheme.
\begin{rem}
\label{rem:F is Fdual tensor det F} If $\shF$ is a locally free
sheaf of rank $2$ then the canonical map $\shF\otimes\shF\arr\det\shF$
induces an isomorphism
\[
\shF\simeq\duale{\shF}\otimes\det\shF
\]
If $y,z$ is a basis of $\shF$ then the above map is given by $y\arr-z^{*}\otimes(y\wedge z)$,
$z\arr y^{*}\otimes(y\wedge z)$.
\end{rem}

\begin{defn}
\label{not: trace of a map} Let $\shF$ be a locally free sheaf,
$\shQ$ be an $\odi T$-module and $\zeta\colon\shQ\otimes\shF\to\shF$
be a morphism. We define $\tr\zeta\colon\shQ\to\odi T$, called the
\emph{trace }of $\zeta$, as
\[
\tr\zeta\colon\shQ\to\Endsh(\shF)\simeq\shF^{\vee}\otimes\shF\arrdi{\textup{ev}}\odi T
\]
where the last map is the evaluation. We denote by $\Homsh_{\tr=0}(\shQ\otimes\shF,\shF)$
the subsheaf of morphisms $\Homsh(\shQ\otimes\shF,\shF)$ of maps
whose trace is $0$. In particular
\[
\Homsh_{\tr=0}(\shQ\otimes\shF,\shF)\simeq\Homsh(\shQ,\Homsh_{\tr=0}(\shF,\shF))
\]

If $\alA$ is a locally free sheaf of algebras on $T$ then its trace
map is $\tr_{\alA}=\tr(\alA\otimes\alA\to\alA)\colon\alA\to\odi T$.
\end{defn}

\begin{rem}
If a basis of $\shF$ is fixed and $\rk\shF=n$ then $\Endsh(\shF)\simeq\shM_{n,n}$,
the quasi-coherent sheaf of $n\times n$-matrices and the composition
\[
\shM_{n,n}\simeq\Endsh(\shF)\simeq\shF^{\vee}\otimes\shF\arrdi{\text{ev}}\odi T
\]
is the usual trace of matrices. In particular if $\zeta\colon\shQ\otimes\shF\to\shF$
is a map then $(\tr\zeta)(q)=\tr(\zeta(q\otimes-)).$
\end{rem}

\begin{defn}
\label{def:locally multiple identity} Let $\shF$ be a locally free
sheaf, $\shQ$ be an $\odi T$-module and $\zeta\colon\shQ\otimes\shF\to\shF$
be a morphism. We say that $\zeta$ is locally a multiple of the identity
if all the morphisms in the image of $\shQ\to\Endsh(\shF)$ are fpqc
locally a multiple of the identity. We denote by $\Homsh_{\lambda\id}(\shQ\otimes\shF,\shF)$
the subsheaf of $\Homsh(\shQ\otimes\shF,\shF)$ of those morphisms.
\end{defn}

\begin{lem}
\label{lem:multiple identity} Let $\shF$ be a locally free sheaf
and $\shQ$ be an $\odi T$-module. Then $\odi T\simeq\Endsh_{\lambda\id}(\shF)$,
that is a morphism $\shF\to\shF$ is locally a multiple of the identity
if it is a multiple of the identity. More generally, applying $\Homsh(\shQ,-)$,
we obtain 
\[
\Homsh(\shQ,\odi T)\simeq\Homsh(\shQ,\Endsh_{\lambda\id}(\shF))=\Homsh_{\lambda\id}(\shQ\otimes\shF,\shF)\comma\mu\longmapsto\mu\otimes\id
\]
\end{lem}

\begin{proof}
The map $\odi T\to\Endsh_{\lambda\id}(\shF)$ is injective. Moreover
it is an epimorphism in the fpqc topology, hence it is an isomorphism.
The last claim follows from definition.
\end{proof}
\begin{lem}
\label{lem:decomposition trace identity} Let $\shF$ be a locally
free sheaf and $\shQ$ be an $\odi T$-module. Then there is an exact
sequence of $\odi T$-modules
\[
0\to\Homsh_{\tr=0}(\shQ\otimes\shF,\shF)\to\Homsh(\shQ\otimes\shF,\shF)\arrdi{\tr}\Homsh(\shQ,\odi T)\to0
\]
If $\rk\shF\in\odi T^{*}$ then $\tr$ restricts to an isomorphism
$\Homsh_{\lambda\id}(\shQ\otimes\shF,\shF)\to\Homsh(\shQ,\odi T)$,
whose inverse is 
\[
\Homsh(\shQ,\odi T)\to\Homsh_{\lambda\id}(\shQ\otimes\shF,\shF)\comma\mu\mapsto(\mu/\rk\shF)\otimes\id_{\shF}
\]
In particular in this case $\Homsh(\shQ\otimes\shF,\shF)=\Homsh_{\tr=0}(\shQ\otimes\shF,\shF)\oplus\Homsh_{\lambda\id}(\shQ\otimes\shF,\shF)$.
\end{lem}

\begin{proof}
The surjectivity of $\tr$ can be checked locally when $\shF$ is
free. The second claim instead follows easily from \ref{lem:multiple identity}.
\end{proof}
\begin{rem}
\label{rem:Lambda and duals} Let $R$ be a ring, $M,N$ be $R$-modules
and $m\in N$. Given $\eta\colon M\otimes N\to R$ we define
\[
\Lambda^{m}M\otimes\Lambda^{m}N\to R\comma(x_{1}\wedge\cdots\wedge x_{m})\otimes(y_{1}\wedge\cdots\wedge y_{m})\longmapsto\det((\eta(x_{i}\otimes y_{j})))
\]
Applying this construction to the evaluation $M^{\vee}\otimes M\to R$
we obtain a map
\[
\Lambda^{m}(M^{\vee})\to(\Lambda^{m}M)^{\vee}\comma(\phi_{1}\wedge\cdots\wedge\phi_{m})\longmapsto(y_{1}\wedge\cdots\wedge y_{n}\mapsto\det(\phi_{i}(y_{j})))
\]
A direct check shows that if $M$ is free with basis $e_{1},\dots,e_{n}$
then $(e_{i_{1}}^{*}\wedge\cdots\wedge e_{i_{m}}^{*})$ is mapped
to $(e_{i_{1}}\wedge\cdots\wedge e_{i_{m}})^{*}$ for $1\leq i_{1}<\cdots<i_{m}\leq n$.
In particular the above map is an isomorphism for locally free $R$-modules.

The map $\Lambda^{m}M\otimes\Lambda^{m}N\to R$ can also be obtained
as
\[
M\otimes N\to R\then M\to N^{\vee}\then\Lambda^{m}M\to\Lambda^{m}(N^{\vee})\to(\Lambda^{m}N)^{\vee}\then\Lambda^{m}M\otimes\Lambda^{m}N\to R
\]
\end{rem}

\begin{rem}
\label{rem:discriminant} If $\alA$ is a locally free sheaf of algebras
its discriminant $\Delta_{\alA}\colon(\det\alA)^{2}\to\odi T$ is
the map induced by $\tr_{\alA}(-\cdot-)\colon\alA\otimes\alA\to\odi T$
as in \ref{rem:Lambda and duals}. It defines an effective Cartier
divisor on $T$ whose complement coincides with the étale locus of
$\Spec\alA\to T$ (see \cite[Proposition 4.10]{SGA1}). Assume that
$\alA=\odi T\oplus\E$ with $(\tr_{\alA})_{|\E}=0$ and denote by
$\pi\colon\alA\to\odi T$ the projection. The map $\E\otimes\E\to\alA\arrdi{\pi}\odi T$
also defines a map $\Delta\colon((\det\alA)^{2}\simeq)(\det\E)^{2}\to\odi T$
via \ref{rem:Lambda and duals} and $\Delta_{\alA}=n^{n}\Delta$,
because $\tr_{\alA}=\tr_{\alA}\circ\pi=\tr_{\alA}(1)\pi=n\pi$.
\end{rem}

\subsection{\label{subsec:appendix delta Miranda} Trace zero maps of the form
\texorpdfstring{$\Sym^{2}\shF\to\shF$}{Sym^2 F --> F}}

In this subsection $\shF$ will denote a locally free sheaf of rank
$2$.
\begin{notation}
\label{not: trace of beta}Given $\beta\colon\Sym^{2}\shF\arr\shF$
we set
\[
\tr\beta=\tr(\shF\otimes\shF\arr\Sym^{2}\shF\arrdi{\beta}\shF)\colon\shF\arr\odi T
\]
If $y,z$ is a basis of $\shF$ and $\beta(y^{2})=ay+bz$, $\beta(yz)=cy+dz$,
$\beta(z^{2})=ey+fz$, then $(\tr\beta)(y)=a+d$, $(\tr\beta)(z)=c+f$.
\end{notation}

\begin{prop}
\label{prop:Miranda correspondence} \cite{Miranda1985,Bolognesi2009}
If $\beta\colon\Sym^{2}\shF\arr\shF$ is a map such that $\tr\beta=0$
then there exists a unique dashed map $\delta$ as in  \[   \begin{tikzpicture}[xscale=2.4,yscale=-1.0]     \node (A0_0) at (0, 0) {$\Sym^2\shF\otimes \shF$};     \node (A0_1) at (1, 0) {$\shF\otimes\shF$};     \node (A1_0) at (0, 1) {$\Sym^3 \shF$};     \node (A1_1) at (1, 1) {$\det \shF$};     \path (A0_0) edge [->]node [auto] {$\scriptstyle{\beta\otimes \id}$} (A0_1);     \path (A0_0) edge [->]node [auto] {$\scriptstyle{}$} (A1_0);     \path (A0_1) edge [->]node [auto] {$\scriptstyle{}$} (A1_1);     \path (A1_0) edge [->,dashed]node [auto] {$\scriptstyle{\delta}$} (A1_1);   \end{tikzpicture}   \] 
This association yields an isomorphism \begin{align}\label{iso:delta beta}
 \begin{tikzpicture}[xscale=5.0,yscale=-0.8]     \node (A0_0) at (0, 0) {$\Homsh_{\tr = 0}(\Sym^2 \shF,\shF)$};     \node (A0_1) at (1, 0) {$\Homsh(\Sym^3 \shF,\det\shF)$};     \node (A1_0) at (0, 1) {$\left(\begin{array}{ccc} a & c & e\\ b & -a & -c \end{array}\right)$};     \node (A1_1) at (1, 1) {$\left(\begin{array}{cccc} -b & a & c & e\end{array}\right)$};     \path (A0_0) edge [->]node [auto] {$\scriptstyle{}$} (A0_1);     \path (A1_0) edge [|->,gray]node [auto] {$\scriptstyle{}$} (A1_1);   \end{tikzpicture}
\end{align}where the last row describes how this map behaves if a basis $y,z$
of $\shF$ is chosen. Here the chosen basis for $\Sym^{2}\shF$ and
$\Sym^{3}\shF$ are $y^{2},yz,z^{2}$ and $y^{3},y^{2}z,yz^{2},z^{3}$
respectively.
\end{prop}

\begin{notation}
\label{nota:Miranda} In the hypothesis of \ref{prop:Miranda correspondence}
we will denote the correspondence (\ref{iso:delta beta}) by $\beta\longmapsto\delta_{\beta}$
and $\delta\longmapsto\beta_{\delta}$. Given $\delta\colon\Sym^{3}\shF\to\det\shF$
we also define maps
\[
\eta_{\delta}\colon\Sym^{2}\shF\to\odi T\comma\alpha_{\delta}\colon\det\shF\otimes\shF\to\shF\text{ and }m_{\delta}\colon(\det\shF)^{2}\to\odi T
\]
 as follows. Define $\eta_{\delta}$ as the map
\begin{equation}
\Sym^{2}\shF\arrdi u\Lambda^{2}\Sym^{2}\shF\otimes\duale{\Lambda^{2}\shF}\arrdi v\odi S\label{eq:etabeta}
\end{equation}
where $v$ is induced by $\Lambda^{2}\beta_{\delta}\colon\Lambda^{2}\Sym^{2}\shF\arr\Lambda^{2}\shF$
and $u$ is induced by   \[   \begin{tikzpicture}[xscale=4.7,yscale=-0.6]     \node (A0_0) at (0, 0) {$\Lambda^2\shF\otimes \Sym^2 \shF$};     \node (A0_1) at (1, 0) {$\Lambda^2 \Sym^2 \shF$};     \node (A1_0) at (0, 1) {$(x_1\wedge x_2)\otimes x_3 x_4$};     \node (A1_1) at (1, 1) {$-x_1x_3\wedge  x_2 x_4-x_1 x_4\wedge x_2 x_3$};     \path (A0_0) edge [->]node [auto] {$\scriptstyle{}$} (A0_1);     \path (A1_0) edge [|->,gray]node [auto] {$\scriptstyle{}$} (A1_1);   \end{tikzpicture}   \] If
$2\in\odi T^{*}$, the map $\alpha_{\delta}$ 
\[
\det\shF\otimes\duale{\shF}\otimes\shF\arrdi{\simeq}\shF\otimes\shF\arrdi{\eta_{\delta}/2}\odi S
\]
where $\det\shF\otimes\duale{\shF}\simeq\shF$ is the canonical isomorphism
of \ref{rem:F is Fdual tensor det F} and $m_{\delta}$ is
\[
(\det\shF)^{2}\otimes\det\shF\simeq\det(\det\shF\otimes\shF)\arrdi{-\det\alpha_{\delta}}\det\shF
\]
\end{notation}

\begin{rem}
\label{rem:local Miranda} In the hypothesis of \ref{prop:Miranda correspondence},
if $y,z$ is a basis of $\shF$, we identify $\det\shF\simeq\odi T$
using the generator $y\wedge z\in\det\shF$ and we write $\delta$
as
\begin{equation}
\delta(y^{3})=-b\comma\delta(y^{2}z)=a\comma\delta(yz^{2})=c\comma\delta(z^{3})=e\label{eq:expression of delta}
\end{equation}
then we have expressions
\begin{equation}
\eta_{\delta}(y^{2})=2(a^{2}+bc)\comma\eta_{\delta}(yz)=ac+be\comma\eta_{\delta}(z^{2})=2(c^{2}-ae)\label{eq:expression of eta delta}
\end{equation}
\[
2\alpha_{\delta}(y)=\eta_{\delta}(yz)y-\eta_{\delta}(y^{2})z\comma2\alpha_{\delta}(z)=\eta_{\delta}(z^{2})y-\eta_{\delta}(yz)z
\]
\end{rem}

\subsection{\label{subsec:Trace zero alpha} Trace zero maps of the form \texorpdfstring{$\shQ\otimes\shF\to\shF$}{Q tensor F --> F}}

In this subsection $\shF$ will denote a locally free sheaf of rank
$2$ and $\shQ$ an $\odi T$-module. Moreover we assume $2\in\odi T^{*}$.
\begin{rem}
Given a map $\alpha\colon\shQ\otimes\shF\arr\shF$, we have a factorization
  \[   \begin{tikzpicture}[xscale=1.8,yscale=-0.5]     \node (A0_0) at (0, 0) {$s\otimes p \otimes q$};     \node (A0_2) at (2, 0) {$\alpha(s\otimes q)p-\alpha(s\otimes p)q$};     \node (A1_0) at (0, 1) {$\shQ\otimes\shF\otimes \shF$};     \node (A1_2) at (2, 1) {$\Sym^2 \shF$};     \node (A3_1) at (1, 3) {$\shQ\otimes \det \shF$};     \path (A1_0) edge [->]node [auto] {$\scriptstyle{}$} (A1_2);     \path (A3_1) edge [->,dashed]node [auto] {$\scriptstyle{}$} (A1_2);     \path (A1_0) edge [->]node [auto] {$\scriptstyle{}$} (A3_1);     \path (A0_0) edge [|->,gray]node [auto] {$\scriptstyle{}$} (A0_2);   \end{tikzpicture}   \] 
and this association defines a map   \begin{align}\label{iso:primo alpha}
  \begin{tikzpicture}[xscale=4.6,yscale=-0.7]     \node (A0_0) at (0, 0) {$\Homsh(\shQ\otimes\shF,\shF)$};     \node (A0_1) at (1, 0) {$\Homsh(\shQ\otimes \det \shF,\Sym^2\shF)$};     \node (A1_0) at (0, 1) {$\left(\begin{array}{cc} u & v\\ w & z \end{array}\right)$};     \node (A1_1) at (1, 1) {$(v\ \ \ \ \ z-u\ \   -w)$};     \path (A0_0) edge [->]node [auto] {$\scriptstyle{}$} (A0_1);     \path (A1_0) edge [|->,gray]node [auto] {$\scriptstyle{}$} (A1_1);   \end{tikzpicture}
\end{align}where the last row describes the behaviour if $\shQ=\odi T$ and a
basis $y,z$ of $\shF$ is given.
\end{rem}

If $\shQ=\odi T$ we obtain a map $\Endsh(\shF)\to\Homsh(\det\shF,\Sym^{2}\shF)$
and the general map (\ref{iso:primo alpha}) is obtained applying
$\Homsh(\shQ,-)$ to the previous map.
\begin{lem}
\label{lem:Trace zero map alpha description} The map (\ref{iso:primo alpha})
restricts to an isomorphism
\[
\Homsh_{\tr=0}(\shQ\otimes\shF,\shF)\to\Homsh(\shQ\otimes\det\shF,\Sym^{2}\shF)
\]
and its kernel is $\Homsh_{\lambda\id}(\shQ\otimes\shF,\shF)\simeq\Homsh(\shQ,\odi T).$
In particular
\[
\odi T\oplus\Homsh(\det\shF,\Sym^{2}\shF)\simeq\Endsh(\shF)
\]
\end{lem}

\begin{proof}
All the claims follows from the case $\shQ=\odi T$, the local description
of (\ref{iso:primo alpha}) and \ref{lem:decomposition trace identity}.
\end{proof}
\begin{rem}
If $\shL$ is an invertible sheaf the map   \[   \begin{tikzpicture}[xscale=4.1,yscale=-0.6]     \node (A0_0) at (0, 0) {$\Sym^2 (\Homsh(\shF,\shL))$};     \node (A0_1) at (1, 0) {$\Homsh(\Sym^2 \shF,\shL^2)$};     \node (A1_0) at (0, 1) {$\xi\eta$};     \node (A1_1) at (1, 1) {$(uv\longmapsto\xi(u)\eta(v)+\eta(u)\xi(v))$};     \path (A0_0) edge [->]node [auto] {$\scriptstyle{}$} (A0_1);     \path (A1_0) edge [|->,gray]node [auto] {$\scriptstyle{}$} (A1_1);   \end{tikzpicture}   \] is
an isomorphism: if $\shL=\odi T$ and a basis $y,z$ of $\shF$ is
given, the above association maps $(y^{*})^{2},y^{*}z^{*},(z^{*})^{2}$
to $2(y^{2})^{*},(yz)^{*},2(z^{2})^{*}$.

Using (\ref{rem:F is Fdual tensor det F}), the composition
\[
\Sym^{2}\shF\simeq\Sym^{2}(\Homsh(\shF,\det\shF))\simeq\Homsh(\Sym^{2}\shF,(\det\shF)^{2})
\]
yields an isomorphism \begin{align}\label{iso:terzo alpha}
\begin{tikzpicture}[xscale=5,yscale=-0.6]     \node (A0_0) at (0, 0) {$\Homsh((\det \shF)^2,\Sym^2\shF)$};     \node (A0_1) at (1, 0) {$\Homsh(\Sym^2\shF,\odi T)$};     \node (A1_0) at (0, 1) {$(u\ \ \ \ \ v\ \  \ \ \  w)$};     \node (A1_1) at (1, 1) {$(2w \ \ -v\ \ \ \ \   2u)$};     \path (A0_0) edge [->]node [auto] {$\scriptstyle{}$} (A0_1);     \path (A1_0) edge [|->,gray]node [auto] {$\scriptstyle{}$} (A1_1);   \end{tikzpicture}  
\end{align}where the last row describes its behaviour if a basis $y,z$ of $\shF$
is given.
\end{rem}

\begin{notation}
\label{not: associated check map}If $\shN$ is an invertible sheaf,
applying $\Homsh(\shN,-)$ to (\ref{iso:terzo alpha}) we obtain an
isomorphism
\[
\Homsh(\shN\otimes(\det\shF)^{2},\Sym^{2}\shF)\to\Homsh(\Sym^{2}\shF,\shN^{-1})
\]
This association will be denoted by $\zeta\longmapsto\check{\zeta}\text{ and }\hat{\eta}\longmapsfrom\eta$.
Notice that the above map has the same local description of (\ref{iso:terzo alpha}).
\end{notation}

\bibliographystyle{amsalpha}
\bibliography{biblio}

\end{document}